\documentclass[aop]{imsart}

\RequirePackage{amsthm,amsmath,amsfonts,amssymb}
\RequirePackage[numbers,sort&compress]{natbib}
\RequirePackage[colorlinks,citecolor=blue,urlcolor=blue]{hyperref}
\RequirePackage{graphicx}

\usepackage{anyfontsize}

\usepackage[utf8]{inputenc}
\usepackage{bm}
\usepackage{latexsym}
\usepackage{mathtools}
\usepackage[dvipsnames]{xcolor}
\usepackage{tikz}
\usetikzlibrary{shapes, arrows, positioning, calc, math}
\usepackage{ifthen}
\usepackage{enumitem}
\usepackage{multirow}

\usepackage{mathrsfs}
\usepackage{amsthm}

\usepackage{forest}
\forestset{
  default preamble={
   for tree={circle, draw, 
            minimum size=2.1em, 
            inner sep=1pt} 
  }
}

\numberwithin{equation}{section}

\DeclareMathOperator{\sgn}{sgn}

\makeatletter
\let\save@mathaccent\mathaccent
\newcommand*\if@single[3]{%
  \setbox0\hbox{${\mathaccent"0362{#1}}^H$}%
  \setbox2\hbox{${\mathaccent"0362{\kern0pt#1}}^H$}%
  \ifdim\ht0=\ht2 #3\else #2\fi
  }
\newcommand*\rel@kern[1]{\kern#1\dimexpr\macc@kerna}
\newcommand*\widebar[1]{\@ifnextchar^{{\wide@bar{#1}{0}}}{\wide@bar{#1}{1}}}
\newcommand*\wide@bar[2]{\if@single{#1}{\wide@bar@{#1}{#2}{1}}{\wide@bar@{#1}{#2}{2}}}
\newcommand*\wide@bar@[3]{%
  \begingroup
  \def\mathaccent##1##2{%
    \let\mathaccent\save@mathaccent
    \if#32 \let\macc@nucleus\first@char \fi
    \setbox\z@\hbox{$\macc@style{\macc@nucleus}_{}$}%
    \setbox\tw@\hbox{$\macc@style{\macc@nucleus}{}_{}$}%
    \dimen@\wd\tw@
    \advance\dimen@-\wd\z@
    \divide\dimen@ 3
    \@tempdima\wd\tw@
    \advance\@tempdima-\scriptspace
    \divide\@tempdima 10
    \advance\dimen@-\@tempdima
    \ifdim\dimen@>\z@ \dimen@0pt\fi
    \rel@kern{0.6}\kern-\dimen@
    \if#31
      \overline{\rel@kern{-0.6}\kern\dimen@\macc@nucleus\rel@kern{0.4}\kern\dimen@}%
      \advance\dimen@0.4\dimexpr\macc@kerna
      \let\final@kern#2%
      \ifdim\dimen@<\z@ \let\final@kern1\fi
      \if\final@kern1 \kern-\dimen@\fi
    \else
      \overline{\rel@kern{-0.6}\kern\dimen@#1}%
    \fi
  }%
  \macc@depth\@ne
  \let\math@bgroup\@empty \let\math@egroup\macc@set@skewchar
  \mathsurround\z@ \frozen@everymath{\mathgroup\macc@group\relax}%
  \macc@set@skewchar\relax
  \let\mathaccentV\macc@nested@a
  \if#31
    \macc@nested@a\relax111{#1}%
  \else
    \def\gobble@till@marker##1\endmarker{}%
    \futurelet\first@char\gobble@till@marker#1\endmarker
    \ifcat\noexpand\first@char A\else
      \def\first@char{}%
    \fi
    \macc@nested@a\relax111{\first@char}%
  \fi
  \endgroup
}

\newcommand*\other[1]{\widebar{#1}}
\newcommand{\torus}[1]{\mathbb{T}_{#1}}

\newcommand{\indic}[1]{\mathrm{1}_{\{#1\}}}
\newcommand{\indicpower}[1]{\mathrm{1}{\{#1\}}}

\newcommand{\URD}{\mathrm{URD}}
\newcommand{\IP}[2]{\left\langle #1, #2 \right\rangle}
\newcommand{\sigmabar}{\other{\sigma}}
\newcommand{\URDbar}{\overline{\URD}}
\newcommand{\relsigma}{\sim_{\sigma}}

\newcommand{\os}{\other{\sigma}}

\newcommand{\fixedbox}[1][1.2]{
\hspace{0.5pt}
\includegraphics[scale = #1]{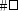}
}

\newcommand{\rightbox}[1][1.2]{
\hspace{0.5pt}
\includegraphics[scale=#1]{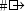}
}

\newcommand{\upbox}[1][1.2]{
\hspace{0.5pt}
\includegraphics[scale = #1]{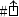}
}

\newcommand{\downbox}[1][1.2]{
\hspace{0.5pt}
\raisebox{-0.7ex}{\includegraphics[scale=#1]{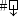}}
}

\startlocaldefs
\theoremstyle{plain}

\newtheorem{thm}{Theorem}[section]
\newtheorem{lemma}[thm]{Lemma}
\newtheorem{rem}{Remark}
\newtheorem{cor}[thm]{Corollary}
\newtheorem{proposition}[thm]{Proposition}
\theoremstyle{remark}

\endlocaldefs

\begin{document}

\begin{frontmatter}
\title{The Integrable Snake Model}
\runtitle{The Integrable Snake Model}

\begin{aug}
\author{\fnms{Samuel G. G.}~\snm{Johnston}\ead[label=e1]{samuel.g.johnston@kcl.ac.uk}}
\and
\author{\fnms{Rohan}~\snm{Shiatis}\ead[label=e2]{rohan.shiatis@kcl.ac.uk}}
\address{Department of Mathematics,
King's College London \printead[presep={,\ }]{e1}}

\address{Department of Mathematics,
King's College London \printead[presep={,\ }]{e2}}
\end{aug}

\begin{abstract}
A pure snake configuration is a bijection $\sigma:\mathbb{Z}^2 \to \mathbb{Z}^2$ containing no two-cycles and such that for each $x \in \mathbb{Z}^2$ we have
\begin{align*}
\sigma(x) \in \{ x , x+ \mathbf{e}^1, x+\mathbf{e}^2 , x- \mathbf{e}^2 \}.
\end{align*}
The non-trivial cycles of a pure snake configuration may be regarded as a collection of non-intersecting paths in $\mathbb{Z}^2$ that may travel right, up, or down (but not left) from a given vertex. 
Pure snake configurations are a generalisation of lozenge tilings, which are in natural correspondence with paths that only travel right or up.
We introduce a partition function on a finite version of this model and study the probabilistic properties of random pure snake configurations chosen according to their contribution to this partition function.
Under a suitable weighting, the model is integrable in the sense that we have access to explicit formulas for its partition function and correlation function. We utilise the integrable structure of this model in several applications through its various scaling limits.

\end{abstract}



\end{frontmatter}

\begin{figure}[h!t]
    \includegraphics{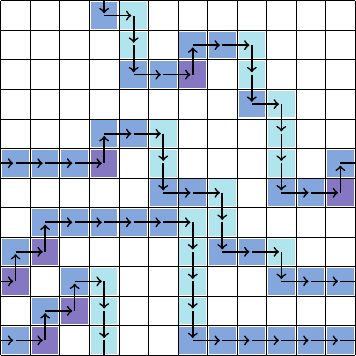}
\end{figure}

\section{Introduction and overview}


\subsection{{Kasteleyn theory}}

Consider the classical problem of counting $T_n$, the number of ways of tiling an $n \times n$ square with $2 \times 1$ dominoes. This seemingly innocuous combinatorial problem is far more challenging than one might initially expect, with attempts at elementary inductive or bijective arguments falling short. 
It was not until the 1960s work of Kasteleyn \cite{kasteleyn} and Temperley and Fisher \cite{TF} that the necessary tools were developed to tackle this problem. The `Kasteleyn theory' approach hinges on an ingenious idea: one may associate each pair of domino tilings of the $n \times n$ square with cycles of a bijection that maps lattice points in the $n \times n$ square to their neighbours. One can consequently use this correspondence to show that the number of domino tilings of the $n \times n$ square coincides with the square root of the determinant of a certain adjacency operator with complex weights. By diagonalising this operator and evaluating its determinant as a product of eigenvalues, one can derive an exact expression for $T_n$. Namely, provided $n$ is even, we have that
\begin{equation*}
T_n = \prod_{j=1}^n \prod_{k=1}^n \left| 2 \cos \left( \frac{\pi j}{n+1} \right) + 2i \cos \left( \frac{\pi k}{n+1}\right) \right|^{1/2}.
\end{equation*}
With the hindsight of such a mysterious formula for $T_n$, perhaps it is unsurprising that the number of domino tilings cannot be easily counted by elementary means.

The related problem of counting the number of ways of tiling a hexagon with side lengths $A,B,C$ with lozenges was first studied by MacMahon over a century ago \cite{macmahon}, with MacMahon conjecturing --- and later proving --- his eponymous product formula $\prod_{a=1}^A \prod_{b=1}^B\prod_{c=1}^C \frac{a+b+c-1}{a+b+c-2}$ for the number of such tilings. Again, lozenge tilings correspond bijectively with the cycles of certain bijections on the hexagon and by using this relation, MacMahon's formula can be proved rapidly using the tools of Kasteleyn theory. We refer the reader to the excellent recent book of Gorin \cite{gorin} for further information on this calculation, as well as to Kenyon and Okounkov's survey article \cite{KO} on dimer models. 

Beyond its power in computing partition functions for statistical mechanics models, a remarkable feature of the Kasteleyn theory approach is its ability to provide explicit determinantal formulas for the local correlations of random objects, selected based on their contribution to a partition function. These correlations can be expressed as determinants of the inverse Kasteleyn operator, situating such models within the framework of determinantal processes—a class of stochastic processes whose correlations are governed by determinants \cite{HKPV, soshnikov}. These formulas have significant applications, such as in establishing a variational principle for domino tilings \cite{CKP} and in connecting the asymptotic behaviour of large dimer models with the spectral curve of the Kasteleyn operator \cite{KOS}.

In the wider integrable probability literature there are a host of models that are exactly solvable in the sense that we have access to explicit formulas for their partition functions and correlation structure. The motivations for understanding these models arise from both mathematical physics and from subfields of pure mathematics such as representation theory. Such models include Schur and Macdonald processes \cite{OR, BC, BGC} and other models arising in representation theory \cite{BG}, square ice {models} and alternating sign matrices \cite{gorinASM, AR}, 
the Ising model and 
six-vertex models \cite{baxter}, random polymers \cite{BCF} and connections with the RSK correspondence \cite{OSZ, COSZ}, longest increasing subsequences in uniform permutations \cite{BDJ, romik}, uniform spanning trees \cite{pemantle, schramm}, 
TASEP and ASEP \cite{BL, BL2},
self-avoiding walks \cite{D-CS}, TASEP/pushTASEP interactions on Gelfand-Tsetlin patterns \cite{assiotis} and a myriad of connections with random matrix theory \cite{ACJ, johannsonR}. This list, of course, is by no means exhaustive. We refer the reader to Baxter's monograph \cite{baxter} and Borodin and Gorin's lecture notes \cite{BGlectures} {for further examples.}
\vspace{3mm}

\subsection{{Lozenge tilings}}
We now home in on \textbf{lozenge tilings}, the exactly solvable model most relevant to our analysis.
We choose to take a slightly unconventional perspective on lozenge tilings, which facilitates a more transparent development of their Kasteleyn theory. Namely, a standard coordinate change (see e.g.\ Chapter 2 of Gorin's book \cite{gorin} or Figure~\ref{fig:lozenge}) tells us that a lozenge tiling on the plane may be regarded as a bijection $\sigma:\mathbb{Z}^2 \to \mathbb{Z}^2$ with the property that
\begin{align} \label{eq:bijloz}
\sigma(x) \in \{ x, x+\mathbf{e}^1, x+\mathbf{e}^2\} \quad \text{for all $x \in \mathbb{Z}^2$}.
\end{align}
Here $\mathbf{e}^1 := (1,0)$ and $\mathbf{e}^2 := (0,1)$ denote the standard basis of $\mathbb{Z}^2$. 

Under this association, one may also consider lozenge tilings of not just the plane but also the semi-infinite discrete cylinder or discrete torus. One can `draw' a bijection of the form in \eqref{eq:bijloz} by drawing an arrow from the point $x \in \mathbb{Z}^2$ to the neighbouring point
$\sigma(x)$ and subsequently associate with the bijection a collection of bi-infinite non-intersecting paths in the plane. The three possibilities for the value of $\sigma(x)$ correspond to the three types of lozenge which tile the plane. We can associate these three possibilities with `fixed points', `right points' and `up points'.

\begin{figure}[h!t]
    \centering
    \includegraphics{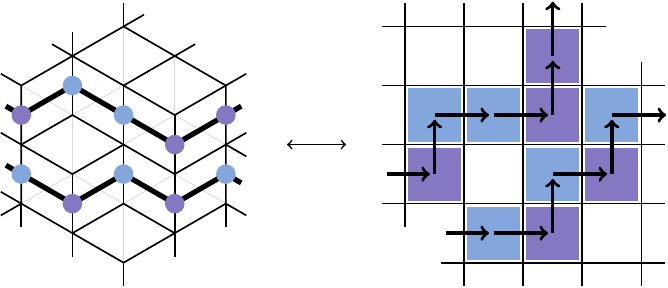}
    \caption{Visual representation of the bijection between lozenge tilings and bi-infinite non-intersecting paths that only move up or right.}
    \label{fig:lozenge}
\end{figure}

A notable result in the theory of lozenge tilings is that one may define a probability measure on tilings of the plane, with positive weights $1,\beta,\gamma > 0$ (by rescaling) which satisfy the triangle inequality controlling the preference for occurrences of fixed, right and up points respectively in that, loosely speaking, the probability of choosing a lozenge tiling associated with a bijection $\sigma$ is given by 
\begin{align} \label{eq:propto}
\mathbf{P}^{\beta,\gamma}(\sigma) \propto  \beta^{\# \{ x : \sigma(x) = x+\mathbf{e}^1 \} }  \gamma^{\# \{ x : \sigma(x) = x+\mathbf{e}^2 \} } \qquad \beta,\gamma \geq 0,
\end{align}
where it is possible to make \eqref{eq:propto} precise by first restricting the model to a finite torus and then letting the {side-lengths} of this torus tend to infinity.

It transpires that correlations of a tiling chosen randomly according to $\mathbf{P}^{\beta,\gamma}$ have an explicit integrable structure: the local correlations of the values taken by $\sigma$ are captured in terms of determinants involving the celebrated (generalised) discrete sine kernel. Namely, for such a bijection chosen randomly with these weights, given any distinct spatial locations $x_1,\dots,x_k$ in $\mathbb{Z}^2$, we have
\begin{align} \label{eq:lozcorr}
\mathbf{P}^{\beta,\gamma} \left( \bigcap_{i=1}^k \{ \sigma(x^i) = x^i + \mathbf{e}^1 \} \right) = \det_{i,j=1}^k \left( - \int_{\overline{w_{\beta,\gamma}} }^{w_{\beta,\gamma}} \frac{\mathrm{d}w}{2\pi \iota w } \frac{ w^{ - (x^j_2-x^i_2)}}{ (1 + \gamma w )^{ x^j_1 - x^i_1}} \right),
\end{align}
where the contour integral $\int_{\overline{w_{\beta,\gamma}} }^{w_{\beta,\gamma}}$ travels from $\overline{w_{\beta,\gamma}}$ to $w_{\beta,\gamma}$, going left of the origin if $x^j_1 \leq x^i_1$ and to the right of the origin if $x^j_1 > x^i_1$. Here $w_{\beta,\gamma}$ is the unique complex solution with positive imaginary part to the equation $|1 + \gamma w | = |\beta|$. 
The equation \eqref{eq:lozcorr} appears (up to a change of parameter) in Section 2.2 of \cite{GP}, though also follows from the special case $\delta = 0$ of Theorem \ref{thm:m1m2infinity} of the present article.

The correlation kernel in \eqref{eq:lozcorr} is sometimes called the extended discrete sine kernel for the reason that if all of the $x^i$ lie on the same vertical line, then \eqref{eq:lozcorr} reduces to
\begin{align} \label{eq:lozcorr2}
\mathbf{P}^{\beta,\gamma} \left( \bigcap_{i=1}^k \{ \sigma(x^i) = x^i + \mathbf{e}^1 \} \right) = \det_{i,j=1}^k \left( \frac{ \sin( \pi \tau (h^j - h^i)) }{ \pi (h^j - h^i) } \right), \quad  \text{$x^i = (0,h^i), h^i \in \mathbb{Z}, 1 \leq i \leq k$},
\end{align}
where $\pi \tau$ is the angle between $w_{\beta,\gamma}$ and the negative real axis. The equation \eqref{eq:lozcorr2} may be obtained from \eqref{eq:lozcorr} after a brief calculation.

The correlation structure appearing in \eqref{eq:lozcorr} and \eqref{eq:lozcorr2} is ubiquitous as a limit for countless models emerging not just in statistical physics \cite{GP, petrov} but also in representation theory. Okounkov and Reshetikhin \cite{OR} used tools from representation theory to show that the discrete sine kernel describes the local correlations of random three-dimensional Young diagrams. From a probabilistic perspective, Gorin and Petrov \cite{GP} showed that this model can be regarded as a limit of independent discrete time walkers on $\mathbb{Z}$, that jump up with probability $\beta$, or stay at the same level with probability $1-\beta$ and are thereafter conditioned to never collide. The discrete sine kernel itself appears in Borodin, Okounkov and Olshanski \cite{BOO} as a descriptor of the bulk cycle length structure of a Plancherel measure on the symmetric group.

Boutillier \cite{boutillier} used a scaling limit of the lozenge tiling model to obtain an explicit correlation function for the bead model. Roughly speaking, this follows from studying in \eqref{eq:lozcorr} the scaling limit $\gamma = \varepsilon T$ and $x_i = \lfloor s_i/\varepsilon \rfloor$, so that the rate of `up' squares becomes small, but horizontal space is rescaled accordingly so that the occurrences of these "up"-squares, which occur continuously in space in the scaling limit, occur with a frequency $O(1)$. See also Sun \cite{sun}.

In \cite{johnston}, the first author takes a local version of Boutillier's bead model on the semi-discrete torus to obtain explicit correlation functions for the model and relate these to a stochastic process introduced by Gordenko in \cite{gordenko}. This stochastic process of Gordenko involves $\ell$ walkers on the ring $\mathbb{Z}_n = \{0,1,\dots,n-1\}$, each independently jumping from $x$ to $x+1$ for $0 \leq x < n-1$, or from $n-1$ to $0$, at rate 1 and are subsequently conditioned to avoid collisions. 

For further literature on lozenge tilings and their limits, the reader might consult \cite{gorin17, gorin08, kenyon97, johansson05, BK98, aggarwal, AG22, LT, petrov15}.

\subsection{{The integrable snake model}}
In the article at hand, we study a natural generalisation of lozenge tiling models, which we call the \textbf{integrable snake model}. The integrable snake model consists of bijections
$\sigma\colon\mathbb{T} \to \mathbb{T}$ on one of the sets
\begin{align*}
\mathbb{T} = \mathbb{Z}^2, \quad \mathbb{T} = \mathbb{Z} \times \mathbb{Z}_{m_2}, \quad \mathbb{T} = \mathbb{Z}_{m_1} \times \mathbb{Z} \quad \text{or} \quad \mathbb{T} =  \mathbb{Z}_{m_1}\times \mathbb{Z}_{m_2}
\end{align*}
satisfying the relaxation of \eqref{eq:bijloz} given by 
\begin{align} \label{eq:URDcond}
\sigma(x) \in \{ x, x+\mathbf{e}^1, x+\mathbf{e}^2 , x -\mathbf{e}^2 \} \qquad \text{for $x \in \mathbb{T}$.}
\end{align}
We call any bijection $\sigma\colon\mathbb{T} \to \mathbb{T}$ satisfying \eqref{eq:URDcond} a \textbf{generalised snake configuration on $\mathbb{T}$.} We say a generalised snake configuration is a \textbf{pure snake configuration} if it contains no two-cycles, i.e.\ for which
\begin{align*}
\text{$\sigma$ is pure} \iff \text{there does not exist $x \in \mathbb{T}$ such that $\sigma(x) = x+\mathbf{e}^2$ and  $\sigma(x+\mathbf{e}^2) = x$}.
\end{align*}
The integrable snake model may be regarded as a dimer model (in other words, a matching on a bipartite graph) as follows. For each location $x \in \mathbb{T}_m$, we associate a black vertex $b_x$ and a white vertex $w_x$. It is helpful to depict the white vertex situated just to the right of the black vertex. With every generalised snake configuration $\sigma$, we may associate a matching by connecting $w_x$ to $b_{\sigma(x)}$. We have depicted this underlying dimer model in Figure \ref{fig:dimer}. 

\begin{figure}[h!tp]
\centering 
\includegraphics[]{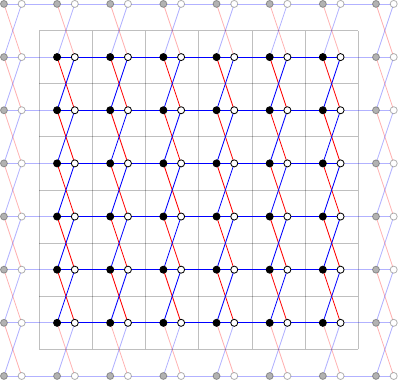}
\caption{The bipartite graph structure underlying the integrable snake model.}
\label{fig:dimer}
\end{figure}

The usual dimer representation of conventional lozenge tilings is recovered by excluding the red edges from Figure~\ref{fig:dimer}. With these edges removed, the underlying graph is planar. Our model, by contrast, is built on a bipartite graph that is nonplanar, owing to the presence of two-cycles, which appear in the dimer picture as crossing edges. The condition that a snake configuration be pure is therefore precisely the condition that the associated matching be noncrossing. Thus, our problem may be viewed as that of sampling a random planar matching from a nonplanar graph.

At first sight, one might expect the Kasteleyn theory approach to break down once the planar structure of the graph is lost. The central innovation of our approach is that, by suitably weighting the model to account for this non-planarity, the model remains integrable in the sense that we retain explicit formulas for its partition function and correlation functions. We will provide the full details in the next section, but for now, let us close this introductory section with a summarised verbal account of our foundational main results.

\begin{thm} \label{thm:premain}
A suitably weighted version of the pure snake configuration model on the discrete torus $\mathbb{T}_{\mathbf{m}} := \mathbb{Z}_{m_1}\times \mathbb{Z}_{m_2}$ is exactly solvable in that it is a pushforward of a signed determinantal process. This provides an explicit formula for its partition function and correlation functions. 
\end{thm}

The precise mathematical content of Theorem \ref{thm:premain} can be found across Theorem \ref{thm:pf} and Theorem \ref{thm:cf} (or its generalisation, Theorem~\ref{thm:cffull0}).

\subsection{Overview of the article}
We now overview the remainder of the article:
\begin{itemize}
\item In Section \ref{sec:statements} we state our foundational main results in full, giving an explicit description of the integrable structure of the pure snake model.
\item In Section \ref{sec:applications} we describe the applications of the integrable structure of the pure snake model through its various scaling limits.
\item In Section \ref{sec: fibonacci} we develop our Kasteleyn theory for the integrable snake model, ultimately leading to proofs of our foundational results, Theorem \ref{thm:pf} and Theorem \ref{thm:cf}.
\item In Section \ref{sec:m1infinity} we study the $m_1 \to \infty$ and $m_1,m_2 \to \infty$ limits of our model, proving Theorem \ref{thm:m1infinity} and Theorem \ref{thm:m1m2infinity}.
\item In Section \ref{sec:noncoll} we study a scaling limit with $m_1 \to \infty$ and $\gamma,\delta$ taken to zero with a subsequent scaling limit of horizontal space. Thereafter, we study basic probabilistic aspects of the model in the scaling limit.
\item In Section \ref{sec:equiv}, we use the cyclic Karlin-McGregor formula and techniques from the theory of $h$-transforms to show that the model of the previous section coincides with that of asymmetric Poisson walkers on $\mathbb{Z}_n$ conditioned never to collide. 
\item In Section \ref{sec:traffic} we prove Theorem \ref{thm:traffic}, which connects the asymmetric exclusion process {(ASEP)} on $\mathbb{Z}_n$ with walkers conditioned to never collide through an exponential martingale change of measure involving the `traffic' of a configuration of particles.
\end{itemize}

\section{Statements of our main results}  \label{sec:statements}

\subsection{Foundational results} \label{sec:found}

\begin{figure}[h!t]
  \includegraphics{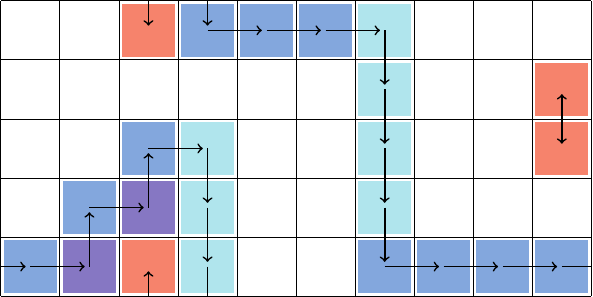}
  \caption{A generalised snake configuration. The snakelets are highlighted in red. If the two-cycles corresponding to snakelets were removed and replaced with fixed points, then we would obtain a pure snake configuration.}
  \label{fig: defgreatsnakeconfig}
\end{figure}

Let us take $\mathbb{T} = \torus{\mathbf{m}} \coloneqq \mathbb{Z}_{m_1} \times \mathbb{Z}_{m_2}$. Recall that a snake configuration is simply a bijection $\sigma\colon  \torus{\mathbf{m}} \to \torus{\mathbf{m}}$ satisfying \eqref{eq:URDcond}. We say that a snake configuration is \textbf{pure} if it contains no two-cycles. We refer to these two-cycles as \textbf{snakelets}. Consequently, we can define the sets
\begin{align*}
\URDbar_\mathbf{m} &:= \{ \os \text{ generalised snake configuration on $\torus{\mathbf{m}}$} \},\\
\URD_{\mathbf{m}} &:= \{ \sigma \text{ pure snake configuration on $\torus{\mathbf{m}}$} \}.
\end{align*}
Of course, $\mathrm{URD}_{\mathbf{m}} \subseteq \URDbar_{\mathbf{m}}$. We will generally notate generalised snake configurations with $\os$ and pure snake configurations with $\sigma$. 

We begin by constructing a partition function on pure snake configurations on the discrete torus $\mathbb{T}_{\mathbf{m}}$ in the vein of \eqref{eq:propto}. To this end, given a generalised snake configuration $\os:\mathbb{T}_{\mathbf{m}} \to \mathbb{T}_{\mathbf{m}}$ we define 
\begin{align*}
\fixedbox(\os) &:= \# \{ x : \os(x) = x \}, & \rightbox(\os) &:= \# \{ x : \os(x) = x + \mathbf{e}^1 \}, \\
\upbox(\os) &:= \# \{ x : \os(x) = x + \mathbf{e}^2 \}, & \downbox(\os) &:= \# \{ x : \os(x) = x - \mathbf{e}^2 \}.
\end{align*}
Since $\mathbb{T}_{\mathbf{m}}$ has $m_1m_2$ elements, for any generalised snake configuration $\os$ we have
\begin{align*}
\fixedbox(\os)  + \rightbox(\os) + \upbox(\os)   +  \downbox(\os)  = m_1m_2
\end{align*}
Given parameters $\alpha > 0$, $\beta,\gamma,\delta \geq 0$, define the weighting of a pure snake configuration by 
\begin{align} \label{eq:wdef}
w(\sigma) = w_{\alpha,\beta,\gamma,\delta}(\sigma) = \alpha^{\fixedbox(\sigma)} \beta^{\rightbox (\sigma)}\gamma^{\upbox (\sigma)}\delta^{\downbox (\sigma)}.
\end{align}
Ideally, we would like to compute the partition function $Z':= \Sigma_{\sigma \in \URD_{\mathbf{m}}} w(\sigma)$ and thereafter go about choosing a random pure snake configuration according to the probability measure $P'(\sigma) = w(\sigma)/Z'$. However, as far as the authors can tell, this model is not integrable, in that there are no simple closed formulas for the partition function or the correlation functions.

However, by adjusting the weighting of our model, we can make the process fully integrable, meaning that we have access to explicit determinantal formulas for the partition function and correlation functions. 
 Given a pure snake configuration on $\mathbb{T}_\mathbf{m}$, a vertical pre-gap is a set of the form $g = \{ x, x +\mathbf{e}^2, x+2\mathbf{e}^2,\dots, x+(k-1)\mathbf{e}^2\}$ with the property that $\sigma(z) = z$ for each $z \in g$. A vertical gap is a vertical pre-gap that is maximal in the sense that it is not properly contained in any larger vertical pre-gap. We denote by $|g|$ the cardinality of the vertical gap. We let $G_\sigma$ denote the collection of all vertical gaps associated with $\sigma$. 
We define the \textbf{Fibonacci factor} of a pure snake configuration $\sigma$ to be the quantity
\begin{equation} \label{eq:J1}
J(\sigma) := J_{\alpha,\beta,\gamma,\delta}(\sigma) = \prod_{g \in G_\sigma} f_{|g|}\left(-\frac{\gamma \delta}{\alpha^2}\right)
\end{equation}
where, when $\lambda \neq \frac{1}{4}$,
\begin{align} \label{eq:fdef}
f_n( - \lambda) = \frac{1}{c_\lambda} \left\{ \left( \frac{1+c_\lambda}{2} \right)^{n+1} - \left( \frac{1-c_\lambda}{2} \right)^{n+1} \right\}, \qquad c_\lambda := \sqrt{1-4\lambda},
\end{align}
and $f_n(-1/4) = (n+1)2^{-n}$.
  
Observe that $f_n(-\lambda)$ is real-valued, even when $1 - 4\lambda < 0$ (i.e. $c_\lambda$ is purely imaginary). In fact, $f_n(-\lambda)$ is {a polynomial with real coefficients}. Our motivation behind calling this the Fibonacci factor is that the function $f$ satisfies the following Fibonacci-like recurrence relation 
\begin{equation} \label{eq:recc}
f_n(\lambda) = f_{n-1}(\lambda) + \lambda f_{n-2}(\lambda).
\end{equation}
In particular, the numbers $\{ f_n(1):n=1,2,\dots\}$ are the Fibonacci numbers. Moreover, the fact that $f_n$ is a polynomial follows directly by induction on the recurrence \eqref{eq:recc}, since $f_1$ and $f_2$ are linear in $\lambda$.
    
In any case, the Fibonacci factor $J(\sigma)$ of a pure snake configuration arises naturally from counting the number of ways of tiling a vertical gap between consecutive long cycles of snakes with dominoes. Each tiling of the dominoes corresponds to a specific arrangement of snakelets; see Section \ref{sec: fibonacci} for further discussion.

We now define a partition function on pure snake configurations on the $m_1 \times m_2$ discrete torus $\mathbb{T}_{\mathbf{m}} = \mathbb{Z}_{m_1}\times \mathbb{Z}_{m_2}$ by setting
\begin{align} \label{eq:PF}
Z_{\mathbf{m}} = Z_{\mathbf{m}}(\alpha,\beta,\gamma,\delta) := \sum_{ \sigma:\mathbb{T}_\mathbf{m} \to \mathbb{T}_{\mathbf{m}} \text{ pure} } {w}(\sigma)J(\sigma),
\end{align}
where ${w}(\sigma)$ is as in \eqref{eq:wdef}, $J(\sigma)$ is as in \eqref{eq:J1} and the sum in \eqref{eq:PF} is taken over all pure snake configurations on $\mathbb{T}_{\mathbf{m}}$. We call the quantity $w(\sigma)J(\sigma)$ the \textbf{Fibonacci weight} of a pure snake configuration.

We are now equipped to state our foundational results on the exactly solvable structure of the pure snake configuration model with Fibonacci weightings. The first of these results is an exact formula for the partition function in \eqref{eq:PF}. We will show that the partition function may be expressed as a signed sum of determinants of four operators, each of which admits an explicit closed-form expression. These operators $K_{\theta,\mathbf{m}}:\mathbb{T}_{\mathbf{m}} \times \mathbb{T}_{\mathbf{m}} \to \mathbb{C}$ are defined for each of the four values of $\theta =(\theta_1,\theta_2) \in \{0,1\}^2$ and are given by 
\begin{align}\label{defK0}
 K_{\theta, \mathbf{m}}(x,y) := \alpha \indic{y = x} + \beta e^{\pi \iota \frac{\theta_1}{m_1}} \indic{y = x + \mathbf{e}^1} + \gamma e^{\pi \iota \frac{\theta_2}{m_2}} \indic{y = x + \mathbf{e}^2} + \delta e^{-\pi \iota \frac{\theta_2}{m_2}} \indic{y = x - \mathbf{e}^2}.
\end{align}
Our first result is a Kasteleyn theorem for the integrable snake model with Fibonacci weightings, stating that the partition function can be written as a sum of determinants:

\begin{thm} \label{thm:pf}
The partition function $Z_{\mathbf{m}}$ in \eqref{eq:PF} may be written as a sum of four determinants:
\begin{align} \label{eq:pf1}
Z_{\mathbf{m}} = \sum_{ \theta \in \{0,1\}^2 } C_{\theta,\mathbf{m}} \det(K_{\theta,\mathbf{m}}),
\end{align}
where $C_{\theta,\mathbf{m}} = \frac{1}{2}(-1)^{(\theta_1 + m_1 + 1)(\theta_2 + m_2 + 1)}$.
Moreover, the determinant of $K_{\theta,\mathbf{m}}$ is given by 
  \begin{equation}\label{detKevalsequation0}
    \det(K_{\theta,\mathbf{m}}) = \prod_{z^{m_1}=(-1)^{\theta_1}, w^{m_2} = (-1)^{\theta_2}} (\alpha + \beta z + \gamma w + \delta w^{-1} ),
  \end{equation}
where the product is taken over all pairs $(z,w)$ of complex numbers where $z^{m_1} = (-1)^{\theta_1}$ and $w^{{m_2}} = (-1)^{\theta_2}$. 
\end{thm}

We would now like to define a probability measure $P_{\mathbf{m}}$ on pure snake configurations $\sigma$ on $\mathbb{T}_{\mathbf{m}}$ by sampling a random $\sigma$ according to its contribution to the partition function $Z_{\mathbf{m}}$. Before this however, we need to consider the question of whether the Fibonacci factor $J(\sigma)$ is non-negative. 

We say that our four non-negative parameters $\alpha > 0$, $\beta,\gamma,\delta \geq 0$ lie in the \textbf{probabilistic regime} if they satisfy the positive discriminant equation
\begin{align} \label{eq:probregime}
{\alpha^2 - 4 \gamma \delta \geq 0.}
\end{align}
Of course, the value of $\beta$ does not affect whether or not we are in the probabilistic regime. The equation \eqref{eq:probregime} {(including the case of equality)} implies the strict positivity of $f_n(- \gamma \delta/\alpha^2)$ for all $n \geq 1$. In particular, we have the following remark:
  \begin{rem} \label{rem:positivity}
      If \eqref{eq:probregime} holds, then $J(\sigma) > 0$ for all $\sigma$. 
  \end{rem}

It is also possible to have $J(\sigma)$ non-negative for all pure snake configurations without \eqref{eq:probregime} holding. In this situation, where $J(\sigma) > 0$ but \eqref{eq:probregime} fails, all of our subsequent results on the finite lattice remain true. 
However, the condition \eqref{eq:probregime} arises naturally in other contexts and is imperative for the analysis of any scaling limit; see Section~\ref{sec:m1infinity} for further discussion. The precise conditions depend on the value of $m_2$ in a rather complicated manner, so for simplicity, when discussing correlations for varying choices of $m_2$, we will assume that \eqref{eq:probregime} holds.

Regardless, it follows that, in the probabilistic regime, we can define a probability measure $P_\mathbf{m}$ on the set of pure snake configurations by setting it to be proportional to the Fibonacci weight of $\sigma$,
\begin{align} \label{eq:Pdef}
P_{\mathbf{m}}(\sigma) \coloneqq \frac{ {w}(\sigma)J(\sigma)}{Z_{\mathbf{m}}}.
\end{align}

Our second foundational result describes the local correlations of pure snake configurations under $P_{\mathbf{m}}$. In order to state the result, let us say that the parameters $\alpha$, $\beta$, $\gamma$, $\delta$ are \textbf{generic} if $|\alpha +\gamma w + \delta w^{-1}| \neq \beta$ for every $w$ with $w^{n} = 1$ or $w^n = -1$. In Section~\ref{sec:m1infinity}, we will revisit this condition in greater depth and its relation to the phase transitions of the model. For now, we note that it furnishes a sufficient condition for the denominator in the following expression to be non-zero.

\begin{thm} \label{thm:cf}
In the probabilistic regime \eqref{eq:probregime} with generic parameters, $P_{\mathbf{m}}$ defines a probability measure on pure snake configurations. Moreover, for any distinct elements $x^1,\dots,x^k$ of $\mathbb{T}_{\mathbf{m}}$ we have the determinantal correlation formula
  \begin{align} \label{eq:cf}
P_\mathbf{m}\left( \bigcap_{i=1}^k \{ \sigma(x^i) = x^i + \mathbf{e}^1 \} \right) &=  \sum_{\theta \in \{0,1\}^2 } \lambda_{\theta,\mathbf{m}}\det_{i,j = 1}^k  \left( G_{\theta,\mathbf{m}}(x^i,x^j) \right),
  \end{align}
  where $\lambda_{\theta,\mathbf{m}} = C_{\theta,\mathbf{m}}\det(K_{\theta,\mathbf{m}})/Z_{\mathbf{m}}$ and the operator $G_{\theta,\mathbf{m}}: \mathbb{Z} \times \mathbb{Z} \to \mathbb{C}$ is given by
\begin{align} \label{eq:Jdef}
   G_{\theta,\mathbf{m}}(x,y)   &= \frac{\beta}{m_1 m_2} \sum_{\substack{z^{m_1} = (-1)^{\theta_1}\\ w^{m_2} = (-1)^{\theta_2}} } \frac{ z^{ - (y_1-x_1-1) } w^{ - (y_2 - x_2) } }{\alpha + \beta z + \gamma w + \delta w^{-1}}.
  \end{align}
\end{thm}
The operator $G_{\theta,\mathbf{m}}$ is in fact  $\beta K_{\theta,\mathbf{m}}^{-1}$, with the extra prefactor being a consequence of taking the probability of events of the form \{$\sigma(x^i) = x^i + \mathbf{e}^1$\}. 

In a moment, we outline how it is also possible to provide determinantal formulas for probabilities of more complicated events involving up and down-moves too, such as $P_{\mathbf{m}}( \sigma(x^1) = x^1 + \mathbf{e}^1 , \sigma(x^2) = x^2 + \mathbf{e}^2, \sigma(x^3) = x^3 - \mathbf{e}^2)$. Each of these probabilities may be written in terms of a signed sum of determinants; see Theorem \ref{thm:cffull0}.

\subsection{Proof ideas: torus links, Kasteleyn theory and pushforwards}

In this section we take a moment to expound on the ideas behind the proofs of Theorem~\ref{thm:pf} and Theorem~\ref{thm:cf}. 
Recall from the beginning of Section \ref{sec:found} that a generalised snake configuration is a bijection $\sigma:\torus{\mathbf{m}} \to \torus{\mathbf{m}}$ satisfying \eqref{eq:URDcond}, that a snake configuration is pure if it has no two cycles, and that $\sigma$ (resp.\ $\bar{\sigma})$ is used to denote pure (resp.\ generalised) snake configurations. Let us now extend \eqref{eq:wdef} from pure snake configurations $\sigma$ to generalised snake configurations $\os$ by defining 
\begin{align} \label{eq:weightsnake}
{w}(\other{\sigma}) := {w}_{\alpha,\beta,\gamma,\delta}(\other{\sigma}) :=  (-1)^{N(\other{\sigma})}\alpha^{\fixedbox(\other{\sigma})}\beta^{\rightbox(\other{\sigma})}\gamma^{\upbox(\other{\sigma})}\delta^{\downbox(\other{\sigma})}
\end{align}
where, as before, $\fixedbox,\rightbox,\upbox,\downbox \,$ denote the number of fixed, right, up and down points associated with $\other{\sigma}$ respectively and $N(\other{\sigma})$ denotes the number of snakelets, that is
\begin{equation} \label{eq:ns}
  N(\other{\sigma}) = \frac{1}{2}\#\{x : \other{\sigma}^2 (x) = x, \sigma(x) \neq x\}.
\end{equation}
We emphasise that if $\overline{\sigma}$ is a pure snake configuration, then $N(\overline {\sigma}) = 0$ and \eqref{eq:weightsnake} and \eqref{eq:wdef} coincide.

Given a generalised snake configuration $\other{\sigma}$, consider the decomposition of $\other{\sigma}$ into disjoint cycles. We may write $\other{\sigma} = C_1 \cdots C_r \tau_1 \cdots \tau_s$, where $C_i$ are cycles of length $\geq 3$, referred to as \textbf{long cycles}, and $\tau_i$ are transpositions. We refer to the pure snake configuration $\mathrm{sh}(\other{\sigma}) = \sigma := C_1\cdots C_r$ obtained from $\os$ by deleting all of its two-cycles as the \textbf{shape} of $\other{\sigma}$. Note that this procedure gives rise to a projection
\begin{align*}
\mathrm{sh}:\URDbar_\mathbf{m} \to \URD_\mathbf{m}
\end{align*}
sending every generalised snake configuration to its shape. 

For $\mathrm{sh}(\other{\sigma}) = \sigma$, consider the relationship between ${w}(\other{\sigma})$ and ${w}(\sigma)$. 
The projection $\mathrm{sh}(\os)$ deletes $N(\overline{\sigma})$ snakelets from $\overline{\sigma}$. Since replacing a snakelet by two fixed squares replaces an instance of $\gamma \delta$ with $\alpha^2$ in the weighting \eqref{eq:weightsnake}, we have 
\begin{align} \label{eq:creat}
{w}(\other{\sigma}) = (-\gamma \delta/\alpha^2)^{N(\other{\sigma})} {w}(\sigma) \qquad \sigma = \mathrm{sh}(\other{\sigma}).
\end{align}
The following result outlines the source of the Fibonacci weighting defined in \eqref{eq:J1}.

\begin{lemma} \label{lem:shapepush}
Write $\mathrm{sh}^{-1}(\sigma) := \{ \other{\sigma} : \mathrm{sh}(\other{\sigma}) = \sigma \}$. Then with $J(\sigma)$ as in \eqref{eq:J1} we have
\begin{align} \label{eq:shapepush}
{w}(\sigma)J(\sigma) = \sum_{ \other{\sigma} \in \mathrm{sh}^{-1}(\sigma) } {w}(\other{\sigma}).
\end{align}
\end{lemma}

In light of \eqref{eq:creat}, an alternative formation of \eqref{eq:shapepush} is 
\begin{align}\label{eq:Jaltform}
J(\sigma) = \sum_{ \other{\sigma} \in \mathrm{sh}^{-1}(\sigma) } (-\gamma \delta/\alpha^2)^{N(\other{\sigma})}.
\end{align}

Lemma \ref{lem:shapepush} 
is proven in Section \ref{sec: fibonacci}. 

One consequence of Lemma \ref{lem:shapepush} is that the partition function $Z_{\mathbf{m}} := \sum_{ \sigma \in \URD_{\mathbf{m}}} J(\sigma)w(\sigma)$ defined in \eqref{eq:PF} may alternatively be written
\begin{align} \label{eq:pfnew}
Z_{\mathbf{m}}(\alpha,\beta,\gamma,\delta) = \sum_{ \other{\sigma} \in \URDbar_{\mathbf{m}}} {w}(\other{\sigma}),
\end{align}
where the sum is over all generalised snake configurations (recall that \eqref{eq:weightsnake} generalises $w$ to $\URDbar$). It is in fact this representation which we exploit to calculate the partition function. 

Given a generalised snake configuration $\os$, the long cycles of $\os$ may be associated with a \textbf{torus link}, namely, a collection of disjoint knots on the (continuous) two-dimensional torus $S^1 \times S^1$. In developing our Kasteleyn theory for generalised snake configurations, we appeal to the following well-known result in elementary knot theory.

\begin{proposition} \label{prop:links}
Let $K_1,\dots,K_r$ be a collection of disjoint knots on the two-dimensional torus. Then the winding numbers of each of these knots coincide. Moreover, the horizontal and vertical winding numbers $q_1$ and $q_2$ are coprime.
\end{proposition}

See, e.g., Chapter 7 of Murasugi \cite{murasugi}. We discuss the knot-theoretic aspects of our approach further in Section \ref{sec:knots}.

By exploiting Proposition \ref{prop:links} we are able to establish that with the operators $K_{\theta,\mathbf{m}}$ defined in \eqref{defK0} and coefficients $C_{\theta,\mathbf{m}}$ as in Theorem \ref{thm:pf}, for any generalised snake configuration we have
\begin{align} \label{eq:weightsum0pre}
\sum_{\theta \in \{0,1\}^2 } C_{\theta,\mathbf{m}} \mathrm{sgn}(\os) \prod_{x \in \mathbb{T}_{\mathbf{m}}} K_{\theta,\mathbf{m}}(x ,\os(x) ) = w(\os).
\end{align}
By summing \eqref{eq:weightsum0pre} over all generalised snake configurations and using \eqref{eq:pfnew}, we obtain the first equation, \eqref{eq:pf1}, of Theorem \ref{thm:cf}. 
The latter equation, \eqref{detKevalsequation0}, follows from diagonalising the operator $K_{\theta,\mathbf{m}}$; see Lemma \ref{lem:eigensystem}.

We now turn to the question of how the correlation structure of pure snake configurations arises in relation to the pushforward map. To this end, provided the determinant $\det(K_{\theta,\mathbf{m}})$ is non-zero, for each $\theta \in \{0,1\}^2$, we may define a \emph{complex} probability measure $P_{\theta,\mathbf{m}}:\URDbar_{\mathbf{m}} \to \mathbb{C}$ on generalised snake configurations by setting
\begin{align*}
P_{\theta,\mathbf{m}}( \os) := \frac{1}{\det(K_{\theta,\mathbf{m}})}  \mathrm{sgn}(\os) \prod_{x \in \mathbb{T}_{\mathbf{m}}} K_{\theta,\mathbf{m}}(x ,\os(x) ).
\end{align*}
A complex probability measure on a finite set $V$ is simply a function $P:V \to \mathbb{C}$ satisfying $\sum_{v \in V} P(v) = 1$, which we extend to a function on any subset $W \subset V$ by $P(W) = \sum_{v \in W} P(v)$. We will see shortly that $P_{\theta,\mathbf{m}}$ has highly tractable correlation structure. 
Moreover, using \eqref{eq:shapepush} in conjunction with \eqref{eq:weightsum0pre} and the definition of the genuine probability measure $P_{\mathbf{m}}$ on pure snake configurations, for any pure snake configuration $\sigma$ we have 
\begin{align*}
P_{\mathbf{m}}(\sigma) = \sum_{ \os \in \mathrm{sh}^{-1}(\sigma) } \sum_{\theta \in \{0,1\}^2 } \lambda_{\theta,\mathbf{m}} P_{\theta,\mathbf{m}}(\os).
\end{align*}
We now seek to describe in full the determinantal correlation structure of the probability measure $P_{\mathbf{m}}$. Given $x \in \mathbb{T}_{\mathbf{m}}$ and $\mathbf{f} \in \{0,\mathbf{e}^1,\mathbf{e}^2,-\mathbf{e}^2\}$, define the indicator function $R_x^{\mathbf{f}}:\URDbar_{\mathbf{m}} \to \{0,1\}$ on generalised snake configurations by 
\begin{align*}
R_x^{\mathbf{f}}(\os) := \indic{\os(x) = x + \mathbf{f}}.
\end{align*}
It is paramount to our program to understand how the indicator $R_x^{\mathbf{f}}$ interacts with the shape projection, $\mathrm{sh}:\URDbar_{\mathbf{m}} \to \URD_{\mathbf{m}}$. We are able to prove the following representations for the pushforwards of $R_x^{\mathbf{f}}$ under the shape map:
\begin{align}
T^0_x(\os) := R^0_x(\mathrm{sh}(\os)) &= R^0_x(\os) + R^{\mathbf{e}^2}_x(\os) R^{-\mathbf{e}^2}_{x+\mathbf{e}^2}(\os) + R^{-\mathbf{e}^2}_x(\os) R^{\mathbf{e}^2}_{x-\mathbf{e}^2}(\os) \label{eq:ev1b0} \\
T^{\mathbf{e}^1}_x(\os) := R^{\mathbf{e}^1}_x(\mathrm{sh}(\os)) &= R^{\mathbf{e}^1}_x(\os) \label{eq:ev2b0} \\
T^{\mathbf{e}^2}_x(\os) := R^{\mathbf{e}^2}_x(\mathrm{sh}(\os)) &= R^{\mathbf{e}^2}_x(\os) ( 1 - R^{-\mathbf{e}^2}_{x+\mathbf{e}^2}(\os) ) \label{eq:ev3b0} \\
T^{ - \mathbf{e}^2}_x(\os) := 
R^{-\mathbf{e}^2}_x(\mathrm{sh}(\os)) &= R^{-\mathbf{e}^2}_x(\os) ( 1 - R^{\mathbf{e}^2}_{x-\mathbf{e}^2}(\os) ). \label{eq:ev4b0}
\end{align}
The pertinent fact here is that every product $\prod_{i=1}^k T_{x^i}^{\mathbf{f}^i}$ has an expression
\begin{align} \label{eq:wordy0}
\prod_{i=1}^k T_{x_i}^{\mathbf{f}_i} = \sum_{\kappa \in \mathcal{K}} a_\kappa \prod_{i=1}^{r_\kappa} R_{x^{\kappa,i}}^{\mathbf{f}^{\kappa,i}}, 
\end{align}
as a sum of words in the variables $(R_x^{\mathbf{f}})$. Here, we have indexed the words by $\kappa$ in some indexing set $\mathcal{K}$ and the $\alpha_\kappa$ are integer coefficients. 

With \eqref{eq:wordy0} at hand, we are able to give our full description of the correlation structure of the probability measure $P_{\mathbf{m}}$:

\begin{thm} \label{thm:cffull0}
Suppose that the product $\prod_{i=1}^k T_{x_i}^{\mathbf{f}_i}$ has an expression $\prod_{i=1}^k T_{x_i}^{\mathbf{f}_i} = \sum_{\kappa \in \mathcal{K}} a_\kappa \prod_{i=1}^{r_\kappa} R_{x^{\kappa,i}}^{\mathbf{f}^{\kappa,i}}$. Then, for generic parameters $\alpha$, $\beta$, $\gamma$ and $\delta$, we may write
\begin{align} \label{eq:corrstructure}
    P_{\mathbf{m}}\left( \bigcap_{i=1}^k \{ \sigma(x^i) = x^i + \mathbf{f}^i \} \right) =  \sum_{\kappa \in \mathcal{K}} a_\kappa \sum_{\theta \in  \{0,1\}^2 } \lambda_{\theta,\mathbf{m}} \other{P}_{\theta,\mathbf{m}} \left[\prod_{i=1}^{r_\kappa} R_{x^{\kappa,i}}^{\mathbf{f}^{\kappa,i}}  \right],
\end{align}
where $\lambda_{\theta,\mathbf{m}} := C_{\theta,\mathbf{m}} \det(K_{\theta,\mathbf{m}})/Z_{\mathbf{m}}$ and the signed correlations are given by 
\begin{align} \label{eq:dform}
\other{P}_{\theta,\mathbf{m}} \left[\prod_{i=1}^r R_{x^i}^{\mathbf{f}^i} \right] = \prod_{i=1}^r K_{00,\mathbf{m}}(x^i,x^i+\mathbf{f}^i) \det_{i,j=1}^r \left[ H_{\theta,\mathbf{m}}( x^i+\mathbf{f}^i, x^j ) \right],
\end{align}
where $H_{\theta,\mathbf{m}}$ is given by
\begin{align} \label{eq:nform0}
H_{\theta,\mathbf{m}}(x,x') := \frac{1}{m_1m_2}\sum_{z^{m_1}=(-1)^{\theta_1},w^{m_2}=(-1)^{\theta_2}} \frac{ z^{-(x'_1-x_1)} w^{ - (x_2'-x_2) } }{ \alpha +\beta z + \gamma w + \delta w^{-1} }.
\end{align}

\end{thm}
We note that by \eqref{defK0}, the terms $K_{00,\mathbf{m}}(x^i,x^i+\mathbf{f}^i)$ occurring in \eqref{eq:dform} are equal to $\alpha,\beta,\gamma,\delta$ according to whether $\mathbf{f}^i = 0,\mathbf{e}^1,\mathbf{e}^2,-\mathbf{e}^2$ respectively.

By \eqref{eq:ev2b0}, the case of \eqref{eq:wordy0} is particularly simple when $\mathbf{f}^i = \mathbf{e}^1$ for each $i$. Namely, we simply have the relation $\prod_{i=1}^k T_{x_i}^{\mathbf{f}_i} = \prod_{i=1}^k R_{x_i}^{\mathbf{f}_i}$. This accounts for the simpler structure given in Theorem \ref{thm:cf}, which is merely this special case of Theorem \ref{thm:cffull0}.

We mention here that the technique at the heart of this article---that is, the representation of a desirable statistical mechanics model as a pushforward of a signed determinantal process---has been used in more recent work by the authors on the six-vertex model \cite{JS2}. 

\section{Applications of our main results} \label{sec:applications}

In this section we review several applications of Theorem \ref{thm:pf} and Theorem \ref{thm:cf} which we believe to be of interest. Let us highlight in particular the results Theorem \ref{thm:noncoll}, which gives an explicit determinantal formula for the correlation structure of $\ell$ asymmetric Poissonian walkers on $\mathbb{Z}_n$ conditioned to never collide, and Theorem \ref{thm:traffic}, which is a generalisation of a result appearing in \cite{johnston} relating the aforementioned model the asymmetric exclusion process on the ring.


\subsection{\texorpdfstring{The integrable snake model on the infinite cylinder $\mathbb{Z} \times \mathbb{Z}_n$}{The integrable snake model on the infinite cylinder Z x Zn}}
The first scaling limit we consider is fixing $m_2 = n$ and sending $m_1 \to \infty$ with all other parameters remaining fixed. Without loss of generality, we may scale the parameters such that $\alpha = 1$, since multiplying all the weights by a constant does not change the model. We find that in the asymptotic regime where $m_1 \to \infty$,
for given $\gamma,\delta$, the parameter $\beta$ controls the correlation structure of the process through a series of phase transitions, related to the number of $n^{\text{th}}$ roots of $\pm 1$ for which the modulus of $1 + \gamma w + \delta w^{-1}$ exceeds $\beta$. To describe this correlation structure, for integers $0 \leq \ell \leq n$, define the root sets
\begin{align}
\mathcal{L} &:= \text{The $\ell$ elements of $\{ w \in \mathbb{C} : w^n = (-1)^{n-\ell+1} \}$ with least real part}, \label{eq:rootset1}\\
\mathcal{R} &:= \text{The $n - \ell$ elements of $\{ w \in \mathbb{C} : w^n = (-1)^{n-\ell+1} \}$ with greatest real part}. \label{eq:rootset2}
\end{align}

We show the following.
\begin{thm}\label{thm:m1infinity}
    Let $\gamma,\delta \geq 0$ be parameters satisfying $4 \gamma \delta \leq 1$ and let $0 \leq \ell \leq n$ be an integer. Suppose also that the product $\prod_{i=1}^k T_{x_i}^{\mathbf{f}_i}$ has an expression $\prod_{i=1}^k T_{x_i}^{\mathbf{f}_i} = \sum_{\kappa \in \mathcal{K}} a_\kappa \prod_{i=1}^{r_\kappa} R_{x^{\kappa,i}}^{\mathbf{f}^{\kappa,i}}$. Then there is a probability measure $P_{\ell,n}^{\gamma, \delta}$ on pure snake configurations on $\mathbb{Z}\times \{0,\dots,n-1\}$ with correlations
    \begin{align}\label{eq:prem1}
    P_{\ell,n}^{\gamma, \delta} \left( \bigcap_{i=1}^k \{ \sigma(x^i) = x^i + \mathbf{f}^i \} \right) := \sum_{\kappa \in \mathcal{K}} a_\kappa \other{P}_{\ell,n}^{\gamma, \delta} \left[\prod_{i=1}^{r_\kappa} R_{x^{\kappa,i}}^{\mathbf{f}^{\kappa,i}}  \right],
    \end{align}
    where
      \begin{align} \label{eq:dform2b}
    \other{P}_{\ell,n}^{\gamma, \delta}  \left[\prod_{i=1}^r R_{x^i}^{\mathbf{f}^i} \right] = \det_{i,j=1}^r \left[ (-1)^{\indicpower{ \mathbf{f}^i = \mathbf{e}^1}} \gamma^{\indicpower{ \mathbf{f}^i = \mathbf{e}^2}} \delta^{\indicpower{ \mathbf{f}^i = -\mathbf{e}^2 }} 
    H_{\ell,n}^{\gamma, \delta}( x^i+\mathbf{f}^i, x^j ) \right],
    \end{align}
    and
    \begin{align} \label{eq:dform3bi}
        H_{\ell,n}^{\gamma, \delta}(x,y) = \indic{y_1 \geqslant x_1}\frac{1}{n} \sum_{w \in \mathcal{R}} \frac{w^{-(y_2 - x_2)}}{(1 + \gamma w + \delta w^{-1})^{y_1 - x_1 + 1}} - \indic{y_1 < x_1}\frac{1}{n} \sum_{w \in \mathcal{L}} \frac{w^{-(y_2 - x_2)}}{(1 + \gamma w + \delta w^{-1})^{y_1 - x_1 + 1}}.
    \end{align}
\end{thm}

Again, we have a simpler form when we let $\mathbf{f}^i = \mathbf{e}^1$ for each $i$. It holds still that $\prod_{i=1}^k T_{x_i}^{\mathbf{f}_i} = \prod_{i=1}^k R_{x_i}^{\mathbf{f}_i}$ and so we can write
\begin{align}\label{eq:simpledetformula}
    P_{\ell,n}^{\gamma, \delta}  \left( \bigcap_{i=1}^k \{ \sigma(x^i) = x^i + \mathbf{f}^i \} \right) = \det_{i,j=1}^r \left[ - 
    H_{\ell,n}^{\gamma, \delta}( x^i+\mathbf{f}^i, x^j ) \right].
\end{align}

We mention that the case $n = 2$ may be used to recover a determinantal point process appearing in Soshnikov \cite{soshnikov} and studied in further detail by Lyons and Steif \cite{LS}. Following \cite[Example 1.7]{LS}, consider a renewal process on $\mathbb{Z}$ where the gaps between consecutive renewals have the law of a negative binomial distribution with parameters $n=2$ and $p \in (0,1)$. That is, the interarrivals have the same law as the number of flips required to obtain two heads from a coin with probability $p$ of heads. It has been shown that the subset of $\mathbb{Z}$ consisting of arrivals is a determinantal point process with correlation kernel $K(x,y) = \frac{1-p}{1+p} p^{|y-x|}$ for $x,y \in \mathbb{Z}$. But also, from \eqref{eq:simpledetformula}, as one might expect, under $P_{1,2}^{\gamma,\delta}$, it can be shown that the collection of points $\{ x \in \mathbb{Z} \times \{0\} : \sigma(x) = x + \mathbf{e}^2 \}$ has the same distribution with parameter $p = \frac{\gamma+\delta}{1+\gamma + \delta}$. Thus $P_{\ell,n}^{\gamma,\delta}$ governs the law of a determinantal process that may be regarded as a (rather substantial) generalisation of the aforementioned determinantal process appearing in \cite{soshnikov} and \cite{LS}.

\subsection{\texorpdfstring{The integrable snake model on the infinite plane $\mathbb{Z}^2$}{The integrable snake model on the infinite cylinder Z2}}
We may now study a subsequent limit as $n \to \infty$ of the probability measure $P_{\ell, n}^{\gamma,\delta}$ to obtain a probability measure $P_\tau^{\gamma,\delta}$ on pure snake configurations on $\mathbb{Z}^2$; see Section~\ref{subsec:m1m2infinity}. The analogue of \eqref{eq:prem1} holds, in the following form.

\begin{thm}\label{thm:m1m2infinity}
    Given parameters $\tau \in [0,1]$ and $\gamma,\delta \geq 0$ satisfying $4 \gamma \delta \leq 1$, there is a probability measure $P^{\gamma,\delta}_\tau$ on pure snake configurations on $\mathbb{Z}^2$ with the following correlations.
   
    Suppose that the product $\prod_{i=1}^k T_{x_i}^{\mathbf{f}_i}$ has an expression $\prod_{i=1}^k T_{x_i}^{\mathbf{f}_i} = \sum_{\kappa \in \mathcal{K}} a_\kappa \prod_{i=1}^{r_\kappa} R_{x^{\kappa,i}}^{\mathbf{f}^{\kappa,i}}$. Then
    
    \begin{align*}
    P^{\gamma,\delta}_\tau \left( \bigcap_{i=1}^k \{ \sigma(x^i) = x^i + \mathbf{f}^i \} \right) := \sum_{\kappa \in \mathcal{K}} a_\kappa \other{P}^{\gamma,\delta}_\tau \left[\prod_{i=1}^{r_\kappa} R_{x^{\kappa,i}}^{\mathbf{f}^{\kappa,i}}  \right],
    \end{align*}
    where
      \begin{align} \label{eq:dform5}
    \other{P}^{\gamma,\delta}_\tau \left[\prod_{i=1}^r R_{x^i}^{\mathbf{f}^i} \right] = \det_{i,j=1}^r \left[ (-1)^{\indicpower{ \mathbf{f}^i = \mathbf{e}^1 }} \gamma^{\indicpower{ \mathbf{f}^i = \mathbf{e}^2 }} \delta^{\indicpower{ \mathbf{f}^i = -\mathbf{e}^2 }} H_\tau^{\gamma,\delta}( x^i+\mathbf{f}^i, x^j ) \right],
    \end{align}
    and 
    \begin{align} \label{eq:dform6}
      H_\tau^{\gamma,\delta}(x,y) = \int_{\overline{w_\tau} }^{w_\tau} \frac{\mathrm{d}w}{2\pi \iota w } \frac{ w^{ - (y_2-x_2)}}{ (1 + \gamma w + \delta w^{-1})^{ y_1 - x_1 + 1}},
    \end{align}
where $w_\tau = - e^{ - 2 \pi \iota \tau}$ and the contour travels to the right of the origin if $y_1 \geq x_1$ and to the left of the origin if $y_1 < x_1$.
\end{thm}

We note that, setting $\delta = 0$, we recover the extended sine kernel \eqref{eq:lozcorr}.
We mention that it is also possible to obtain Theorem \ref{thm:m1m2infinity} by simultaneously taking $m_1$ and $m_2$ to infinity in the setting of Theorem \ref{thm:cffull0}.

\subsection{\texorpdfstring{The continuous-time integrable snake model on $\mathbb{R} \times \mathbb{Z}_n$}{The continuous-time integrable snake model on RxZn}}
We will be most interested, however, in studying a different subsequent scaling limit of the probability measures $P_{\ell,n}^{\gamma,\delta}$. Namely, for fixed parameters {$T, T' \geq 0$} we study the asymptotics of $P_{\ell,n}^{\varepsilon T, \varepsilon T' }$ as $\varepsilon \downarrow 0$ and find that we can recover the law of a stochastic process $\mathbf{P}^{\ell,n}$ that is closely related to the asymmetric exclusion process on the ring.

To define this process, as above, let $\mathbb{Z}_n := \{0,1,\dots,n-1\}$ denote the ring with $n$ elements. We use the word \textit{ring} here to refer to a discrete circle with $n$ points, not an algebraic structure. Let $T,T' \geq 0$ and consider $\ell$ independent indistinguishable Poisson walkers on $\mathbb{Z}_n$ undergoing the following dynamics:
\begin{itemize}
\item A walker at $x$ jumps to $x+1$ at rate $T$.
\item A walker at $x$ jumps to $x-1$ at rate $T'$.
\end{itemize}
Here $x+1$ and $x-1$ are taken mod $n$.

Let $x = (x_1,\dots,x_\ell)$ denote a set of distinct starting locations and write $P_x$ for the law of the process starting from $x$. Further let $(\mathcal{F}_t)_{t\geq 0}$ be the natural filtration of
this process and $\mathcal{F}_{\infty} \coloneqq \sigma\left(\bigcup_{t \geq 0} \mathcal{F}_{t}\right)$ be the terminal $\sigma$-algebra for this filtration. Then consider the stopping time
\begin{align*}
\tau := \inf \{ t \geq 0 : \text{two walkers occupy the same location at time $t$} \}.
\end{align*}

We study the dynamics of this process conditioned on non-collision. Indeed, we define a second probability measure by defining, for $A \in \bigcup_{t \geq 0} \mathcal{F}_{t}$,
\begin{align*}
\mathbf{P}_x^{\ell,n}(A) := \lim_{t \uparrow \infty} P_x(A | \tau > t ),
\end{align*}
and by noting that we can extend this definition of $\mathbf{P}_x^{\ell,n}$ to events in $\mathcal{F}_\infty$ uniquely by Carath\'eodory's extension theorem. We call the measures $\{ \mathbf{P}_x^{\ell,n} : x \subseteq \mathbb{Z}_n , \# x = \ell \}$ the laws of $\ell$ walkers with starting positions $x$ and conditioned to never collide.

For a process $(X_t)_{t \geq 0}$ with law $\mathbf{P}_x^{\ell,n}$, we will be interested in studying the correlations of the following events. 
\begin{align*}
E(t,h,0) &:= \{ h \notin X_t \}\\
E(t,h,\mathbf{e}^1) &:= \{ h \in X_t\}\\
E(t,h,\pm \mathbf{e}^2) &:= \{ \text{A particle jumps from $h$ to $h \pm 1$ in $[t,t+\mathrm{d}t)$} \}.
\end{align*}

By using a cyclic version of the Karlin-McGregor formula, which was inspired by Liechty and Wang's \cite{LW} formula for the transition density of non-intersecting Brownian motion on the circle (the kernel of the idea also appears in Fulmek \cite{fulmek}), we are able to show that a certain scaling limit of the integrable snake model matches exactly with the Poissonian walkers on the ring conditioned never to collide. With this in hand, we prove the following result, which is a generalisation of works by Gordenko \cite{gordenko} and the first author \cite{johnston} in the special case $T'=0$, stating that $\mathbf{P}^{\ell,n}$ has an explicit determinantal structure:

\begin{thm}\label{thm:noncoll}
Let $\mathbf{P}^{\ell,n}$ denote the law of $\ell$ independent $(T,T')$ Poissonian walkers conditioned to never collide and started at stationarity. Then
\begin{align} \label{eq:light}
&\mathbf{P}^{\ell,n} \left( \bigcap_{i=1}^k E(t^i,h^i,\mathbf{f}^i) \right)\nonumber \\
&= \det_{i,j=1}^r \left[ (-1)^{\indicpower{ \mathbf{f}^i = \mathbf{e}^1 } }  T^{\indicpower{ \mathbf{f}^i = \mathbf{e}^2 }}{T'}^{\indicpower{ \mathbf{f}^i = -\mathbf{e}^2 }} 
H_{\ell,n}(\mathbf{f}^i, (s^i,h^i), (s^j,h^j) \right] \prod_{i : \mathbf{f}^i= \pm \mathbf{e}^2 } \mathrm{d}t^i,
\end{align}
where
\begin{align} \label{eq:dform3bii}
    &H_{\ell,n}(\mathbf{f},(s,h),(s',h'))
    \nonumber \\
    &=
    \begin{cases}
         \frac{1}{n} \sum_{w \in \mathcal{R}} w^{-(h' - h + \indicpower{\mathbf{f} = -\mathbf{e}^2} -\indicpower{\mathbf{f} = \mathbf{e}^2})} e^{ - (Tw + Tw^{-1})(s'-s)} \qquad &\text{if $s' \succeq s + \mathbf{f}$}\\
    -\frac{1}{n} \sum_{w \in \mathcal{L}} w^{-(h' - h + \indicpower{\mathbf{f} = -\mathbf{e}^2} -\indicpower{\mathbf{f} = \mathbf{e}^2})} e^{ - (Tw + Tw^{-1})(s'-s)}  &\text{if $s \prec s + \mathbf{f}$}.         
    \end{cases}
    \end{align}
    Here we have $\{ s' \succeq s + \mathbf{f}\} = \{ s' > s \} \cup \{ s' = s, \mathbf{f} \neq \mathbf{e}^1 \}$ and $\{ s' \prec s + \mathbf{f}\}$ is simply the complement of this event. The sets $\mathcal{R}$ and $\mathcal{L}$ are as in \eqref{eq:rootset1} and \eqref{eq:rootset2}.
\end{thm}
The case $T' = 0$ may be regarded as an analogue of a model studied in Metcalfe, O'Connell and Warren \cite{MOW} with the roles of time and space interchanged; see \cite{johnston} for further discussion of this connection.

For an $\ell$-tuple $x = ( x_1,\dots,x_\ell )$ of distinct elements of $\mathbb{Z}_n$, we define
\begin{align*}
\Delta(x) \coloneqq \prod_{1 \leq i < j \leq \ell} |e^{2 \pi \iota x_j/n} - e^{2 \pi \iota x_i/n} |.
\end{align*}
Using the previous result, we are able to prove in the sequel that the measure $\mathbf{P}_x^{\ell,n}$ governs the law of a continuous time Markov chain with stationary distribution given by
\begin{align*}
\lim_{t \to \infty} \mathbf{P}_x^{\ell,n}(X_t = y) = \Delta(y)^2/n^\ell,
\end{align*}
for any $\ell$-tuple $y = (y_1,\dots,y_\ell)$ of distinct elements in $\mathbb{Z}_n$.

We now define a third probability measure on the ring associated with the asymmetric exclusion process. Namely, under $\mathbf{P}^{\text{ASEP}, \ell,n}$ we have $\ell$ Markovian walkers on the ring who interact according to the following rules:
\begin{itemize}
\item If a walker is at a location $h \in \mathbb{Z}_n$ at time $t$ and $h+1$ is unoccupied, then the walker jumps to $h+1$ at rate $T$.
\item If a walker is at a location $h \in \mathbb{Z}_n$ at time $t$ and $h-1$ is unoccupied, then the walker jumps to $h-1$ at rate $T'$.
\end{itemize}
There are no other forms of jumps.

We call $\mathbf{P}^{\text{ASEP}, \ell,n}$ the asymmetric exclusion process (ASEP) on the ring. This a generalisation of the totally asymmetric exclusion process (TASEP) on the ring, which corresponds to the case $T'=0$. TASEP on the ring has been studied extensively by Baik and Liu \cite{BL, BL2, BL3, BL4}.

The final result of this section is a generalisation from TASEP to ASEP of the traffic representation of TASEP given in Theorem 1.12 of \cite{johnston}. Given a subset $E$ of $\mathbb{Z}_n$, we define the \textbf{traffic} of $E$ to be the number of pairs of adjacent elements of $E$, i.e.
\begin{align*}
\mathrm{Traffic}(E) := \# \{ h,h' \in E : h'-h \equiv 1 \text{ mod $n$} \}. 
\end{align*}
It turns out that ASEP is an exponential martingale change of measure of the asymmetric Poisson walkers conditioned to never collide:

\begin{thm}[Generalisation of Theorem 1.12 of \cite{johnston}]  \label{thm:traffic}
We have 
        \begin{equation*}
            \left.\frac{d\mathbf{P}^{\text{ASEP}, \ell,n}}{d\mathbf{P}^{\ell,n}}\right|_{\mathcal{F}_t} = \frac{\Delta(X_0)}{\Delta(X_t)}\exp\left((T + T')\int_{0}^t (\operatorname{Traffic}(X_s) -c_{\ell,n}) \; ds\right)
        \end{equation*}
        where $c_{\ell,n} = \ell - \sin(\pi \ell/n)/\sin(\pi/n)$.

\end{thm}

We close by mentioning that it is possible to take a subsequent scaling limit of Theorem~\ref{thm:noncoll} with $\ell$ and $n$ tending to infinity in such a way that $\ell/n \to \tau \in (0,1)$. We obtain a space-time stochastic process with infinitely many non-colliding walkers located at positions in $\mathbb{Z}$ with an average density $\tau$. Roughly speaking, these walkers jump up at rate $T$ and jump down at rate $T'$, but are thereafter `conditioned never to collide' in a suitable sense. This process may be regarded as a generalisation of Boutillier's bead process \cite{boutillier}: indeed, in the special case $T' = 0$ (which has no downward jumps) the space-time points at which the walkers jump up coincide in law with the beads in Boutillier's bead process. Theorem \ref{thm:traffic} says that the asymmetric exclusion process on $\mathbb{Z}$ may then be regarded as a change of measure of this process where the particles are encouraged to spend more time in traffic.

That completes the statements of our main results. In the next section, we develop our Kasteleyn theory for the integrable snake model, ultimately leading to proofs of our foundation results, Theorem \ref{thm:pf} and Theorem \ref{thm:cf}.

\section{Kasteleyn theory for the integrable snake model}\label{sec: fibonacci}

\subsection{Definitions} \label{sec:def}
We recall from the introduction that a generalised snake configuration on the semi-discrete torus $\torus{\mathbf{m}}:= \mathbb{Z}_{m_1} \times \mathbb{Z}_{m_2}$ is a bijection $\sigma: \torus{\mathbf{m}} \to \torus{\mathbf{m}}$ with the properties that, for all $x \in \torus{\mathbf{m}}$,
\begin{align*}
  \sigma(x) \in \{x, x+\mathbf{e}^1, x + \mathbf{e}^2, x - \mathbf{e}^2\}.
\end{align*}
In the case that $m_2 = 2$, $\mathbf{e}^2 = -\mathbf{e}^2$. So in this case, we will consider down jumps not to exist.
To circumvent technicalities, we will assume that $m_2 \geq 3$. 

A pure snake configuration is a generalised snake configuration with no two-cycles. $\URDbar_\mathbf{m}$ represents the set of generalised snake configurations on $\torus{\mathbf{m}}$, while $\URD_{\mathbf{m}}$ represents the set of pure snake configurations. Since primarily we will have $m_1$ and $m_2$ for the dimensions of our discrete torus, we will often relax the subscript $\mathbf{m}$ when this is the case. There is a natural projection $\mathrm{sh}\colon \URDbar \to \URD$ from the set of generalised snake configurations to the set of pure snake configurations obtained by deleting all two-cycles. 

For a pure snake configuration $\sigma \in \URD_{\mathbf{m}}$, define the set of possible snakelets compatible with $\sigma$ by
 \begin{align*}
  S_\sigma := \{(x,x+\mathbf{e}^2) : (x,x+\mathbf{e}^2) \text{ is a cycle disjoint to } \sigma\}.
\end{align*}
Equivalently, $S_\sigma$ is the set of pairs $(x,x+\mathbf{e}^2)$ in $\torus{\mathbf{m}}$ such that $\sigma(x) = x$ and $\sigma(x+\mathbf{e}^2) = x+\mathbf{e}^2$.
Define the set of nests of $\sigma$ as the set of possible collections of snakelets, i.e.
\begin{align*}
  N_\sigma := \{N \in \mathcal{P}(S_\sigma): \text{all 2-cycles in } N \text{ are pairwise disjoint} \},
\end{align*}
where $\mathcal{P}(S_\sigma)$ is the collection of subsets of $S_\sigma$. Figure \ref{fig: nest} demonstrates one possible nest.

\begin{rem} \label{rem:corr}
There is a one-to-one correspondence between generalised snake configurations $\other{\sigma}$ and pairs $(\sigma,N)$, where $\sigma \in \URD$ is a pure snake configuration and $N \in N_\sigma$ is a nest of 
$\sigma$. 
\end{rem}
Rephrasing this correspondence, we have a natural bijection between $\mathrm{sh}^{-1}(\sigma)$ and $\{ (\sigma,N) : N \in N_\sigma \}$. We use this correspondence freely through the remainder of the article. 

\begin{figure}[h!t]
  \includegraphics{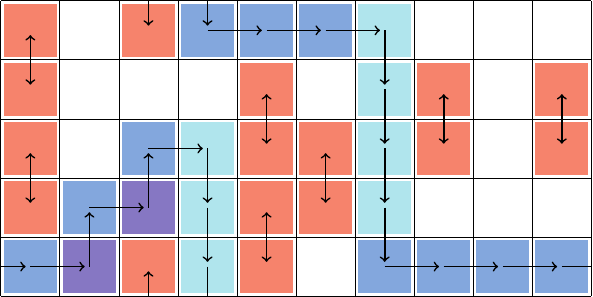}
  \caption{The diagram illustrates one possible nest, $N$. The underlying snake configuration $\sigma$ is marked by the blue points and the snakelets are marked with red points.}
  \label{fig: nest}
\end{figure}

\subsection{Proof of Lemma \ref{lem:shapepush} and non-negativity of the Fibonacci factor} \label{sec:nonneg}

Recall from \eqref{eq:weightsnake} that the weight of a generalized snake configuration is given by \[{w}(\other{\sigma}) = (-1)^{N(\other{\sigma})}\alpha^{\fixedbox(\other{\sigma})}\beta^{\rightbox(\other{\sigma})}\gamma^{\upbox(\other{\sigma})}\delta^{\downbox(\other{\sigma})},\]
where $N(\other{\sigma})$ is defined in \eqref{eq:ns}.
In this section, we prove Lemma \ref{lem:shapepush}, which relates the weights of generalised snake configurations and their underlying pure snake configuration in terms of the Fibonacci weighting. We begin by giving a more precise definition of this weighting.

For $\sigma$ a pure snake configuration, consider the equivalence relation on $\{x \in \torus{\mathbf{m}}:\sigma(x) = x\}$, the set of points in $\torus{\mathbf{m}}$ fixed by $\sigma$, given by $x \relsigma y$ if and only if there exists $k \in \mathbb{Z}_{\geqslant 0}$ such that $y = x + k \mathbf{e}^2$ and $x + i \mathbf{e}^2$ is fixed by $\sigma$ for all $0 \leqslant i \leqslant k$, or $y = x - k \mathbf{e}^2$ and $x - i \mathbf{e}^2$ is fixed by $\sigma$ for all $0 \leqslant i \leqslant k$. We call the equivalence classes under this relation \textbf{vertical gaps}, shown pictorially in Figure~\ref{fig: vertgaps}.

\begin{figure}[h!t]
    \includegraphics{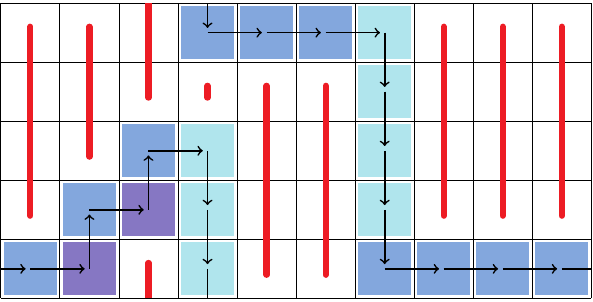} 
    \caption{The red lines join the points in the same equivalence class, i.e. they represent the vertical gaps.}
    \label{fig: vertgaps}
\end{figure}

For a pure snake configuration $\sigma$, we will here define the Fibonacci factor $J$ by the sum 
\begin{align*}
J(\sigma) := \sum_{N \in N_{\sigma}}(- \gamma\delta/\alpha^2)^{|N|},
\end{align*}
so proving Lemma~\ref{lem:shapepush} amounts to showing that \eqref{eq:J1} is an equivalent definition. The following lemma tells us that we can recast $J(\sigma)$ in terms of a product over the vertical gaps of $\sigma$ -- c.f.\ \eqref{eq:J1} and the preceding discussion.

\begin{lemma}\label{lem: Jformula}
  Let $G_\sigma$ be the set of all vertical gaps for a given pure snake configuration $\sigma$. Then 
  \begin{equation*}
    J(\sigma) = \prod_{g \in G_\sigma} \tilde{f}_{|g|}\left(-\frac{\gamma \delta}{\alpha^2}\right),
  \end{equation*}
  where
  \begin{equation*}
    \tilde{f}_n(\lambda) := \sum_{0 \leqslant k \leqslant n/2} \binom{n-k}{k}\lambda^k.
  \end{equation*}  
\end{lemma}

\begin{proof}
  Since placing snakelets in a vertical gap does not affect the placement of snakelets in other vertical gaps, counting snakelets can be done independently in each vertical gap. The number of ways of placing snakelets in a vertical gap of size $n$ is the same, wherever the vertical gap is, so we can denote $\tilde{f}_n(\lambda)$ as the polynomial such that the coefficient of $\lambda^k$ counts the number of ways we can place $k$ snakelets in a vertical gap of size $n$. So the total number of ways of adding $k$ snakelets to $\sigma$ is given by the coefficient of $\lambda^k$ in $\prod_{g \in G_\sigma} \tilde{f}_{|g|}(\lambda)$.

  To count the number of ways of placing $k$ snakelets in a vertical gap of size $n$, we can reframe the problem as ordering $k$ snakelets and $n-2k$ points, of which there are $\binom{n-k}{k}$ ways. This gives the form of $\tilde{f}_n$ as required.
\end{proof}

We are now ready to give a proof of Lemma \ref{lem:shapepush}.

\begin{proof}[Proof of Lemma \ref{lem:shapepush}]
In light of Lemma \ref{lem: Jformula} and \eqref{eq:creat}, it is sufficient to establish $\tilde{f}_n(\lambda) = f_n(\lambda)$ for all $n \geq 1$, where 
\begin{align*}
\tilde{f}_n(\lambda) &:= \sum_{0 \leq k \leq n/2} \binom{n-k}{k} \lambda^k\\
f_n(\lambda) &:= \frac{1}{\sqrt{1+4\lambda}} \left\{ \left( \frac{1 + \sqrt{1+4\lambda}}{2} \right)^{n+1} - \left( \frac{1 - \sqrt{1+4\lambda}}{2} \right)^{n+1}  \right\}.
\end{align*}
A brief calculation tells us that we have $f_n(\lambda) = \tilde{f}_n(\lambda) = 1$ for both $n = 0$ and $ n =1$. As for larger values of $n$, one may verify quickly that both $\tilde{f}_n(\lambda)$ and $f_n(\lambda)$ satisfy the recurrence relation
\begin{align*}
h_n(\lambda) = h_{n-1}(\lambda) + \lambda h_{n-2}(\lambda).
\end{align*}
It follows that $f_n(\lambda)=\tilde{f}_n(\lambda)$ for all $n \geq 0$.
\end{proof}

To close this subsection, recall that we say that $\alpha, \beta, \gamma, \delta$ lie in the probabilistic regime when
\begin{equation}\label{eq:probregime2}
    \alpha^2 - 4\gamma \delta \geq 0
\end{equation}
Under this regime, $\tilde{f}_n\left(-\frac{\gamma \delta}{\alpha^2}\right) = f_n\left(-\frac{\gamma \delta}{\alpha^2}\right)$ is positive for all $n$, and so by Lemma~\ref{lem: Jformula}, we have confirmed Remark~\ref{rem:positivity} and that \eqref{eq:probregime2} is a sufficient condition for $P_\mathbf{m}$ (see Equation~\eqref{eq:Pdef}) to be a probability measure.

\subsection{Kasteleyn theory for snake configurations}

Recall from \eqref{eq:weightsnake} that we defined the weight of a generalised snake configuration $\other{\sigma}:\mathbb{T}_\mathbf{m} \to \mathbb{T}_\mathbf{m}$ by 
\begin{align*}
{w}(\other{\sigma}) := {w}_{\alpha,\beta,\gamma,\delta}(\other{\sigma}) :=  (-1)^{N(\other{\sigma})}\alpha^{\fixedbox(\other{\sigma})}\beta^{\rightbox(\other{\sigma})}\gamma^{\upbox(\other{\sigma})}\delta^{\downbox(\other{\sigma})}.
\end{align*}
As observed previously in \eqref{eq:creat}, if $\other{\sigma} = (\sigma,N)$, where $\sigma$ is a pure snake configuration and $N$ is a nest of $\sigma$, then ${w}(\sigma,N) = (- \delta \gamma/\alpha^2)^{|N|} {w}(\sigma,\varnothing)$. Hence, we can define the \textbf{Fibonacci weight} of a pure snake configuration as
\begin{align} \label{eq:reweight}
  w(\sigma)J(\sigma) = {w}(\sigma,\varnothing)  \sum_{N \in N_{\sigma}}(- \delta \gamma/\alpha^2)^{|N|} = \sum_{N \in N_{\sigma}} {w}((\sigma, N)).
\end{align}
Note that, by virtue of Remark \ref{rem:corr} together with \eqref{eq:reweight} we have
\begin{align} \label{eq:pf0}
Z(\alpha,\beta,\gamma,\delta) := \sum_{ \sigma \in \URD_{\mathbf{m}}} w(\sigma)J(\sigma) = \sum_{ \other{\sigma} \in \URDbar_\mathbf{m} } {w}(\other{\sigma}).
\end{align}
We would now like to connect this quantity to the determinants of the operators $K_{\theta,\mathbf{m}}:\mathbb{T}_{\mathbf{m}} \times \mathbb{T}_{\mathbf{m}} \to \mathbb{C}$ defined in \eqref{defK0} for $\theta \in \{0,1\}^2$ and given by 
\begin{align}\label{defK01}
 K_{\theta, \mathbf{m}}(x,y) := \alpha \indic{y = x} + \beta e^{\pi \iota \frac{\theta_1}{m_1}} \indic{y = x + \mathbf{e}^1} + \gamma e^{\pi \iota \frac{\theta_2}{m_2}} \indic{y = x + \mathbf{e}^2} + \delta e^{-\pi \iota \frac{\theta_2}{m_2}} \indic{y = x - \mathbf{e}^2}.
\end{align}
Note that we may alternatively write 
\begin{align}\label{eq:weightsnakeK}
{w}(\other{\sigma}) =(-1)^{N(\other{\sigma})} \prod_{x \in \mathbb{T}_\mathbf{m}} K_{00,\mathbf{m}}(x ,\os(x) ) 
\end{align} 
In particular, the partition function $Z(\alpha,\beta,\gamma,\delta)$ resembles the determinant of $K_{00,\mathbf{m}}$ in that 
\begin{align*}
\det(K_{00,\mathbf{m}}) &= \sum_{ \os\colon \mathbb{T}_{\mathbf{m}} \to \mathbb{T}_{\mathbf{m}}} \mathrm{sgn}(\os)\prod_{x \in \mathbb{T}_m} K_{00,\mathbf{m}}(x ,\os(x) ) 
\\
&= \sum_{ \other{\sigma} \in \URDbar_\mathbf{m} } \mathrm{sgn}(\os) (-1)^{N(\os)} {w}(\os)\\
&= \sum_{ \other{\sigma} \in \URDbar_\mathbf{m} } \mathrm{sgn}(\operatorname{sh}(\os)) {w}(\os);
\end{align*}
c.f.\ \eqref{eq:pf0}. 
This sets out the task we undertake in the following sections: to study how the sign of a pure snake configuration on $\mathbb{T}_{\mathbf{m}}$ is related to its topological properties. Namely, its entries and exits on the boundaries of the torus.

\subsection{Snake configurations and torus links}\label{sec:knots}

In this section, we investigate how the topological properties of pure snake configurations are related to their sign structure as permutations. More specifically, consider a pure snake configuration $\sigma:\mathbb{T}_\mathbf{m} \to \mathbb{T}_\mathbf{m}$. We can decompose $\sigma$ into disjoint cycles $\sigma = C_1\cdots C_r$ of length at least $3$. Each cycle $C_i$ is itself a pure snake configuration.

Take a cyclic pure snake configuration $C$. Then we can write $C = (x_1,\dots,x_r)$ where for each $1 \leq i \leq r$, $x_{i+1} \in \{ x_i + \mathbf{e}^1, x_i + \mathbf{e}^2, x_i - \mathbf{e}^2 \}$, where we are taking $x_{r+1} := x_1$. It is easy to check that there exist integers $q_1,q_2$, with $q_1 \geq 0$ such that
\begin{align*}
\rightbox ( C) = q_1(C) m_1 \quad \text{and} \quad \upbox(C) - \downbox(C) = q_2(C) m_2.
\end{align*}
We call the pair $(q_1(C),q_2(C))$ the winding number of the cycle $C$.

Recall that we say two integers are coprime if they share no common prime divisor. By convention, zero is divisible by every prime number, thus zero is coprime to only $+1$ and $-1$. 

The main result of this subsection is the following:

\begin{proposition} \label{prop:decomp}
    Let $\sigma = C_1\cdots C_r$ be a decomposition of a pure snake configuration into disjoint cycles. Then there exist coprime integers $(q_1,q_2)$ with $q_1 \geq 0$ such that $q_i(C_j) = q_i$ for all $1 \leq j \leq r$.
\end{proposition}

To prove this result, we begin by recalling some basic notions from knot theory \cite{CF}. A \textbf{knot} is an embedding of the unit circle $S^1$ into three-dimensional Euclidean space $\mathbb{R}^3$. Two knots are \textbf{equivalent} if there is a homeomorphism of $\mathbb{R}^3$ mapping one to the other. A \textbf{torus knot} is a knot that is a subset of the torus in $\mathbb{R}^3$. A \textbf{link} (resp.\ \textbf{torus link}) is a collection of disjoint knots (resp.\ torus knots). 

There is a natural association between the torus $T$ embedded in $\mathbb{R}^3$ and the unit square $[0,1]^2$. With this, each torus knot can be naturally associated with a torically continuous function $\phi:[0,1] \to [0,1]^2$ satisfying $\phi(0)=\phi(1)$, but injective when restricted to $[0,1)$. We define the \textbf{winding number} of the knot $K$ to be the integers $(q_1,q_2)$ given by 
\begin{align*}
q_1 &= \text{Number of exits rightwards out of $[0,1]^2$} - \text{Number of exits leftwards out of $[0,1]^2$}\\
q_2 &= \text{Number of exits upwards out of $[0,1]^2$} - \text{Number of exits downwards out of $[0,1]^2$}.
\end{align*}
Knots with winding number $(0,0)$ are equivalent, and we call any such knot \textbf{trivial}. All other knots, we call \textbf{non-trivial}.

Given two different torically continuous functions from $[0,1]$ to $[0,1]^2$ that are injective on $[0,1)$, if they have the same image, their winding numbers $(q_1,q_2)$ and $(q_1',q_2')$ must satisfy $(q_1',q_2') := \pm (q_1,q_2)$. Thus, we will always suppose without loss of generality that $q_1 \geq 0$ and that $q_2 \geq 0$ if $q_1 = 0$.

We give a sketch proof of two well-known lemmas for torus knots and links. For further information, the reader might consult Chapter 7 of Murasugi \cite{murasugi}. 

\begin{lemma} \label{lem:coprime}
If $K$ is a non-trivial torus knot, its winding number takes the form $(q_1,q_2)$ where $q_1 \geq 0$ and $(q_1,q_2)$ are coprime.
\end{lemma}

\begin{proof}[Sketch proof]
Recall that two knots $K$ and $K'$ are said to be equivalent if $K'$ is equal to $K$ under a homeomorphism of $\mathbb{R}^3$. Regarding a knot $K$ as a subset of the torus $[0,1]^2$, it is possible to show that every non-trivial knot is equivalent to a knot $K'$ which is the image of a function $\phi:[0,1] \to [0,1]^2$ taking the form 
\begin{align*}
\phi(s) := (q_1s, q_2s) \quad \text{ mod}  \quad [0,1]^2
\end{align*}
where $(q_1,q_2)$ is the winding number of the knot $K$. Note that if $\phi(s) = \phi(s')$ for some distinct $0 \leq s < s' < 1$, there for some integers $k_1$ and $k_2$ we have 
\begin{align*}
q_1(s'-s) = k_1 \qquad \text{and}\qquad  q_2(s'-s) = k_2.
\end{align*}
This equation is solvable with some $s'-s \neq 0$ if and only if $q_1$ and $q_2$ share a common prime factor. Thus, the injectivity of $\phi$ on $[0,1)$ is equivalent to $(q_1,q_2)$ being coprime.
\end{proof}

\begin{lemma} \label{lem:link}
Let $K$ be a torus link consisting of disjoint non-trivial knots $K_1,\dots,K_r$. Then the winding numbers of the distinct knots $K_1,\dots,K_r$ are the same.
\end{lemma}

\begin{proof}[Sketch proof]
There is a homeomorphism of the torus such that the distinct knots $K_1,\dots,K_r$ are equivalent to the images $K_i$ of functions $\phi_i:[0,1] \to [0,1]^2$ where 
\begin{align*}
\phi_i(s) = (q^i_1s,q^i_2s) + x^i \quad \mathrm{mod} \quad [0,1]^2.
\end{align*}
for some $x^1,\dots,x^r$ in $[0,1]^2$. We have written $(q^i_1,q^i_2)$ for the winding numbers of $K^i$.

In order for the images of the distinct functions $\phi_1,\dots,\phi_r$ not to intersect, it must be the case, for one, that $x^i \neq (q^j_1s,q^j_2  s) + x^j$ for all $j \neq i$, $s \in [0,1)$, and for seconds that $q^1_1 = \cdots = q^r_1$ and $q^r_2 = \cdots = q^r_2$. In particular, the winding numbers of $K_1,\dots,K_r$ coincide, completing the proof.
\end{proof}

We are now equipped to complete the proof of Proposition \ref{prop:decomp}. 
\begin{proof}[Proof of Proposition \ref{prop:decomp}]
If $\sigma = C_1\cdots C_r$ is a decomposition of a pure snake configuration into disjoint cycles, we can associate $\sigma$ with a torus link with $r$ non-trivial component knots. By Lemma~\ref{lem:link}, these knots have the same winding number. By Lemma \ref{lem:coprime}, this winding number takes the form $(q_1,q_2)$ for $(q_1,q_2)$ coprime.
\end{proof}

\subsection{The partition function as a sum of determinants} \label{sec:kast}

From this section onwards, we will use the notation $\theta = (\theta_1,\theta_2)$ for an arbitrary element of $\{0,1\}^2$. On occasion, for fixed $\theta_1 \in \{0,1\}$, we will write $\sum_{\theta_2 \in \{0,1\}} f(\theta_1,\theta_2)$ to denote $f(\theta_1,0) + f(\theta_1,1)$, and likewise with the roles of $\theta_1$ and $\theta_2$ swapped. 

\begin{lemma}\label{lemmasignpermutations}
For $\os \in \URDbar_{\mathbf{m}}$ and $\theta_2 \in \{0,1\}$ we have
  \begin{align} \label{eq:signrep0}
    \sum_{\theta_1 \in \{0,1\}} C_{\theta,\mathbf{m}}\exp\left(\pi \iota \left[\frac{\theta_1 \rightbox(\os)}{m_1} + \frac{\theta_2 (\upbox(\os) - \downbox(\os))}{m_2} \right]\right) = \indic{\theta_2+m_2+1+rq_1\equiv 0}(-1)^{N(\os)}\mathrm{sgn}(\os),
  \end{align}
where $C_{\theta,\mathbf{m}} \coloneqq \frac{1}{2}(-1)^{(\theta_1 + m_1 + 1)(\theta_2 + m_2 + 1)}$, $q_1$ and $q_2$ are the coprime winding numbers for the individual cycles of the pure snake configuration $\sigma = \mathrm{sh}(\os)$ as in Proposition~\ref{prop:decomp} and the indicator function denotes the event that $\theta_2+m_2+1+rq_1\equiv 0 \mod 2$, where $rq_1 =\rightbox(\os)/{m_1}$ is the horizontal winding number of the underlying pure snake configuration. In particular, taking the sum over all $\theta \in \{0,1\}^2$, we have
  \begin{align} \label{eq:signrep}
    \sum_{\theta \in \{0,1\}^2} C_{\theta,\mathbf{m}}\exp\left(\pi \iota \left[\frac{\theta_1 \rightbox(\os)}{m_1} + \frac{\theta_2 (\upbox(\os) - \downbox(\os))}{m_2} \right]\right) =(-1)^{N(\os)}\mathrm{sgn}(\os).
  \end{align}

\end{lemma}
\begin{proof}
Note that if $\mathrm{sh}(\os) = \sigma$, then $\upbox(\sigma) - \downbox(\sigma) = \upbox(\os) - \downbox(\os)$, $\rightbox(\sigma) = \rightbox(\os)$ and $\mathrm{sgn}(\os) = (-1)^{N(\os)} \mathrm{sgn}(\sigma)$. Thus, in order to prove the result, it is sufficient to consider the case when $\sigma$ is pure, i.e.\ $N(\sigma) = 0$.

    Assuming $\sigma$ is a pure snake configuration, $\sigma$ has a decomposition into disjoint cycles, $C_1\cdots C_r$. The number of elements of a cycle $C_i$ is given by $\rightbox(C_i) + \upbox(C_i) + \downbox(C_i)$. In particular, since the sign of a permutation coincides with $(-1)$ to the number of even cycles, it follows that
\begin{align} \label{eq:blox}
\mathrm{sgn}(\sigma) &= (-1)^{ \sum_{i=1}^r \left( \rightbox(C_i) + \upbox(C_i) + \downbox(C_i) + 1 \right) } \nonumber \\
&=(-1)^{ \sum_{i=1}^r \left( \rightbox(C_i) + \upbox(C_i) - \downbox(C_i) + 1 \right) } \nonumber \\
&= (-1)^{ \sum_{i=1}^r \left( m_1 q_1(C) + m_2 q_2(C) + 1  \right) } \nonumber \\
&= (-1)^{ r m_1 q_1 + r m_2 q_2 + r}.
\end{align}

Conversely, note that $rm_1q_1 = \rightbox(\sigma)$ and $rm_2q_2 = \upbox(\sigma) - \downbox(\sigma)$. In particular, it follows that
\begin{align} \label{eq:orange}
   W_\sigma &:=  \sum_{\theta_1 \in \{0,1\}} C_{\theta,\mathbf{m}}\exp\left(\pi \iota \left[\frac{\theta_1 \rightbox(\os)}{m_1} + \frac{\theta_2 (\upbox(\os) - \downbox(\os))}{m_2} \right]\right) \nonumber\\
    &= \frac{1}{2}\sum_{\theta_1 \in \{0,1\}} (-1)^{(\theta_1+m_1+1)(\theta_2+m_2+1)} (-1)^{\theta_1 rq_1 + \theta_2 r q_2 }.
    \end{align}
For any integer $k$ we have the simple identity $\frac{1}{2}\sum_{\theta_1\in \{0,1\}} (-1)^{\theta_1 k} = \indic{k \equiv 0 \text{ mod } 2}$. Using this fact in \eqref{eq:orange} we have
\begin{align} \label{eq:orange2}
   W_\sigma = \indic{\theta_2+m_2+1+rq_1 \equiv 0} (-1)^{(m_1+1)(\theta_2+m_2+1) + \theta_2 rq_2} . 
    \end{align}
Eliminating $\theta_2$ in the exponent on the event $\{ \theta_2+m_2+1+rq_1 \equiv 0 \}$ we have
\begin{align} \label{eq:orange3}
   W_\sigma = \indic{\theta_2+m_2+1+rq_1 \equiv 0} (-1)^{(m_1+1)rq_1 + rq_2(m_2+1+rq_1)} . 
    \end{align}
We now study the power of $-1$ in \eqref{eq:orange3}. Using the fact that $r^2 = r \mod{2}$ we have
\begin{align} \label{eq:orange4}
(m_1+1)rq_1 + rq_2(m_2+1+rq_1) \equiv r ( (m_1+1)q_1 + (m_2+1)q_2 + q_1q_2) \mod{2}.
\end{align}
Now $(q_1,q_2)$ are coprime. It follows that at least one of $q_1$ or $q_2$ is odd and in particular $(q_1+1)(q_2+1) \equiv 0$ mod $2$. Thus $q_1 + q_2 + q_1q_2 \equiv 1$ mod $2$ and using this fact in \eqref{eq:orange4} we obtain
\begin{align} \label{eq:orange5}
(m_1+1)rq_1 + rq_2(m_2+1+rq_1) \equiv r ( m_1q_1 + m_2 q_2 + 1) \mod{2}.
\end{align}
Using \eqref{eq:orange5} in \eqref{eq:orange3} and comparing with \eqref{eq:blox}, we complete the proof of \eqref{eq:signrep0}.

The equation \eqref{eq:signrep} follows immediately from summing \eqref{eq:signrep0} over $\theta_2 \in \{0,1\}$. 
\end{proof}

We now begin working towards our proof of Theorem \ref{thm:pf}, whilst making some useful definitions.  Recall that we have the operator
\begin{align} \label{eq:Ktheta}
 K_{\theta,\mathbf{m}}(x,y) := \alpha \indic{y = x} + \beta e^{\pi \iota \frac{\theta_1}{m_1}} \indic{y = x + \mathbf{e}^1} + \gamma e^{\pi \iota \frac{\theta_2}{m_2}} \indic{y = x + \mathbf{e}^2} + \delta e^{-\pi \iota \frac{\theta_2}{m_2}} \indic{y = x - \mathbf{e}^2}.
\end{align}

Let us define the function  $\other{Q}_{\theta,\mathbf{m}}:\URDbar_{\mathbf{m}} \to \mathbb{R}$ by setting
\begin{align} \label{eq:Qdef}
\other{Q}_{\theta,\mathbf{m}}(\os) := \mathrm{sgn}(\other{\sigma}) \prod_{x \in \mathbb{T}_{\mathbf{m}}} K_{\theta,\mathbf{m}}(x,\os(x)).
\end{align}
Clearly $\other{Q}_{\theta,\mathbf{m}}$ is constructed so that
\begin{align*}
\det(K_{\theta,\mathbf{m}}) = \sum_{ \other{\sigma} \in \URDbar_{\mathbf{m}} } \other{Q}_{\theta,\mathbf{m}}( \other{\sigma}).
\end{align*}

As in the proof of Lemma \ref{lemmasignpermutations} we have
\begin{align*}
\rightbox(\other{\sigma})/m_1 = r q_1 \qquad \text{and} \qquad (\upbox(\os) - \downbox(\os))/m_2 = rq_2,
\end{align*}
where $r$ is the number of long cycles of $\os$ and $(q_1,q_2)$ is the winding number of any long cycle. Since $\mathrm{sgn}(\os) = \mathrm{sgn}(\sigma)(-1)^{N(\os)}$ and recalling \eqref{eq:weightsnakeK}, we may rewrite \eqref{eq:Qdef} as 
\begin{align} \label{eq:Qdef2}
\other{Q}_{\theta,\mathbf{m}}(\os) := \mathrm{sgn}(\sigma)(-1)^{N(\os) + \theta_1 rq_1 + \theta_2 rq_2}  \prod_{x \in \mathbb{T}_{\mathbf{m}}} K_{00,\mathbf{m}}(x,\os(x)) = \mathrm{sgn}(\sigma)(-1)^{\theta_1 rq_1 + \theta_2 rq_2} {w}(\os).
\end{align}

Note that $q_1$ and $q_2$ are always integer, thus $\other{Q}_{\theta,\mathbf{m}}(\other{\sigma}) $ is always a real number. More importantly, we note by \eqref{eq:signrep0} that for every $\os$ in $\URDbar_{\mathbf{m}}$ and each fixed $\theta_2 \in \{0,1\}$ we have
\begin{align} \label{eq:weightsum00}
\other{\mathcal{Q}}_{\theta_2,\mathbf{m}}(\os) := 
\sum_{\theta_1 \in \{0,1\}} C_{\theta,\mathbf{m}} \other{Q}_{\theta,\mathbf{m}}( \other{\sigma}) = \indic{\theta_2 + m_2 + 1 + rq_1 \equiv 0}{w}(\os),
\end{align}
where $rq_1 = \rightbox(\os)/m_1$ is the horizontal winding number of $\os$. Summing \eqref{eq:weightsum00} over $\theta_2$ we have
\begin{align} \label{eq:weightsum0}
\sum_{\theta \in \{0,1\}^2} C_{\theta,\mathbf{m}} \other{Q}_{\theta,\mathbf{m}}( \other{\sigma}) = {w}(\os) 
\end{align}

Now let us make the following definition. We say the \textbf{occupation number} of a generalised snake configuration is given by 
\begin{align}\label{eq:occupationnumberdef}
L(\os) := rq_1.
\end{align}
It is easily verified that for any $x_1 \in \{0,\dots,m_1-1\}$ we have
\begin{align*}
L(\os) := \# \{ 0 \leq x_2 \leq m_2 -1 : \sigma(x_1,x_2) = (x_1+1,x_2) \}.
\end{align*}
In other words, $L(\os)$ is the number of right moves in each vertical cross-section. Note that the occupation number of $\os$ is the same as that of $\mathrm{sh}(\os)$. 

Thus, the equation \eqref{eq:weightsum00} reads $\other{\mathcal{Q}}_{\theta_2,\mathbf{m}}(\os) =  \indic{ L(\os) \equiv \theta_2 + m_2 + 1  }{w}(\os)$. By summing over all $\other{\sigma}$ in $\URDbar_{\mathbf{m}}$ we then have
\begin{align} \label{eq:prefab}
\mathcal{Z}_{\theta_2,\mathbf{m}}(\alpha,\beta,\gamma,\delta) :=\sum_{ \other{\sigma} \in \URDbar_{\mathbf{m}} } \indic{ L(\os) \equiv \theta_2 + m_2 + 1 } {w}(\other{\sigma}) = \sum_{ \theta_1 \in \{0,1\} } C_{\theta,\mathbf{m}} \det(K_{\theta,\mathbf{m}}).
\end{align}
Summing \eqref{eq:prefab} over $\theta_2 \in \{0,1\}$, we then obtain
\begin{align} \label{eq:pf1b}
Z_{\mathbf{m}}(\alpha,\beta,\gamma,\delta) := \sum_{ \other{\sigma} \in \URDbar_{\mathbf{m}} } {w}(\other{\sigma}) = \sum_{ \theta \in \{0,1\}^2 } C_{\theta,\mathbf{m}} \det(K_{\theta,\mathbf{m}}),
\end{align}
which amounts to one half of the statement of Theorem \ref{thm:pf}. To complete the proof of Theorem \ref{thm:pf}, we need to evaluate the determinants in question.
Our main task here is to understand the eigenvalues of the operator $K_{\theta,\mathbf{m}}$ when considered as a bilinear operator on the vector space, $\mathbb{C}\torus{\mathbf{m}} := \{f\colon\torus{\mathbf{m}} \to \mathbb{C}\}$ and for $\phi: \torus{\mathbf{m}} \to \mathbb{C}$, we can define $(K_{\theta,\mathbf{m}} \phi)(x) := \sum_{y \in \torus{\mathbf{m}}} K_{\theta,\mathbf{m}}(x,y)\phi(y)$, the corresponding `matrix multiplication'. 

The eigensystem is described by the following lemma:

\begin{lemma} \label{lem:eigensystem}
The operator $K_{\theta,\mathbf{m}}$ has a complete set of eigenvalues and corresponding eigenfunctions indexed by the $m_1m_2$ pairs $\{ (z,w) : z^{m_1} = 1, w^{m_2} = 1\}$. The eigenvalues are given by 
  \begin{equation}\label{eq: defeigenvalues}
    \lambda_{z,w} = \alpha + \beta z_\theta + \gamma w_\theta + \delta w_{\theta}^{-1},
  \end{equation}
  where $z_\theta := ze^{\pi \iota \theta_1/m_1}$ and $w_\theta := we^{\pi \iota \theta_2/m_2}$. The corresponding eigenfunctions are given by 
  \begin{equation}\label{eq: defeigenfns}
    \phi_{z,w}(x) = \frac{1}{\sqrt{m_1 m_2}} z^{x_1} w^{x_2}.
  \end{equation}
These eigenfunctions form an orthonormal basis with respect to the natural inner product \[\IP{\phi}{\psi} = \sum_{x \in \torus{\mathbf{m}}} \phi(x) \other{\psi(x)}\] on $\mathbb{C}\mathbb{T}_{\mathbf{m}}$.
\end{lemma}
\begin{proof}
By definition we have
  \begin{align*}
    (K_{\theta,\mathbf{m}} \phi_{z,w}) (x) &= \sum_{y \in \torus{\mathbf{m}}} K_{\theta,\mathbf{m}}(x,y)\phi_{z,w}(y)\\
    &= \alpha\phi_{z,w}(x) + \beta e^{\frac{\pi \iota \theta_1}{m_1}} \phi_{z,w}(x + \mathbf{e}^1)\\
    &+ \gamma e^{\frac{\pi \iota \theta_2}{m_2}} \phi_{z,w}(x + \mathbf{e}^2) + \delta e^{-\frac{\pi \iota \theta_2}{m_2}} \phi_{z,w}(x - \mathbf{e}^2)\\
    &= \left( \alpha + \beta z_\theta + \gamma w_\theta + \delta \overline{w_{\theta}} \right) \phi_{z,w}(x)=: \lambda_{z,w} \phi_{z,w}(x),
  \end{align*}
  which proves that $\phi_{z,w}$ are eigenfunctions with the stipulated eigenvalues. 
  
  The orthonormality is proved easily using the identity $\sum_{i=0}^{n-1} z_1^i \other{z_2}^i = n\indic{z_1=z_2}$ for any $n^{\text{th}}$ roots $z_1$ and $z_2$ of unity.
\end{proof}

In the remainder of this section, we will only be interested in when these eigenvalues are non-zero. To this effect, let us recall an earlier definition. We say that the parameters $\alpha, \beta,\gamma,\delta$ are \textbf{generic with respect to} $\eta \in \{0,1\}$ if $|\alpha +\gamma w + \delta w^{-1}| \neq \beta$ for every $w$ with $w^{n} = (-1)^\eta$. We say that the parameters are \textbf{generic} if they are generic with respect to both $0$ and $1$. For the remainder of this section, \textit{we will implicitly assume that our parameters are generic to avoid repetition}. For now, we can note that this is not a particularly restrictive condition - for any fixed values of $\alpha$, $\gamma$, $\delta$, there are only a finite set of values that $\beta$ cannot take. As we will see in Section~\ref{sec:m1infinity}, these values correspond to phase transitions of the model.

Now we complete the proof of Theorem \ref{thm:pf}. 
\begin{proof}[Proof of Theorem \ref{thm:pf}]
We have proved above the equation \eqref{eq:pf1}. By taking the product of all the eigenvalues, we see that
  \begin{equation}\label{eq:detKevalsequation}
    \det(K_{\theta,\mathbf{m}}) = \prod_{z^{m_1}=1,w^{m_2}=1} \left(\alpha + \beta z_\theta + \gamma w_\theta + \delta \overline{w_{\theta}}\right).
  \end{equation}
  Note that we can rewrite this determinant in  Equation~\eqref{eq:detKevalsequation} as 
\begin{equation}\label{eq:detKniceproduct}
   \det(K_{\theta,\mathbf{m}}) = \prod_{\substack{z^{m_1} = (-1)^{\theta_1} \\ w^{m_2} = (-1)^{\theta_2}} } (\alpha + \beta z + \gamma w + \delta w^{-1}),
\end{equation}  
completing the proof of Theorem \ref{thm:pf}. 
\end{proof}

In the next section we begin working towards our proof of Theorem \ref{thm:cf} and Theorem \ref{thm:cffull0}, which describe the correlations of pure snake configurations under $P_{\mathbf{m}}$.

\subsection{Signed probability measures}

A signed measure $P:V \to \mathbb{R}$ on a finite set $V$ is simply a real-valued function. Given a subset $W$ of $V$ we write $P(W) := \sum_{w \in W} P(w)$. A signed measure $P:V \to \mathbb{R}$ is a signed probability measure on $V$ if $P(V) = \sum_{v \in V} P(v) = 1$. Given a signed measure $P:V \to \mathbb{R}$ and a function $f:V \to \mathbb{R}$, we define the expectation of $f$ to be
\begin{align*}
P[f] := \sum_{v \in V} P(v) f(v).
\end{align*}

If $S:V \to U$ is a function, then this induces by pushforward a signed measure $S_\# P:U \to \mathbb{R}$ on $U$ by setting
\begin{align*}
S_\# P(u) := \sum_{ v \in S^{-1}(u) } P(v).
\end{align*}
Note firstly that the pushforward of a signed probability measure is also a signed probability measure, and secondly that if $g:U \to \mathbb{R}$ and $f = g \circ S$, then
\begin{align} \label{eq:exppush}
S_\# P[ g ] = P [ f ].
\end{align}
Given signed measures $P_1$ and $P_2$ their sum $(P_1+P_2)(v) := P_1(v) + P_2(v)$ is also a signed measure. The affine sum $\lambda P_1 + (1-\lambda) P_2$ of two signed probability measures is also a signed probability measure. 

Pushforwards commute with sums, i.e.
\begin{align*}
S_\# (P_1 + P_2) = S_\#P_1 + S_\#P_2.
\end{align*}

With this in mind, recall the function $\other{Q}_{\theta,\mathbf{m}}:\URDbar_{\mathbf{m}} \to \mathbb{R}$ defined in \eqref{eq:Qdef}. Then we may define a signed probability measure on $\URDbar_{\mathbf{m}}$ by setting
\begin{align} \label{eq:Pthetadef}
\other{P}_{\theta,\mathbf{m}}(\os) := \frac{1}{\det(K_{\theta,\mathbf{m}}) } \other{Q}_{\theta,\mathbf{m}}(\os)=\frac{1}{\det(K_{\theta,\mathbf{m}}) } \mathrm{sgn}(\other{\sigma}) \prod_{x \in \mathbb{T}_{\mathbf{m}}} K_{\theta,\mathbf{m}}(x,\other{\sigma}(x) ).
\end{align}
Recall from \eqref{eq:prefab} that for $\theta_2 = 0,1$ we have $\mathcal{Z}_{\theta_2,\mathbf{m}} = \sum_{\os \in \URDbar_{\mathbf{m}}} \indic{ L(\os) \equiv \theta_2 + m_2 + 1 } w(\os) = \sum_{\theta_1 \in \{0,1\} } C_{\theta,\mathbf{m}} \det(K_{\theta,\mathbf{m}})$. In particular, we may define a signed probability measure $\other{\mathcal{P}}_{\theta_2,\mathbf{m}}$ by setting
\begin{align} \label{eq:barsum}
\other{\mathcal{P}}_{\theta_2,\mathbf{m}}(\os) := \sum_{\theta_1 \in  \{0,1\}} \mu_{\theta,\mathbf{m}} \other{P}_{\theta,\mathbf{m}}( \os).
\end{align}
where $\mu_{\theta,\mathbf{m}} = C_{\theta,\mathbf{m}} \det(K_{\theta,\mathbf{m}})/  \mathcal{Z}_{\theta_2,\mathbf{m}}  $. 

By \eqref{eq:weightsum00} we can also write
\begin{align} \label{eq:evens}
\other{\mathcal{P}}_{\theta_2,\mathbf{m}}(\os) = \frac{1}{\mathcal{Z}_{\theta_2,\mathbf{m}}} \indic{ L(\os) = \theta_2 + m_2 + 1 } {w}(\os).
\end{align}
Define $\lambda_{\theta,\mathbf{m}} := C_{\theta,\mathbf{m}} \det(K_{\theta,\mathbf{m}})/Z_\mathbf{m} = \mu_{\theta,\mathbf{m}} \mathcal{Z}_{\theta_2,\mathbf{m}}/Z_\mathbf{m}$. Then we may also define
\begin{align} \label{eq:barsum2}
\other{P}_{\mathbf{m}}(\sigma) := \sum_{\theta \in \{0,1\}^2} \lambda_{\theta,\mathbf{m}} \other{P}_{\theta,\mathbf{m}}(\os) = \sum_{\theta_2 \in \{0,1\}} \frac{\mathcal{Z}_{\theta_2,\mathbf{m}}}{Z_{\mathbf{m}}} \other{\mathcal{P}}_{\theta_2,\mathbf{m}}(\os).
\end{align}
By \eqref{eq:evens} we have 
\begin{align} \label{eq:evens2}
    \other{P}_{\mathbf{m}}(\sigma) = {w}(\os)/Z_{\mathbf{m}}.
\end{align}

Recall that we have the natural map $\mathrm{sh}:\URDbar_{\mathbf{m}} \to \URD_{\mathbf{m}}$ sending a generalised snake configuration to its underlying pure snake configuration by deleting all snakelets (a.k.a. two-cycles).

For each of the signed measures $\other{P}_{\theta,\mathbf{m}},\other{\mathcal{P}}_{\theta_2,\mathbf{m}}, \other{Q}_{\theta,\mathbf{m}},\other{P}_\mathbf{m}$ we write $P_{\theta,\mathbf{m}},\mathcal{P}_{\theta_2,\mathbf{m}},Q_{\theta,\mathbf{m}},P_{\mathbf{m}}$ for their associated pushforward signed measures on $\URD_{\mathbf{m}}$ under the map $\mathrm{sh}:\URDbar_{\mathbf{m}} \to \URD_{\mathbf{m}}$. Note that since pushforwards commute with sums of measures, the equation \eqref{eq:barsum} pushes forward to
\begin{align} \label{eq:barsum3}
\mathcal{P}_{\theta_2,\mathbf{m}}(\sigma) = \sum_{\theta_1 \in \{0,1\}} \mu_{\theta,\mathbf{m}} P_{\theta,\mathbf{m}}( \sigma );
\end{align}
likewise with \eqref{eq:barsum2}. 

By Lemma \ref{lem:shapepush}, \eqref{eq:evens} and the definition of $\mathcal{P}_{\theta_2,\mathbf{m}}$ being the pushforward of $\other{\mathcal{P}}_{\theta_2,\mathbf{m}}$ under the map $\mathrm{sh}:\URDbar_{\mathbf{m}} \to \URD_{\mathbf{m}}$, we have
\begin{align}
\mathcal{P}_{\theta_2,\mathbf{m}}(\sigma) = \indic{ L(\sigma) = \theta_2 + m_2+  1} w(\sigma)J(\sigma)/\mathcal{Z}_{\theta_2,\mathbf{m}}.
\end{align}
By Remark \ref{rem:positivity}, in the probabilistic regime $\alpha^2 - 4\gamma \delta \geq 0$, we have $J(\sigma) > 0$, and in this case since ${w}(\sigma)$ is also non-negative on pure snake configurations, $\mathcal{P}_{\theta_2,\mathbf{m}}(\sigma)$ is a genuine probability measure. 

Let us summarise our work in this section in the following remark:
\begin{rem}
In the probabilistic regime, $\mathcal{P}_{\theta_2,\mathbf{m}}$ and $P_{\mathbf{m}}$ are probability measures on the set of pure snake configurations. These measures are affine sums \eqref{eq:barsum} of pushforwards of signed determinantal probability measures.
\end{rem}

We will understand the value of this remark in the next section: namely, we will see that the correlations of determinantal product signed probability measures have explicit correlation structure.

\subsection{Correlation functions}
We would now like to consider the local correlation structures of signed probability measures. For reasons that will be apparent later, we will change convention and consider $\torus{\mathbf{m}}$ as the set $[m_1] \times [m_2]$, where $[n]$ denotes the set $\{0,\dots,n-1\}$ and write $\mathbf{e}^1 = (1,0)$ and $\mathbf{e}^2 = (0,1)$ as elements of $\mathbb{Z}^2$. Let $x^i \in \torus{\mathbf{m}}$ and take $\mathbf{f}^i \in \{\mathbf{0}, \mathbf{e}^1, \mathbf{e}^2, -\mathbf{e}^2\}$ for all $i = 1,\dots,k$. Now define the notation
\begin{align} \label{eq:Pbarprob}
\other{P}_{\mathbf{m}} \left( \bigcap_{i=1}^k \{\sigmabar x^i = x^i + \mathbf{f}^i\} \right) &:= \sum_{ \os \in \URDbar_{\mathbf{m}}} \left(\prod_{i=1}^k\indic{\sigmabar x^i = x^i + \mathbf{f}^i }\right) \other{P}_{\mathbf{m}}(\os),
\end{align}
and
\begin{align} \label{eq:Pprob}
P_{\mathbf{m}} \left( \bigcap_{i=1}^k \{\sigma(x^i) = x^i + \mathbf{f}^i\} \right) &:= \sum_{ \sigma \in \URD_{\mathbf{m}} } \left(\prod_{i=1}^k\indic{\sigma(x^i) = x^i + \mathbf{f}^i }\right) P_{\mathbf{m}}(\sigma)
\end{align}
and similar for $\other{\mathcal{P}}_{\theta_2,\mathbf{m}}$, $P_{\theta_2,\mathbf{m}}$, $\other{P}_{\theta,\mathbf{m}}$ and $P_{\theta,\mathbf{m}}$. It is important to be meticulous about this convention for this subsection - when we write $x^i + \mathbf{f}^i$, we will consider this as addition in $\mathbb{Z}^2$, but the event $\{\sigma(x^i) = x^i + \mathbf{f}^i \}$ is merely shorthand for the event that "at $x^i$, the snake $\sigma$ makes a move in the $\mathbf{f}^i$ direction" (or similar for $\os$ instead of $\sigma$). Thus, $ \left(\prod_{i=1}^k\indic{\sigma(x^i) = x^i + \mathbf{f}^i }\right)$ denotes the indicator function that $\sigma(x^i) \equiv x^i + \mathbf{f}^i \mod{(m_1,m_2)}$ for each $i = 1,\dots,k$, and similar for the analogous indicator in \eqref{eq:Pbarprob}.

In the probabilistic regime, where $P_{\mathbf{m}}$  is a genuine probability measure rather than just a signed probability measure, the quantity in \eqref{eq:Pprob} coincides with the natural notion of probability of the event $\bigcap_{i=1}^k \{\sigma(x^i) = x^i + \mathbf{f}^i\}$ for a random $\sigma$ chosen according to $P_{\mathbf{m}}$. 

Herein lies the value of the aforementioned decomposition of $P_{\mathbf{m}}$: the correlation functions under $P_{\theta,\mathbf{m}}$ can be computed explicitly in terms of the inverse operator $K_{\theta,\mathbf{m}}^{-1}$.

\begin{lemma}\label{lem: detKinverseformula}
We have 
  \begin{align}\label{eq: Kthetainv}
    K_{\theta, \mathbf{m}}^{-1}(x,y) &= \frac{1}{m_1m_2} \sum_{ z^{m_1} = 1, w^{m_2} = 1} \frac{ z^{-(y_1-x_1)} w^{ - (y_2-x_2) } }{ \alpha +\beta z_\theta + \gamma w_\theta + \delta w_\theta^{-1} }, 
    \end{align}
    where $z_\theta := z e^{\pi \iota \theta_1/m_1}$ and $w_\theta := we^{\pi \iota \theta_2/m_2}$. 
\end{lemma}
\begin{proof}
Recall from Lemma \ref{lem:eigensystem} that $K_{\theta,\mathbf{m}}$ has an orthonormal system of eigenfunctions $\phi_{z,w}$ with non-zero eigenvalues $\lambda_{z,w}$. Now we simply use the fact that $K_{\theta,\mathbf{m}}^{-1}(x,y) = \sum_{z^{m_1}=1,w^{m_2}=1} \lambda_{z,w}^{-1}\phi_{z,w}(x)\overline{\phi_{z,w}(y)}$, which amounts to \eqref{eq: Kthetainv}.
\end{proof}

\begin{lemma} \label{lem:signedcorr}
We have 
\begin{align*}
\other{P}_{\theta,\mathbf{m}}\left( \bigcap_{i=1}^k \{\sigma(x^i) = x^i + \mathbf{f}^i\} \right) = \prod_{i=1}^k K_{\theta,\mathbf{m}}( x^i, x^i + \mathbf{f}^i) \det_{i,j=1}^k \left( K_{\theta,\mathbf{m}}^{-1}( x^i + \mathbf{f}^i, x^j) \right),
\end{align*}
where $K_{\theta,\mathbf{m}}^{-1}$ is the inverse of the operator $K_{\theta,\mathbf{m}}$.
\end{lemma}

There is a slight subtlety here, where we note that $K_{\theta, \mathbf{m}}$ and $K_{\theta,\mathbf{m}}^{-1}$ technically cannot necessarily take $x^i + \mathbf{f}^i$ as an input (e.g. if $x^i = (0,m_2-1), \mathbf{f}^i = (0,1)$) but we note that $K_{\theta, \mathbf{m}}$ and $K_{\theta,\mathbf{m}}^{-1}$ naturally extend to operators on $(\mathbb{Z}^2)^2$ in such a way that they take the same value on each residue class modulo ${(m_1,m_2)}$, and so we have implicitly extended the operators in the statement of Lemma~\ref{lem:signedcorr}.
\begin{proof}
According to Jacobi's identity (see e.g.\ Horn and Johnson \cite[Section 0.8.4]{HJ}), given an invertible matrix $(A_{x,y})_{1 \leq x,y \leq n}$ and any distinct $x^1,\dots,x^k$ and distinct $y^1,\dots,y^k$ in $\{1,\dots,n\}$, we have
\begin{align*}
\sum_{ \sigma \in \mathcal{S}_n } \prod_{i=1}^k \indic{ \sigma(x^i) = y^i } \mathrm{sgn}(\sigma) \prod_{x=1}^n A_{x,\sigma(x)} = \prod_{i=1}^k A_{x^i,y^i} \det_{x,y=1}^n(A_{x,y}) \det_{i,j = 1}^k ( A^{-1}_{y^i,x^j} ). 
\end{align*}
Now use \eqref{eq:Pthetadef} and let $A_{x,y}$ be the operator $\{ K_{\theta,\mathbf{m}}(x,y) : x,y \in \mathbb{T}_{\mathbf{m}} \}$ with $y^i$ being the unique elements of $\torus{\mathbf{m}}$ such that $y^i \equiv x^i + \mathbf{f}^i \mod{(m_1,m_2)}$, to find that
\begin{align*}
    \other{P}_{\theta,\mathbf{m}}\left( \bigcap_{i=1}^k \{\sigma(x^i) = x^i + \mathbf{f}^i\} \right) = \prod_{i=1}^k K_{\theta,\mathbf{m}}( x^i, y^i) \det_{i,j=1}^k \left( K_{\theta,\mathbf{m}}^{-1}( y^i, x^j) \right)
\end{align*}
Finally, use that  $K_{\theta, \mathbf{m}}$ and $K_{\theta,\mathbf{m}}^{-1}$ evaluate identically within residue classes modulo $(m_1,m_2)$. 
\end{proof}

Combining the previous two results, we can slightly simplify the form of this determinantal structure.

\begin{cor} \label{cor:nform}
We have \[\other{P}_{\theta,\mathbf{m}} \left( \bigcap_{i=1}^k \{\sigma(x^i) = x^i + \mathbf{f}^i\} \right) = \prod_{i=1}^k K_{00,\mathbf{m}}( x^i, x^i + \mathbf{f}^i) \det_{i,j=1}^k \left( H_{\theta,\mathbf{m}}(x^i+\mathbf{f}^i, x^j) \right),\] where 
\begin{align} \label{eq:nform}
H_{\theta,\mathbf{m}}(x,x') := \frac{1}{m_1m_2}\sum_{z^{m_1}=(-1)^{\theta_1},w^{m_2}=(-1)^{\theta_2}} \frac{ z^{-(x'_1-x_1)} w^{ - (x_2'-x_2) } }{ \alpha +\beta z + \gamma w + \delta \overline{w} }
\end{align}
\end{cor}
Note that by reindexing $z$ and $w$, we must take care in that, for example, $z^x$ is no longer well-defined for $x \in \mathbb{Z}_m$.
\begin{proof}
This follows from the fact that 
\begin{align*}
\det_{i,j=1}^k \left( H_{\theta,\mathbf{m}}(x^i+\mathbf{f}^i, x^j) \right) &= \det_{i,j=1}^k \left( e^{\pi \iota \frac{\theta_1}{m_1}[(x^i + \mathbf{f}^i)_1-x^j_1]_{m_1} + \pi \iota \frac{\theta_2}{m_2}[ (x^i + \mathbf{f}^i)_2 - x^j_2]_{m_2} } K_{\theta,\mathbf{m}}^{-1}(x^i + \mathbf{f}^i,x^j) \right)\\
&= \prod_{i=1}^k \frac{K_{\theta,\mathbf{m}}(x^i,x^i + \mathbf{f}^i)}{K_{00,\mathbf{m}}(x^i,x^i + \mathbf{f}^i)} \det_{i,j=1}^k \left( K_{\theta,\mathbf{m}}^{-1}(x^i + \mathbf{f}^i,x^j) \right),
\end{align*}
where we simply used the definition \eqref{eq:Ktheta} to obtain the final equality above.
\end{proof}

Thus we have seen that the correlation functions of the signed probability measure $P_{\theta,\mathbf{m}}$ on $\URDbar_{\mathbf{m}}$ have 
determinantal representations in terms of an operator that can be computed explicitly. We now seek to tackle the question of how these correlations push forward onto pure snake configurations. 

Let $\os = (\sigma,N)$ be a generalised snake configuration with underlying pure snake configuration $\mathrm{sh}(\os) = \sigma$. We would like to relate events of the form 
$\bigcap_{i=1}^k \{\sigma(x^i) = x^i + \mathbf{f}^i\}$ to those of the form $\bigcap_{i=1}^k \{\os (x^i) = x^i + \mathbf{f}^i\}$. 

Note that these events do not coincide; for example, consider that we in fact have the disjoint union 
\begin{align*}
\{ \os(x) = x + \mathbf{e}^2 \} &= \{ \sigma(x) = x+\mathbf{e}^2 \} \cup \{ (x, x+\mathbf{e}^2) \text{ is a snakelet} \}.
\end{align*}
However, the following lemma gives a recipe for inverting such relationships, expressing all $\sigma$ events in terms of $\os$. 

\begin{lemma} \label{lem:events}
Let $\mathrm{sh}(\os) = \sigma$. Then we have the following equalities of events
\begin{align}
\{ \sigma(x) = x \}
&=
\begin{aligned}[t]
&\{ \os(x) = x \}
   \cup \{ \os(x) = x+\mathbf{e}^2 , \os(x+\mathbf{e}^2) = x \} \\
&\quad
   \cup \{ \os(x) = x-\mathbf{e}^2 , \os(x-\mathbf{e}^2) = x \}
\end{aligned}
\label{eq:ev1}
\\
\{ \sigma(x) = x +\mathbf{e}^1 \}
&= \{ \os(x) = x+ \mathbf{e}^1 \}
\label{eq:ev2} \\
\{ \sigma(x) = x+ \mathbf{e}^2 \}
&= \{ \os(x)  = x + \mathbf{e}^2 \}
   - \{ \os(x) = x + \mathbf{e}^2 , \os(x+\mathbf{e}^2) = x \}
\label{eq:ev3} \\
\{ \sigma(x) = x - \mathbf{e}^2 \}
&= \{ \os(x)  = x - \mathbf{e}^2 \}
   - \{ \os(x) = x - \mathbf{e}^2 , \os(x-\mathbf{e}^2) = x \}
\label{eq:ev4}
\end{align}
\end{lemma}
\begin{proof}
For \eqref{eq:ev1} note that $\sigma(x) = x$, then either $\os(x) = x$ or $x$ is part of a snakelet. 
For \eqref{eq:ev2} note that $\sigma(x) = x + \mathbf{e}^1$ if and only if $\os(x) = x + \mathbf{e}^1$. For \eqref{eq:ev3} (resp.\ \eqref{eq:ev4}), note that $\sigma(x) = x+\mathbf{e}^2$ if and only if $\os(x) =  x +\mathbf{e}^2$ (resp.\ $\os(x) = x - \mathbf{e}^2$) and $x$ is not part of a snakelet. 
\end{proof}
 For $\mathbf{f} \in  \{0,\mathbf{e}^1,\mathbf{e}^2,-\mathbf{e}^2\}$ and a pure or generalised snake configuration $\os$, let 
 \begin{align*}
 R_x^{\mathbf{f}}(\os) := \indic { \os(x) = x + \mathbf{f} } 
 \end{align*}
be the indicator function of the event $ \{ \os(x) = x + \mathbf{f} \} $.

Then in this notation, Lemma \ref{lem:events} reads
\begin{align}
T^0_x(\os) := R^0_x(\mathrm{sh}(\os)) &= R^0_x(\os) + R^{\mathbf{e}^2}_x(\os) R^{-\mathbf{e}^2}_{x+\mathbf{e}^2}(\os) + R^{-\mathbf{e}^2}_x(\os) R^{\mathbf{e}^2}_{x-\mathbf{e}^2}(\os) \label{eq:ev1b} \\
T^{\mathbf{e}^1}_x(\os) := R^{\mathbf{e}^1}_x(\mathrm{sh}(\os)) &= R^{\mathbf{e}^1}_x(\os) \label{eq:ev2b} \\
T^{\mathbf{e}^2}_x(\os) := R^{\mathbf{e}^2}_x(\mathrm{sh}(\os)) &= R^{\mathbf{e}^2}_x(\os) ( 1 - R^{-\mathbf{e}^2}_{x+\mathbf{e}^2}(\os) ) \label{eq:ev3b} \\
T^{ - \mathbf{e}^2}_x(\os) := 
R^{-\mathbf{e}^2}_x(\mathrm{sh}(\os)) &= R^{-\mathbf{e}^2}_x(\os) ( 1 - R^{\mathbf{e}^2}_{x-\mathbf{e}^2}(\os) ), \label{eq:ev4b}
\end{align}
which is precisely \eqref{eq:ev1b0}-\eqref{eq:ev4b0} from the introduction.

 The following lemma offers a conversion for expressing correlations under the pushforward measures $P_{\mathbf{m}}, P_{\theta,\mathbf{m}}$ in terms of correlations under $\other{P}_{\mathbf{m}},\other{P}_{\theta,\mathbf{m}}$.

 \begin{lemma} \label{lem:pushprob}
 Let $\other{P}$ be a signed probability measure on $\URDbar_{\mathbf{m}}$ and let $P$ be its pushforward under $\mathrm{sh}:\URDbar_{\mathbf{m}} \to \URD_{\mathbf{m}}$. Then
 \begin{align*}
 P\left[ \prod_{i=1}^k R_{x^i}^{\mathbf{f}^i} \right] = \other{P} \left[ \prod_{i=1}^k T_{x^i}^{\mathbf{f}^i} \right].
 \end{align*}
 \end{lemma}
\begin{proof}
If $f(\os) := \prod_{i=1}^k T_{x^i}^{\mathbf{f}^i}(\os)$ and $g(\sigma) :=  \prod_{i=1}^k R_{x^i}^{\mathbf{f}^i}(\sigma)$ then $f(\os) = g(\mathrm{sh}(\sigma))$. Now use \eqref{eq:exppush}.
\end{proof}

It is also true that every product in the $(T_x^{\mathbf{f}})$ has an expression as a sum of products of the $(R_x^{\mathbf{f}})$ indicators. 

As noted in the introduction, this implies that every word $\prod_{i=1}^k T_{x^i}^{\mathbf{f}^i}$ has an expression
\begin{align} \label{eq:wordy}
\prod_{i=1}^k T_{x^i}^{\mathbf{f}_i} = \sum_{\kappa \in \mathcal{K}} a_\kappa \prod_{i=1}^{r_\kappa} R_{x^{\kappa,i}}^{\mathbf{f}^{\kappa,i}}, 
\end{align}
as a sum of words in the variables $(R_x^{\mathbf{f}})$. Here we have indexed the words by $\kappa$ in some indexing set $\mathcal{K}$, with coefficients $\alpha_\kappa \in \mathbb{Z}$. 

Let us highlight in particular that when $\mathbf{f}^i = \mathbf{e}^1$, by \eqref{eq:ev2b} we have the following particularly simple special case of \eqref{eq:wordy}:
\begin{align} \label{eq:wordyspecial}
\prod_{i=1}^k T_{x^i}^{\mathbf{e}^1} = \prod_{i=1}^k R_{x^i}^{\mathbf{e}^1}
\end{align}

With this in hand, we are now ready to prove Theorem \ref{thm:cffull0}, which is our foundational result on the correlation structure of the probability measure $P_{\mathbf{m}}$. We will prove it as a corollary of a more specific result, about the correlation structure of the probability measure $\mathcal{P}_{\theta_2,\mathbf{m}}$.

\begin{thm} \label{thm:cffull1}
Suppose that the product $\prod_{i=1}^k T_{x^i}^{\mathbf{f}^i}$ has an expression $\prod_{i=1}^k T_{x^i}^{\mathbf{f}^i} = \sum_{\kappa \in \mathcal{K}} a_\kappa \prod_{i=1}^{r_\kappa} R_{x^{\kappa,i}}^{\mathbf{f}^{\kappa,i}}$. Then we can write that
\begin{align} \label{eq:corrstructuretheta2}
\mathcal{P}_{\theta_2,\mathbf{m}}\left( \bigcap_{i=1}^k \{ \sigma(x^i) = x^i + \mathbf{f}^i \} \right) =  \sum_{\kappa \in \mathcal{K}} a_\kappa \sum_{\theta_1 \in  \{0,1\} } \mu_{\theta,\mathbf{m}} \other{P}_{\theta,\mathbf{m}} \left[\prod_{i=1}^{r_\kappa} R_{x^{\kappa,i}}^{\mathbf{f}^{\kappa,i}}  \right].
\end{align}
\end{thm}

\begin{proof}


Using \eqref{eq:barsum3} to obtain the first equality below, the definitions of the indicators $R_x^\mathbf{f}$ to obtain the second, Lemma \ref{lem:pushprob} to obtain the third, we have 
\begin{align*}
\mathcal{P}_{\theta_2,\mathbf{m}} \left( \bigcap_{i=1}^k \{ \sigma(x^i) = x^i + \mathbf{f}^i \} \right) &=\sum_{\theta_1 \in \{0,1\} } \mu_{\theta,\mathbf{m}} P_{\theta,\mathbf{m}} \left( \bigcap_{i=1}^k \{ \sigma(x^i) = x^i + \mathbf{f}^i \} \right)\\
&=  \sum_{\theta_1 \in \{0,1\} } \mu_{\theta,\mathbf{m}} P_{\theta,\mathbf{m}} \left[ \prod_{i=1}^k R_{x^i}^{\mathbf{f}^i} \right]\\
&=\sum_{\theta_1 \in \{0,1\} } \mu_{\theta,\mathbf{m}}\other{P}_{\theta,\mathbf{m}} \left[ \prod_{i=1}^k T_{x^i}^{\mathbf{f}^i} \right],
\end{align*}
where we recall that $P_{\theta,\mathbf{m}}$ is the pushforward under $\mathrm{sh}:\URDbar_{\mathbf{m}} \to \URD_{\mathbf{m}}$ of the signed probability measure $\other{P}_{\theta,\mathbf{m}}$ defined in \eqref{eq:Pthetadef}. Expressing $ \prod_{i=1}^k T_{x^i}^{\mathbf{f}^i}$ as a linear combination of words in the variables $(R_x^{\mathbf{f}})$ as using linearity we in turn obtain
\begin{align*}
\mathcal{P}_{\theta_2,\mathbf{m}} \left( \bigcap_{i=1}^k \{ \sigma(x^i) = x^i + \mathbf{f}^i \} \right) &= \sum_{\theta_1 \in \{0,1\} } \mu_{\theta,\mathbf{m}}\sum_{\kappa \in \mathcal{K}} a_\kappa \other{P}_{\theta,\mathbf{m}} \left[ \prod_{i=1}^{r_\kappa} R_{x^{\kappa,i}}^{\mathbf{f}^{\kappa,i}}  \right].
\end{align*}
Reordering the sums, we obtain Equation~\eqref{eq:corrstructuretheta2} as written. 
\end{proof}

\begin{proof}[Proof of Theorem~\ref{thm:cffull0}]
    Use the pushforward of \eqref{eq:barsum2} together with Theorem~\ref{thm:cffull1} and Corollary~\ref{cor:nform}.
\end{proof}

We close this section with a proof of Theorem \ref{thm:cf}, which is simply a special case of Theorem \ref{thm:cffull0} when $\mathbf{f}^i = \mathbf{e}^1$ for every $i=1,\dots,k$.

\begin{proof}[Proof of Theorem \ref{thm:cf}]
When $\mathbf{f}^i = \mathbf{e}^1$ for every $i$, we have $K_{00,\mathbf{m}}(x^i,x^i+\mathbf{f}^i) = \beta$. Moreover, when $\mathbf{f}^i = \mathbf{e}^1$ for every $i$ by \eqref{eq:ev2b} we have the equality of words $\prod_{i=1}^k T_{x^i}^{\mathbf{f}^i} = \prod_{i=1}^k R_{x^i}^{\mathbf{f}^i}$. In Theorem~\ref{thm:cffull0}, we obtain
\begin{align} \label{eq:corrstructure3}
P_\mathbf{m} \left( \bigcap_{i=1}^k \{ \sigma(x^i) = x^i + \mathbf{f}^i \} \right) =  \sum_{\theta \in \{0,1\}^2 } \lambda_{\theta,\mathbf{m}} \det_{i,j=1}^k \left[ \beta H_{\theta,\mathbf{m}}(x^i+\mathbf{e}^1,x^j) \right].
\end{align}
To obtain \eqref{eq:cf} as written, simply note $\beta H_{\theta,\mathbf{m}}(x+\mathbf{e}^1,x') = G_{\theta,\mathbf{m}}(x,x')$. 
\end{proof}

\section{ \texorpdfstring{The $m_1 \to \infty$ limit of our model}{The m1 to infinity limit of our model}} \label{sec:m1infinity}

\subsection{\texorpdfstring{The integrable snake model on $\mathbb{Z} \times [n]$}{The integrable snake model on Z x [n]}}

Let us take stock of what we have achieved in the last section. In the probabilistic regime $\alpha^2 - 4\gamma \delta \geq 0$, we showed that we can define a probability measure on pure snake configurations of the form
\begin{align*}
\mathcal{P}_{\theta_2,\mathbf{m}}(\sigma) = w(\sigma)J(\sigma)\indic{ L(\sigma) = \theta_2 + m_2 + 1} / \mathcal{Z}_{\theta_2,\mathbf{m}},
\end{align*}
and this probability measure has explicit determinantal correlations given by Theorem~\ref{thm:cffull1} and Corollary~\ref{cor:nform}. In this section, we study the asymptotics of this measure as $m_1 \to \infty$ with all other parameters staying fixed.

Recall our notational convention that, for $n \in \mathbb{N}$, $[n]$ denotes the set $\{0,\dots,n-1\}$. Then, note that any pure snake configuration on the discrete torus \[\mathbb{T}_{m_1,m_2} = [m_1] \times [m_2] \] may be extended to a pure snake configuration on the semi-infinite discrete cylinder \[\mathbb{T}_{m_2} := \mathbb{Z} \times [m_2]\] simply by repeating periodically in $m_1$ (or even to $\mathbb{Z}^2$ by repeating periodically in $m_2$ as well). Thus we can equivalently treat $\mathcal{P}_{\theta_2,\mathbf{m}}$ as a probability measure on pure snake configurations on $\mathbb{T}_{m_2}$. The main result of this section states that this probability measure converges in distribution as $m_1 \to \infty$. It transpires that the limiting behaviour of $\mathcal{P}_{\theta_2,\mathbf{m}}$ undergoes several phase transitions which are related to the magnitude of the complex numbers $\{ \alpha + \gamma w + \delta w^{-1} : w^{m_2} = (-1)^{\theta_2} \}$ in relation to $\beta$.

Since multiplying all of the parameters by some scalar $\lambda > 0$ leaves the law of $\mathcal{P}_{\theta_2,\mathbf{m}}$ unchanged, we may assume without loss of generality that $\alpha = 1$. So in summary, the probabilistic regime is given by
\begin{align*}
\alpha = 1 \quad \text{and} \quad \beta, \gamma, \delta \geq 0 \quad \text{ with } \quad 4\gamma \delta \leq 1.
\end{align*}
Also, to lighten the notation, for the remainder of this section we will write $m$ for $m_1$ and $n$ for $m_2$.
Let us define the sets $\mathcal{L}_\eta = \{w\colon w^n = (-1)^{\eta},|1 + \gamma w + \delta w^{-1}| < \beta\}$ and $\mathcal{R}_{\eta} = \{w\colon w^n = (-1)^{\eta},|1 + \gamma w + \delta w^{-1}| \geqslant \beta\}$ for $\eta = 0,1$, c.f. \eqref{eq:rootset1} and \eqref{eq:rootset2}.

{Recall the definition of \textbf{genericity} from Section~\ref{sec: fibonacci}}. We say that $\beta,\gamma,\delta$ are \textbf{generic with respect to} $\eta = 0,1$ if $|1+\gamma w + \delta w^{-1}| \neq \beta$ for every $w$ with $w^{n} = (-1)^\eta$. It transpires that for generic $\beta,\gamma,\delta$, the parameter $\beta$ only affects the asymptotic behaviour of the probability measure $P_{\theta,(m,n)}$ through the sets $\mathcal{L}_{\theta_2}$ and $\mathcal{R}_{\theta_2}$. In particular, the size of $\mathcal{L}_{\theta_2}$ for a certain value of ${\theta_2}$ dictates the occupation number (recall the definition from \eqref{eq:occupationnumberdef}) in the asymptotic limit. This gives rise to some interesting geometrical conditions that relate to the problem at hand, which we will begin by outlining in advance of our analytic results.

Firstly, the "triangle inequality" conditions $\beta < 1 - \gamma - \delta$ (if possible) and $\beta > 1 + \gamma + \delta$ correspond to $\mathcal{L}_\eta$ (for both $\eta = 0,1$) having sizes $0$ and $n$ respectively and so, will also correspond to our snake configurations having occupation numbers $0$ and $n$ respectively. The asymptotic limit turns out to behave differently in these regimes.

Secondly, we note a geometric lemma that will later justify the appearance of $\mathcal{L}_{\theta_2}$ and $\mathcal{R}_{\theta_2}$. Consider the set $E_{\gamma,\delta} := \{ 1 + \gamma w + \delta w^{-1} : w \in S^1 \}$, which is parametrisable as $1 + (\gamma+\delta)\cos t + i (\gamma-\delta)\sin t$ as $t$ varies over $[0,2\pi)$. Writing this ellipse in terms of its foci, we may alternatively write
\begin{align} \label{eq:Eform}
E_{\gamma,\delta} = \{ z \in \mathbb{C} : |z - (1-\sqrt{4\gamma \delta})| + | z - (1+\sqrt{4\gamma \delta}) | = 2(\gamma+\delta) \} .
\end{align}
We use this representation to prove the following.

\begin{lemma} \label{lem:arc}
Suppose we are in the probabilistic regime where $4 \gamma \delta \leq 1$ and suppose the circle $C_\beta := \{ |z| = \beta \}$ and the ellipse $E_{\gamma,\delta}$ intersect. Then the set $A_\beta := \{ w \in S^1 : |1+\gamma w + \delta w^{-1}| > \beta \}$ is an arc of the form $A_\beta := \{ w = e^{i t} : t \in (-t_\beta,t_\beta) \}$ for some $t_\beta \in [0,\pi]$ depending only on $\beta$. 
\end{lemma}

Notably, $E_{\gamma,\delta}$ cannot intersect $C_\beta$ four times.

\begin{proof}

\begin{figure}[h!t]
\begin{centering}
\includegraphics{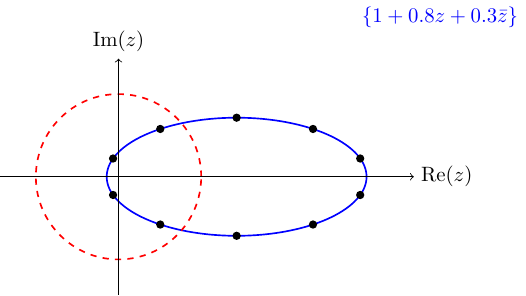}
\end{centering}
\caption{ \label{fig:ellcirc} The blue curve denotes the ellipse $E_{\gamma,\delta}$ and the dotted red curve is the circle $C_\beta$. The points on the blue curve of the form $1 + \gamma w + \delta w^{-1}$ where $w^{n} = (-1)^{\eta}$ are denoted in black. The points lying inside the red curve are $L_\eta$ and outside are $R_\eta$.}
\end{figure}

The statement of the lemma is tantamount to establishing that, either $C_\beta$ and $E_{\gamma,\delta}$ do not intersect, or the leftmost point of $E_{\gamma,\delta}$ does not lie to the left of the leftmost point of $C_\beta$; see Figure \ref{fig:ellcirc}. Equivalently, we need to show that if $C_\beta$ and $E_{\gamma,\delta}$ intersect, then
\begin{align} \label{eq:leftmost}
1 - \gamma - \delta \geq - \beta. 
\end{align}

To show that this is the case, suppose there exists some point $z \in C_\beta \cap E_{\gamma,\delta}$. Then using \eqref{eq:Eform} to obtain the first equality below, the triangle inequality to obtain the following inequality and then the fact that $|z| = \beta $ and $4 \gamma \delta \leq 1$ to obtain the final equality, we have  
\begin{align*}
2(\gamma+\delta ) &= |z - (1-\sqrt{4\gamma \delta})| + | z - (1+\sqrt{4\gamma \delta}) |\\
&\leq 2|z| + | 1-\sqrt{4\gamma \delta}| + | 1+\sqrt{4\gamma \delta} | = 2\beta + 2.
\end{align*}
Rearranging, we obtain \eqref{eq:leftmost}, which completes the proof.
\end{proof}

Additionally, we have the following remark.

\begin{rem}\label{rem:arc}
    The angle $t_\beta$ defining the arc $A_\beta$ is continuously decreasing in $\beta$. 
\end{rem}

Now, a brief lemma that will prove useful for the upcoming analysis.

\begin{lemma}\label{lem:sgnCdetK}
  Suppose we are in the probabilistic regime where $4\gamma \delta \leq 1$. Fix $\theta_2$ and further suppose that our parameters are generic with respect to $\theta_2$. Then
  \begin{equation*}
    \sgn\left(C_{\theta,(m,n)} \det(K_{\theta,(m,n)})\right) = \begin{cases}
      2C_{\theta,(m,n)} & \text{ if } 0 \leq \beta < 1- \gamma - \delta,\\
      1 & \text{ if } 1- \gamma - \delta < \beta < 1 + \gamma + \delta,\\
      (-1)^{(m+\theta_1 + 1)(\theta_2 + 1)} & \text{ if } \beta > 1 + \gamma + \delta.
    \end{cases}
  \end{equation*}
\end{lemma}

\begin{proof}
    Recall from \eqref{eq:detKniceproduct} the formula 
    \begin{align*}
      \det(K_{\theta,(m,n)}) &= \prod_{z^m = (-1)^{\theta_1}} \prod_{w^n = (-1)^{\theta_2}} \left(1 + \beta z + \gamma w + \delta w^{-1}\right).
    \end{align*} 
    Terms in this product with conjugate values of $z$ and $w$ can be paired up to product to a positive real value, so the sign is dictated by the signs of the remaining terms, i.e. whenever $z$ and $w$ are $\pm 1$. That is to say, $\det(K_{\theta,(m,n)})$ has the same sign as 
    \begin{equation}\label{eq:produtdeterminessign}
    (1+\beta + \gamma + \delta)^{\indicpower{\theta_1, \theta_2}}(1 + \beta - \gamma -\delta)^{\indicpower{\theta_1, n + \theta_2}}(1 - \beta + \gamma + \delta)^{\indicpower{m + \theta_1,\theta_2}}(1 - \beta - \gamma - \delta)^{\indicpower{m + \theta_1, n + \theta_2}}
    \end{equation}
    where the indicator $\indicpower{x,y}$ is shorthand for $\indicpower{x \equiv y \equiv 0 \mod{2}}$. Note the identity that \[(-1)^{\indicpower{x,y}} = (-1)^{(x+1)(y+1)}.\]
    Now let us work case by case. For $\beta < 1- \gamma - \delta$, all the terms in brackets in \eqref{eq:produtdeterminessign} are positive, so $\det(K_{\theta,(m,n)})$ is positive and we have the required sign for $C_{\theta,(m,n)} \det(K_{\theta,(m,n)})$.
    
    For $1- \gamma - \delta < \beta < 1 + \gamma + \delta$, one can determine immediately that $1 + \beta + \gamma + \delta$ and $1 - \beta + \gamma + \delta$ are positive whilst $1 - \beta - \gamma - \delta$ is negative. Furthermore, since $1 - \gamma - \delta < \beta < 1 + \gamma + \delta$, we have that $C_\beta$ and $E_{\gamma,\delta}$ intersect, so \eqref{eq:leftmost} holds and $1 + \beta - \gamma - \delta$ is also positive. Hence, $\det(K_{\theta,(m,n)})$ has sign $(-1)^{\indicpower{m + \theta_1, n + \theta_2}} = 2C_{\theta,(m,n)}$ and $C_{\theta,(m,n)} \det(K_{\theta,(m,n)})$ is always positive.
    
    Finally, for $\beta > 1 + \gamma + \delta$, the only negative terms in brackets in \eqref{eq:produtdeterminessign} are $1 - \beta + \gamma + \delta$ and $1 - \beta - \gamma - \delta$. So the sign of $\det(K_{\theta,(m,n)})$ is \begin{align*}
        (-1)^{\indicpower{m + \theta_1, n +\theta_2}} (-1)^{\indicpower{m + \theta_1,\theta_2}} &= (-1)^{(m+\theta_1 + 1)(n+\theta_2 + 1) + (m+\theta_1 + 1)(\theta_2 + 1)}\\
        &= (-1)^{(m + \theta_1 + 1)n},
    \end{align*}
    which also gives the required sign for $C_{\theta,(m,n)} \det(K_{\theta,(m,n)})$.
\end{proof}

Now we are ready to prove our main results.

\begin{thm} \label{thm:m1infinity2}
Fix $\theta_2$ and suppose that $\alpha=1$, the parameters $\beta,\gamma,\delta$ are generic with respect to $\theta_2$, $1 - \gamma - \delta < \beta < 1 + \gamma + \delta$ and assume we are in the probabilistic regime $4\gamma \delta \leq 1$. Let $x^1,\dots,x^k$ in $\mathbb{Z} \times [n]$ and any $\mathbf{f}^1,\dots,\mathbf{f}^k \in \{0,\mathbf{e}^1,\mathbf{e}^2,-\mathbf{e}^2\}$. Suppose $\prod_{i=1}^k T_{x^i}^{\mathbf{f}^i} = \sum_{\kappa \in \mathcal{K}} a_\kappa \prod_{i=1}^{r_\kappa} R_{x^{\kappa,i}}^{\mathbf{f}^{\kappa,i}}$. 
Then we have the distributional convergence
\begin{align*}
&\mathcal{P}_{{\theta_2},n} \left( \bigcap_{i=1}^k \{ \sigma(x^i) = x^i + \mathbf{f}^i \} \right)\\
&:= \lim_{m \to \infty} \mathcal{P}_{{\theta_2},(m,n)} \left( \bigcap_{i=1}^k \{ \sigma(x^i) = x^i + \mathbf{f}^i \} \right) = \sum_{\kappa \in \mathcal{K}} a_\kappa \other{\mathcal{P}}_{{\theta_2},n} \left[\prod_{i=1}^{r_\kappa} R_{x^{\kappa,i}}^{\mathbf{f}^{\kappa,i}}  \right],
\end{align*}
where
  \begin{align} \label{eq:dform2}
\other{\mathcal{P}}_{{\theta_2},n} \left[\prod_{i=1}^r R_{x^i}^{\mathbf{f}^i} \right] = \det_{i,j=1}^r \left[ (-1)^{\indicpower{ \mathbf{f}^i = \mathbf{e}^1 }} \gamma^{\indicpower{ \mathbf{f}^i = \mathbf{e}^2 }} \delta^{\indicpower{ \mathbf{f}^i = -\mathbf{e}^2 }} 
H_{{\theta_2},n}( x^i+\mathbf{f}^i, x^j ) \right],
\end{align}
and
\begin{align} \label{eq:dform3}
    H_{{\theta_2},n}(x,y) = &\indic{y_1 \geqslant x_1}\frac{1}{n} \sum_{w \in \mathcal{R}_{\theta_2}} \frac{w^{-(y_2 - x_2)}}{(1 + \gamma w + \delta w^{-1})^{y_1 - x_1 + 1}} \nonumber \\
    - & \indic{y_1 < x_1}\frac{1}{n} \sum_{w \in \mathcal{L}_{\theta_2}} \frac{w^{-(y_2 - x_2)}}{(1 + \gamma w + \delta w^{-1})^{y_1 - x_1 + 1}}.
\end{align}

\end{thm}

\begin{proof}
With Theorem \ref{thm:cffull0} in mind, we begin by studying the asymptotics of the correlation kernel $H_{\theta,(m,n)}(x,y)$ as $m \to \infty$. Recalling that $\alpha = 1$, we have
\begin{align} \label{eq:tilda}
\tilde{H}_{{\theta_2},n}(x,y) &\coloneqq \lim_{m \to \infty} H_{\theta,(m,n)}(x,y) \nonumber \\
&= \frac{1}{n} \sum_{w^{n}=(-1)^{{\theta_2}}}w^{ - (y_2-x_2) } \cdot \lim_{m \to \infty}  \frac{1}{m}\sum_{z^{m}=(-1)^{\theta_1}} \frac{ z^{-(y_1-x_1)}  }{ 1 +\beta z + \gamma w + \delta w^{-1} } \nonumber \\
&= \frac{1}{n} \sum_{w^{n}=(-1)^{{\theta_2}}}w^{ - (y_2-x_2) } \int_{|z|=1} \frac{1}{2 \pi \iota z} \frac{ z^{-(y_1-x_1)}  }{ 1 +\beta z + \gamma w + \delta w^{-1} },
\end{align}
where we have used that the parameters are generic for $\frac{ z^{-(y_1-x_1)}  }{ 1 +\beta z + \gamma w + \delta w^{-1} }$ to be continuous, hence Riemann integrable to justify the Riemann sum convergence. We highlight in particular that the limit in \eqref{eq:tilda} does not depend on whether $\theta_1 = 0$ or $\theta_1 = 1$. 

Using the residue theorem, we can evaluate the limiting integral exactly. The integrand has a simple pole at $z = -\frac{1 + \gamma w + \delta w^{-1}}{\beta}$ and (assuming $y_1 - x_1 \geqslant 0$) a pole of order $y_1 - x_1 + 1$ at $z = 0$. The residue at $z = -\frac{1 + \gamma w + \delta w^{-1}}{\beta}$ is $\frac{(-1)^{y_1 - x_1 +1}}{2\pi \iota} \frac{\beta^{y_1 - x_1}}{(1 + \gamma w + \delta w^{-1})^{y_1 - x_1+1}}$ and the residue at $z = 0$ is $\frac{(-1)^{y_1 - x_1}}{2\pi \iota} \frac{\beta^{y_1 - x_1}}{(1 + \gamma w + \delta w^{-1})^{y_1 - x_1 + 1}}$. So the integral evaluates as
\begin{align} \label{eq:contour}
    \oint_{C} \frac{dz}{2\pi \iota z} \frac{z^{-(y_1 - x_1)}}{1 + \beta z + \gamma w + \delta w^{-1}} &= \frac{(-\beta)^{y_1 - x_1}}{(1 + \gamma w + \delta w^{-1})^{y_1 - x_1 + 1}} \left( \indic{y_1 - x_1 \geqslant 0} - \indic{|1 + \gamma w + \delta w^{-1}| < \beta} \right).
\end{align}

Using \eqref{eq:contour} in \eqref{eq:tilda} we obtain $\tilde{H}_{{\theta_2},n}(x,y) = (-\beta)^{y_1-x_1}H_{{\theta_2},n}(x,y)$ where
\begin{align} \label{eq:dform4}
   H_{{\theta_2},n}(x,y) = &\indic{y_1 \geqslant x_1}\frac{1}{n} \sum_{w \in \mathcal{R}_{\theta_2}} \frac{w^{-(y_2 - x_2)}}{(1 + \gamma w + \delta w^{-1})^{y_1 - x_1 + 1}} \nonumber \\- &\indic{y_1 < x_1}\frac{1}{n} \sum_{w \in \mathcal{L}_{\theta_2}} \frac{w^{-(y_2 - x_2)}}{(1 + \gamma w + \delta w^{-1})^{y_1 - x_1 + 1}}.
\end{align}
We are now ready to study the large $m$ asymptotics of $ \other{\mathcal{P}}_{{\theta_2},(m,n)} \left[\prod_{i=1}^r R_{x^i}^{\mathbf{f}^i} \right]$. Using Corollary~\ref{cor:nform} to obtain the second equality below, we have 
\begin{align} \label{eq:dformb}
\other{\mathcal{P}}_{{\theta_2},n}  \left[\prod_{i=1}^r R_{x^i}^{\mathbf{f}^i} \right] &\coloneqq \lim_{m \to \infty} \other{\mathcal{P}}_{{\theta_2},(m,n)} \left[\prod_{i=1}^r R_{x^i}^{\mathbf{f}^i} \right] \nonumber \\
&= \lim_{m \to \infty} \prod_{i=1}^r K_{00,(m,n)}(x^i,x^i+\mathbf{f}^i) \det_{i,j=1}^r \left[ H_{\theta,(m,n)}( x^i+\mathbf{f}^i, x^j ) \right] \nonumber \\
&=  \prod_{i=1}^r K_{00,(m,n)}(x^i,x^i+\mathbf{f}^i) \lim_{m \to \infty} \left\{ \det_{i,j=1}^r \left[ H_{\theta,(m,n)}( x^i+\mathbf{f}^i, x^j ) \right]\right\} \nonumber \\
&= \prod_{i=1}^r K_{00,(m,n)}(x^i,x^i+\mathbf{f}^i) \det_{i,j=1}^r \left[ \tilde{H}_{{\theta_2},m_2}( x^i+\mathbf{f}^i, x^j ) \right].
\end{align}
Using $\tilde{H}_{{\theta_2},n}(x,y) = (-\beta)^{y_1-x_1}H_{{\theta_2},n}(x,y)$ and the fact that $\tilde{H}$ is a diagonal conjugation of $H_{{\theta_2},n}$, it follows that 
\begin{align} \label{eq:dformb2}
\other{\mathcal{P}}_{{\theta_2},n}  \left[\prod_{i=1}^r R_{x^i}^{\mathbf{f}^i} \right] = \prod_{i=1}^r K_{00}(x^i,x^i+\mathbf{f}^i) \prod_{i=1}^r (-\beta)^{ -\mathbf{f}^i_1} 
\det_{i,j=1}^r \left[ H_{{\theta_2},n}( x^i+\mathbf{f}^i, x^j ) \right].
\end{align}
Since $\alpha=1$, by using the definition of $K_{00}(x^i,x^i+\mathbf{f}^i)$ we can write the final result as
\begin{align} \label{eq:dformbc}
\other{\mathcal{P}}_{{\theta_2},n}  \left[\prod_{i=1}^r R_{x^i}^{\mathbf{f}^i} \right] = 
\det_{i,j=1}^r \left[ (-1)^{\indicpower{ \mathbf{f}^i = \mathbf{e}^1 }} \gamma^{\indicpower{ \mathbf{f}^i = \mathbf{e}^2 }} \delta^{\indicpower{ \mathbf{f}^i = -\mathbf{e}^2 }} 
H_{{\theta_2},n}( x^i+\mathbf{f}^i, x^j ) \right].
\end{align}
Concluding the proof using Theorem~\ref{thm:cffull1} we have
\begin{align} \label{eq:corrstructurelimit}
\mathcal{P}_{{\theta_2},n}\left( \bigcap_{i=1}^k \{ \sigma(x^i) = x^i + \mathbf{f}^i \} \right) &:= \lim_{m \to \infty} \mathcal{P}_{{\theta_2},(m,n)}\left( \bigcap_{i=1}^k \{ \sigma(x^i) = x^i + \mathbf{f}^i \} \right) \nonumber \\
&=  \sum_{\kappa \in \mathcal{K}} a_\kappa \lim_{m \to \infty} \sum_{\theta_1 \in \{0,1\} } \mu_{\theta,(m,n)} \other{P}_{\theta,(m,n)} \left[\prod_{i=1}^{r_\kappa} R_{x^{\kappa,i}}^{\mathbf{f}^{\kappa,i}}  \right].
\end{align}
Now from Lemma~\ref{lem:sgnCdetK} and recalling the definition of $\mu_{\theta,(m,n)}$ in \eqref{eq:barsum}, it follows that $\mu_{\theta,(m,n)}$ is positive for all $\theta \in \{0,1\}^2$. Note also that the limit in \eqref{eq:dformb} does not depend on whether $\theta_1 = 0$ or $\theta_1 = 1$. Using these facts together with \eqref{eq:dformbc} and $\sum_{\theta_1 \in\{0,1\}} \mu_{\theta,(m,n)} = 1$ we obtain the result.
\end{proof}

Finally, we use Lemma \ref{lem:arc} to afford a more intrinsic perspective on the limiting measure occurring in Theorem \ref{thm:m1infinity2}, namely, the result in Theorem~\ref{thm:m1infinity}.

\begin{proof}[Proof of Theorem~\ref{thm:m1infinity}]
For this proof, we split into 3 cases.

Consider first $0 < \ell < n$. Take $\theta_2 \in \{0,1\}$ such that $\theta_2 \equiv n - \ell + 1 \mod{2}$. Choose a generic $\beta$ such that the set $\{ w: w^{n} = (-1)^{{\theta_2}} , |1 + \gamma w + \delta w^{-1}| < \beta\}$ has exactly $\ell$ elements. Note that by Lemma~\ref{lem:arc}, $\mathcal{L}$ is precisely this set and $\mathcal{R} = \{ w: w^{n} = (-1)^{{\theta_2}} , |1 + \gamma w + \delta w^{-1}| \geq \beta\}$. Now recalling Theorem~\ref{thm:m1infinity2}, set $P_{\ell,n}^{\gamma, \delta} = \mathcal{P}_{\theta_2,n}$ and $\other{P}_{\ell,n}^{\gamma, \delta} = \other{\mathcal{P}}_{\theta_2,n}$ and we are done.

Now consider $\ell = 0$. Again, take $\theta_2 \in \{0,1\}$ such that $\theta_2 \equiv n - \ell + 1 \equiv n+1 \mod{2}$. Then by Lemma~\ref{lem:sgnCdetK}, when $\beta < 1 - \gamma - \delta$, the expression $C_{\theta,(m,n)}\det(K_{\theta,(m,n)})$ is positive. Hence, so is $\mu_{\theta,(m,n)}$. So we can run the same argument as in Theorem~\ref{thm:m1infinity2} and conclude in the same way as the first paragraph of this proof, but by picking an arbitrary $\beta < 1 - \gamma - \delta$.

Similarly, for $\ell = n$, take $\theta_2 = n - \ell + 1 = 1$. By Lemma~\ref{lem:sgnCdetK}, when $\beta < 1 + \gamma + \delta$, the expression $C_{\theta,(m,n)}\det(K_{\theta,(m,n)})$ is positive. Hence, so is $\mu_{\theta,(m,n)}$. So again, we can run the same argument as in Theorem~\ref{thm:m1infinity2} and conclude in the same way as the first paragraph of this proof, but by picking an arbitrary $\beta > 1 + \gamma + \delta$.
\end{proof}

It is possible to show that the probability measure $P_{\ell,n}^{\gamma, \delta}$ is supported on pure snake configurations on $\mathbb{Z} \times \{0,1,\dots,n-1\}$ whose occupation number $L(\sigma)$ is equal to $\ell$ almost surely. We omit the proof of this fact, but note that the argument is the same as the similar result in the sequel; see Lemma \ref{lem: cardinalityl}.


In fact, it is also possible to show that this probability measure describes a Markov process in a certain sense.

\begin{lemma}\label{lem:markov}
    Let $\sigma$ be a pure snake configuration generated from law $P_{\ell, n}^{\gamma, \delta}$ with $4\gamma \delta \leq 1$. For each $s \in \mathbb{Z}$, let
    \begin{align*}
        U_s = \{h \in [n]: \sigma(s,h) = (s+1,h)\}.
    \end{align*}
    Then {$(U_s)_{s \in \mathbb{Z}}$} is a Markov process with respect to $P_{\ell, n}^{\gamma, \delta}$.
\end{lemma}

We will omit the proof of this result but similarly note that it goes in much the same way as in the proof of Theorem~\ref{thm:MarkovProperty}. Briefly note that the formulae given for $P^{\gamma,\delta}_{\ell,n}$ describe the finite dimensional distributions of $(U_s)$, i.e., by Kolmogorov's extension theorem, they are sufficient to uniquely define the law of $(U_s)$ in the first place.

\subsection{\texorpdfstring{The integrable snake model on $\mathbb{Z}^2$}{The integrable snake model on Z2}} \label{subsec:m1m2infinity}
In this brief subsection, we take a subsequent $n \to \infty$ scaling limit of the integrable snake model on $\mathbb{Z} \times [n]$. In doing so, we obtain the claim of Theorem~\ref{thm:m1m2infinity}.

\begin{proof}
The proof amounts to showing that with $H_{\theta_2,n}(x,y)$ as in \eqref{eq:dform3} we have $\lim_{n \to \infty}H_{\theta_2,n}(x,y) = H_\tau^{\delta,\gamma}(x,y)$, where $H_\tau^{\delta,\gamma}(x,y)$ is as in \eqref{eq:dform6}. To do this, simply approximate the sum in \eqref{eq:dform3} using a contour integral and use Remark~\ref{rem:arc} to see that there is some value of $\beta$ such that we obtain $w_\tau$ in the limit of the integral, as desired. 

One can remark that this proof only allows us to use $\beta$ which is generic for 
all $n$. Nonetheless, this argument holds on a dense set of possible values for $\beta > 0$ and so by continuity of $\beta \mapsto t_\beta$ and $\tau \mapsto P_{\gamma,\delta}^{\tau}$, we have the required probability measure for all $0 \leq \tau \leq 1$.
\end{proof}

In the special case $\mathbf{f}^i = \mathbf{e}^1$ for each $i$, using \eqref{eq:wordyspecial} in the setting of Theorem \ref{thm:m1m2infinity} we have
\begin{align}\label{eq:Z^2rightmove}
\mathbf{P} \left( \bigcap_{i=1}^k \{ \sigma(x^i) = x^i + \mathbf{e}^1 \} \right) = \det_{i,j=1}^k \left( - \int_{\overline{w_{\tau}} }^{w_{\tau}} \frac{\mathrm{d}w}{2\pi \iota w } \frac{ w^{ - (x^j_2-x^i_2)}}{ (1 + \gamma w + \delta w^{-1})^{ x^j_1 - x^i_1}} \right).
\end{align}
Of course, setting $\delta = 0$, we obtain \eqref{eq:lozcorr} from the introduction.

As a final remark, we again note that formula \eqref{eq:Z^2rightmove} can be thought of as describing the finite dimensional distributions of the process $(U_s)_{s \in \mathbb{Z}}$ given by $$U_s = \{h \in \mathbb{Z}: \sigma(s,h) = (s+1,h)\}.$$ So this formula genuinely and uniquely defines a probability measure for a discrete Markov process, by looking at the positions of the rightward moves.

\section{The Exclusion Process on the Ring} \label{sec:noncoll}

\subsection{Preliminary results}

In this section we study the asymptotic behaviour of the probability measure $P^{\gamma, \delta}_{\ell,n}$ (abbreviated to $P_{\ell,n}$) occurring in Theorem \ref{thm:m1infinity} under the scaling regime 
\begin{align*}
\gamma = \varepsilon T \qquad \text{and} \qquad \delta = \varepsilon T',
\end{align*}
as $\varepsilon \downarrow 0$. In order to see the correlations on the correct scale, it is also necessary to rescale the horizontal distances between points. Thus, we will consider scaled points $x^{\varepsilon,i} = (\lfloor s^i/\varepsilon \rfloor, h^i)$, where $s^i \in \mathbb{R}$ and $h^i \in [n]$. Here $\lfloor s^i/ \varepsilon \rfloor$ denotes the largest integer not greater than $s^i/\varepsilon$. 

The following result describes the asymptotics of the correlations under this regime.

\begin{proposition} \label{prop:longer}
Let $(s^1,h^1),\dots,(s^k,h^k)$ be elements of $\mathbb{R} \times [n]$. Let $\mathbf{f}^1,\dots,\mathbf{f}^k$ be elements of $\{0,\mathbf{e}^1,\mathbf{e}^2,-\mathbf{e}^2\}$. Let $N := \# \{ i : \mathbf{f}^i = \mathbf{e}^2 \text{ or }\mathbf{f}^i = -\mathbf{e}^2 \}$. Let $x^{\varepsilon,i} = ( \lfloor s^i/\varepsilon \rfloor, h^i)$. Then we have 
\begin{align*}
&\lim_{\varepsilon \downarrow 0} \varepsilon^{-N} P_{\ell,n} \left( \bigcap_{i=1}^k \{ \sigma (x^{\varepsilon,i})  = x^{\varepsilon,i}  + \mathbf{f}^i \} \right)\\
&= \det_{i,j=1}^r \left[ (-1)^{\indicpower{ \mathbf{f}^i = \mathbf{e}^1 } }  T^{\indicpower{ \mathbf{f}^i = \mathbf{e}^2 }}{T'}^{\indicpower{ \mathbf{f}^i = -\mathbf{e}^2 }} 
H_{\ell,n}(\mathbf{f}^i, (s^i,h^i), (s^j,h^j) )\right],
\end{align*}
where
\begin{align} \label{eq:dform3bii2}
    &H_{\ell,n}(\mathbf{f},(s,h),(s',h')) \nonumber \\
    &=
    \begin{cases}
         \frac{1}{n} \sum_{w \in \mathcal{R}_{\ell,n}} w^{-(h' - h + \indicpower{\mathbf{f} = -\mathbf{e}^2} -\indicpower{\mathbf{f} = \mathbf{e}^2})} e^{ - (Tw + T'w^{-1})(s'-s)} \qquad &\text{if $s' \succeq s + \mathbf{f}$,}\\
-\frac{1}{n} \sum_{w \in \mathcal{L}_{\ell,n}} w^{-(h' - h + \indicpower{\mathbf{f} = -\mathbf{e}^2} -\indicpower{\mathbf{f} = \mathbf{e}^2})} e^{ - (Tw + T'w^{-1})(s'-s)}  &\text{if $s' \prec s + \mathbf{f}$}.         
    \end{cases}
    \end{align}
    Here we have $\{ s' \succeq s + \mathbf{f}\} = \{ s' > s \} \cup \{ s' = s, \mathbf{f} \neq \mathbf{e}^1 \}$ and $\{ s' \prec s + \mathbf{f}\}$ is simply the complement of this event. The sets $\mathcal{R}_{\ell,n}$ and $\mathcal{L}_{\ell,n}$ are $\mathcal{R}$ and $\mathcal{L}$ in Theorem \ref{thm:m1infinity}.
\end{proposition}

\begin{proof}
Suppose we have a word decomposition $\prod_{i=1}^k T_{x^i}^{\mathbf{f}^i} = \sum_{\kappa \in \mathcal{K}} a_\kappa \prod_{i=1}^{r_\kappa} R_{x^{\kappa,i}}^{\mathbf{f}^{\kappa,i}}$. Then using $\gamma = \varepsilon T$ and $\delta = \varepsilon T'$ together with Theorem \ref{thm:m1infinity} we have
\begin{align} \label{eq:longer1}
P_{\ell,n} \left( \bigcap_{i=1}^k \{ \sigma(x^i) = x^i + \mathbf{f}^i \} \right) := \sum_{\kappa \in \mathcal{K}} a_\kappa \other{P}_{\ell,n} \left[\prod_{i=1}^{r_\kappa} R_{x^{\kappa,i}}^{\mathbf{f}^{\kappa,i}}  \right],
\end{align}


where
  \begin{align} \label{eq:longer2}
 \other{P}_{\ell,n}  \left[\prod_{i=1}^r R_{x^i}^{\mathbf{f}^i} \right] = \varepsilon^{ \# \{ i : \mathbf{f}^i = \pm \mathbf{e}^2 \} } \det_{i,j=1}^r \left[ (-1)^{\indicpower{ \mathbf{f}^i = \mathbf{e}^1 }} T^{\indicpower{ \mathbf{f}^i = \mathbf{e}^2 }} {T'}^{\indicpower{ \mathbf{f}^i = -\mathbf{e}^2 }} 
H^{\gamma,\delta}_{\ell,n}( x^i+\mathbf{f}^i, x^j ) \right],
\end{align}
and where $H_{\ell,n}^{\gamma,\delta}$ is given in \eqref{eq:dform3bi}. Since the kernel $H_{\ell,n}^{\gamma,\delta}:(\mathbb{Z} \times [n]) \times (\mathbb{Z} \times [n]) \to \mathbb{C}$ is bounded, it follows from \eqref{eq:longer2} that the sum in \eqref{eq:longer1} gets its leading order contribution as $\varepsilon \downarrow 0$ from words $\prod_{i=1}^{r_\kappa} R_{x^{\kappa,i}}^{\mathbf{f}^{\kappa,i}}$ that minimise $N_\kappa := \# \{ 1 \leq i \leq r_\kappa : \mathbf{f}^{\kappa,i} = \pm \mathbf{e}^2 \}$. Using \eqref{eq:ev1b}-\eqref{eq:ev4b}, it is clear that we can write
\begin{align*}
    \prod_{i=1}^k T_{x^i}^{\mathbf{f}^i} = \prod_{i=1}^k R_{x^i}^{\mathbf{f}^i} + \sum_{\kappa \in \mathcal{K}'} a_\kappa \prod_{i=1}^{r_\kappa} R_{x^{\kappa,i}}^{\mathbf{f}^{\kappa,i}},
\end{align*}
where for every $\kappa \in \mathcal{K}'$, $N_\kappa > N := \{ 1 \leq i \leq k : \mathbf{f}^i = \pm \mathbf{e}^2 \}$. In particular, by the boundedness of $H^{\gamma,\delta}_{\ell,n}$ it follows that as $\varepsilon \downarrow 0$ we have 
\begin{align*}
&P_{\ell,n} \left( \bigcap_{i=1}^k \{ \sigma(x^i) = x^i + \mathbf{f}^i \} \right)\\
&= \varepsilon^N\det_{i,j=1}^r \left[ (-1)^{\indicpower{ \mathbf{f}^i = \mathbf{e}^1 }} T^{\indicpower{ \mathbf{f}^i = \mathbf{e}^2 }} {T'}^{\indicpower{ \mathbf{f}^i = -\mathbf{e}^2 }} 
H^{\gamma,\delta}_{\ell,n}( x^{\varepsilon,i} +\mathbf{f}^i, x^{\varepsilon,i}  ) \right] + O(\varepsilon^{N+1}).
\end{align*}
In particular, it follows that 
\begin{align} \label{eq:longer3}
&\lim_{\varepsilon \downarrow 0} \varepsilon^{-N} P_{\ell,n} \left( \bigcap_{i=1}^k \{ \sigma (x^{\varepsilon,i})  = x^{\varepsilon,i}  + \mathbf{f}^i \} \right) \\
&= \det_{i,j=1}^r \left[ (-1)^{\indicpower{ \mathbf{f}^i = \mathbf{e}^1 } }  T^{\indicpower{ \mathbf{f}^i = \mathbf{e}^2 }}{T'}^{\indicpower{ \mathbf{f}^i = -\mathbf{e}^2 }} 
\lim_{\varepsilon \downarrow 0} H^{\gamma,\delta}_{\ell,n}(  x^{\varepsilon,i}  + \mathbf{f}^i , x^{\varepsilon,j} ) \right] . \nonumber
\end{align}
Thus it remains to study $\lim_{\varepsilon \downarrow 0} H^{\gamma,\delta}_{\ell,n}( (\lfloor s/\varepsilon \rfloor , h ) + \mathbf{f} , (\lfloor s'/\varepsilon \rfloor , h') )$. In this direction, we note that if $y = y^\varepsilon \coloneqq(\lfloor s'/\varepsilon \rfloor , h')$ and $x = x^\varepsilon \coloneqq (\lfloor s/\varepsilon \rfloor , h ) + \mathbf{f}$ in Equation~\eqref{eq:dform3bi}, we have 
\begin{align} \label{eq:longer4}
\lim_{\varepsilon \downarrow 0} \frac{w^{-(y^\varepsilon_2 - x^\varepsilon_2)}}{(1 + \varepsilon T w + \varepsilon T' w^{-1})^{y^\varepsilon_1 - x^\varepsilon_1 + 1}} =w^{-(h' - h + \indicpower{\mathbf{f} = -\mathbf{e}^2} -\indicpower{\mathbf{f} = \mathbf{e}^2})} e^{ - (Tw + Tw^{-1})(s'-s)}.
\end{align}
We note also that we have the limit of indicator functions 
$\lim_{\varepsilon \to 0} \indic{\lfloor s'/\varepsilon \rfloor \geq \lfloor s/\varepsilon \rfloor + \mathbf{f}} = \indic{ s' \succeq s + \mathbf{f} }$. Thus, using \eqref{eq:longer4} in \eqref{eq:longer3}, we obtain the result.
\end{proof}

The previous results let us derive a continuous-time stochastic process occurring as a limit. Recall from Lemma~\ref{lem:markov} that we can define $(U_s)_{s \in \mathbb{Z}}$ to be a discrete-time Markov process by generating $\sigma$ with law $P_{\ell,n}$ and defining
\begin{align*}
U_s := \{ h \in [n] : \sigma(s,h) = (s+1,h) \}.
\end{align*}
Given $\varepsilon > 0$, consider the continuous time stochastic process $(X^\varepsilon_t)_{t \in \mathbb{R}}$ defined by setting
\begin{align*}
X^\varepsilon_t := U_{\lfloor t/\varepsilon \rfloor},
\end{align*}
where $(U_s)_{s \in \mathbb{Z}}$ is governed by the probability measure $P_{\ell,n}$ with $\gamma = \varepsilon T$ and $\delta = \varepsilon T'$. The previous result states that the stochastic process $X^\varepsilon_t$ converges in distribution as $\varepsilon \downarrow 0$. Let us denote the limiting stochastic process by $(X_t)_{t\in \mathbb{R}}$, and note that 
it takes values in the collection of subsets of $[n]$. In fact, the cardinality of this set is constant and, noting that our limiting process is right-continuous, every discontinuity time $t$ of $X_t$ takes the form
\begin{align*}
X_t = X_{t-} \cup \{h \pm 1\} - \{h\},
\end{align*}
(considered with looping, a convention we will, in future, use without comment) for some $h \in X_{t-}$ and some $h\pm 1 \notin X_{t-}$. In other words, this limiting stochastic process has dynamics similar to the asymmetric exclusion process: we have a fixed number of particles occupying distinct positions in $[n]$ considered as a ring and these particles may undertake jumps to unoccupied neighbouring sites at certain rates.

Proposition \ref{prop:longer} provides an explicit description of the correlations of $(X_t)_{t \in \mathbb{R}}$. To describe these, for $h \in [n]$, $t \in \mathbb{R}$, $\mathbf{f} \in \{0,\mathbf{e}^1,\mathbf{e}^2,-\mathbf{e}^2\}$ define the events
\begin{align*}
E(t,h,0) &:= \{ h \notin X_t \},\\
E(t,h,\mathbf{e}^1) &:= \{ h \in X_t\},\\
E(t,h,\pm \mathbf{e}^2) &:= \{ \text{A particle jumps from $h$ to $h \pm 1$ in $[t,t+\mathrm{d}t)$} \}.
\end{align*}
If at time $t$ a particle jumps from $h$ to $h+1$ (resp.\ from $h$ to $h-1$), we say that the location $(t,h)$ in $\mathbb{R} \times [n]$ is an upbead (resp.\ downbead). Then $E(t,h,\mathbf{e}^2)$ (resp.\ $E(t,h,-\mathbf{e}^2)$) denotes the event that the interval $[t,t+\mathrm{d}t) \times \{h\}$ contains an upbead (resp.\ downbead).

Then, defining $P^{\ell,n}$ as the law of $(X_t)_{t \in \mathbb{R}}$, Proposition \ref{prop:longer} gives 
\begin{align} \label{eq:light2}
& P^{\ell,n} \left( \bigcap_{i=1}^k E(t^i,h^i,\mathbf{f}^i) \right) \nonumber\\
&= \det_{i,j=1}^r \left[ (-1)^{\indicpower{ \mathbf{f}^i = \mathbf{e}^1 } }  T^{\indicpower{ \mathbf{f}^i = \mathbf{e}^2 }}{T'}^{\indicpower{ \mathbf{f}^i = -\mathbf{e}^2 }} 
H_{\ell,n}(\mathbf{f}^i, (t^i,h^i), (t^j,h^j)) \right] \prod_{i : \mathbf{f}^i= \pm \mathbf{e}^2 } \mathrm{d}t^i.
\end{align}
In the next section, we begin discussing the basic properties of $P^{\ell,n}$ in earnest.

\subsection{\texorpdfstring{Basic properties of $P^{\ell,n}$}{Basic properties of Pn,l}}

In this section, we study basic properties of the determinantal process $(X_t)_{t \in \mathbb{R}}$ under the probability laws $P^{\ell,n}$. We seek to show that the process is Markov and calculate the stationary distribution of the process, its transition rates and some basic densities. The calculations in this subsection are generalisations of work in Section 6 of \cite{johnston}.

Our following result says that the size of the set $X_t$ (which has constant cardinality as $t$ ranges across $\mathbb{R}$) is almost surely equal to $\ell$ under $P^{\ell,n}$.

\begin{lemma}\label{lem: cardinalityl}
Let $(X_t)_{t \in \mathbb{R}}$ have law $P^{\ell,n}$. Then for any $t \in \mathbb{R}$ we have 
\begin{align*}
P^{\ell,n}( \# X_t = p ) = \indic{\ell = p}.
\end{align*}
 \end{lemma}
\begin{proof}
Fix \(t \in \mathbb{R}\) and specialise Equation~\eqref{eq:light2} to the case $1 \leq i \leq p$, with $t^i = t$ and $\mathbf{f}^i = \mathbf{e}^1$ for all $i$. Then, for any distinct $h^1, \dots, h^p \in [n]$,
\begin{align*}
    P^{\ell,n}(X_t \supset \{h_1,\dots,h_p\}) &= \det_{j,k = 1}^p H^{\ell,n}(0,h_k - h_j)\\
        &= \det_{j,k = 1}^p \frac{1}{n}\sum_{z \in \mathcal{L}_{\ell,n}} z^{-h_k + h_j}.
\end{align*}
In other words, $X_t$ is a determinantal point process with kernel $\tilde{H}:[n]\times[n] \to \mathbb{C}$ given by
\begin{align*}
    \tilde{H}(j,k) = \frac{1}{n}\sum_{z \in \mathcal{L}_{\ell,n}} z^{-k + j}.
\end{align*}
It is a simple check that $\tilde{H}$ is Hermitian, and that $\tilde{H}^2 = \tilde{H}$. But it is known (e.g. Lemma~17 of \cite{HKPV}) that for such kernels, the number of points in the process is $\operatorname{tr}(\tilde{H}) = \ell$ almost surely.
\end{proof}

Enumerate the roots $\mathcal{L}_{\ell,n} = \{z_1,\dots,z_\ell\}$ and $\mathcal{R}_{\ell,n} = \{z_{\ell + 1},\dots,z_n\}$. For $\mathbf{h} = (h_1,\dots,h_\ell)$ an $\ell$-tuple in $[n]$, an important matrix to understand in the upcoming discussion is
\begin{equation}\label{defA}
    A_{j,k}^{\mathbf{h}} \coloneqq z_k^{h_j}, \hspace{0.5cm}\text{ for }\hspace{0.5cm} 1 \leq j,k \leq \ell.
\end{equation}
In particular, we have the following key result.
\begin{lemma}\label{lem: vandermonde}
    Write $[n]^{(\ell)}$ for the set of $\ell$-tuples of distinct elements of $[n]$. Also define the function $\Delta \colon [n]^{(\ell)} \to [0,\infty)$ by 
    \begin{equation*}
        \Delta(\mathbf{h}) \coloneqq \prod_{1 \leq j < k \leq \ell}\left|\omega_n^{h_k} - \omega_n^{h_j}\right|,
    \end{equation*}
    where $\omega_n \coloneqq e^{\frac{2\pi \iota}{n}}$ and $\mathbf{h} = (h_1,\dots,h_\ell)$. Then,
    \begin{equation*}
        \det_{j,k = 1}^\ell A_{j,k}^{\mathbf{h}} = \sgn(\sigma_\mathbf{h})(-1)^{\sum_{r = 1}^\ell h_r}\iota^{\ell(\ell-1)/2}\Delta(\mathbf{h}),
    \end{equation*}
    where $\sigma_{\mathbf{h}}$ is the unique permutation ordering the elements of $\mathbf{h}$.
\end{lemma}
\begin{proof}
    We begin by noting that we can take $z_k = \omega_n^{\frac{n-\ell + 1}{2} + k-1}$. Let us also define $\xi_j = \omega_n^{h_j}$. Then 
    \begin{align*}
        \det_{j,k = 1}^{\ell} A_{j,k}^{\mathbf{h}} &= \det_{j,k = 1}^{\ell} \left(\omega_n^{\frac{n-\ell + 1}{2} + k-1}\right)^{h_j}\\
        &= \omega_n^{\frac{n-\ell + 1}{2} \sum_{r=1}^{\ell}h_r}\det_{j,k = 1}^{\ell} \xi_j^{k-1},
    \end{align*}
    and by the Vandermonde determinant formula, we can write the right-hand side of this expression to obtain that 
    \begin{align*}
        \det_{j,k = 1}^{\ell} A_{j,k}^{\mathbf{h}} &= \omega_n^{\frac{n-\ell + 1}{2} \sum_{r=1}^{\ell}h_r} \prod_{1 \leq j < k \leq \ell}\left(\xi_k - \xi_j\right).
    \end{align*}
    Finally, noting that
    \begin{align*}
        \xi_k - \xi_j &= e^{\frac{2\pi \iota}{n}h_k} - e^{\frac{2\pi \iota}{n}h_j}\\
        &= (-1)^{\indicpower{h_j > h_k}}\iota e^{\frac{\pi \iota}{n}(h_k + h_j)}\left|\xi_k - \xi_j\right|,
    \end{align*}
    where we have used that $e^{\frac{2\pi \iota}{n}h_k} - e^{\frac{2\pi \iota}{n}h_j} = (-1)^{\indicpower{h_j > h_k}} 2\iota\sin(\pi (h_k - h_j)/n)e^{\frac{2\pi\iota}{n}(h_j + h_k)}$,
    and noting that $\sgn(\sigma_{\mathbf{h}}) = \prod_{1 \leq j < k \leq \ell} (-1)^{\indicpower{h_j > h_k}}$ gives the required result upon simplifying.
\end{proof}
\begin{cor}\label{cor: stateprobability}
    For any $t \in \mathbb{R}$ and $h_1,\dots,h_\ell$ distinct elements of $[n]$ and $\mathbf{h} = (h_1,\dots,h_\ell)$, we have
    \begin{equation*}
        P^{\ell,n}(X_t = \{h_1,\dots,h_\ell\}) = n^{-\ell}\Delta(\mathbf{h})^2.
    \end{equation*}
\end{cor}
\begin{proof}
    By Lemma \ref{lem: cardinalityl}, the events $\{X_t = \{h_1,\dots,h_\ell\}\}$ and $\{X_t \supset \{h_1,\dots,h_\ell\}\}$ are equivalent in $P^{\ell,n}$-measure. Hence, making use of \eqref{eq:light2},
    \begin{align*}
        P^{\ell,n}(X_t = \{h_1,\dots,h_\ell\}) &= P^{\ell,n}(X_t \supset \{h_1,\dots,h_\ell\})\\
        &= \det_{j,k = 1}^\ell H^{\ell,n}(\mathbf{e}_1,(t,h_j),(t,h_k))\\
        &= \det_{j,k = 1}^\ell \frac{1}{n}\sum_{z \in \mathcal{L}_{\ell,n}} z^{-h_k + h_j}\\
        &= n^{-\ell}\left(\det_{j,r = 1}^{\ell}z_r^{h_j}\right)\overline{\left(\det_{r,k = 1}^{\ell}z_r^{h_k}\right)}\\
        &= n^{-\ell}\left(\det_{j,r = 1}^{\ell}A_{j,r}^{\mathbf{h}}\right)\overline{\left(\det_{r,k = 1}^{\ell}A_{r,k}^{\mathbf{h}}\right)},
    \end{align*}
    which gives the result upon using Lemma~\ref{lem: vandermonde}.
\end{proof}
We briefly perform some simple calculations as sanity checks. E.g. we can find the probability any particular position is occupied,
\begin{equation*}
    P^{\ell,n}(x \text{ is occupied}) = -H_{\ell,n}(\mathbf{e}^1,0,0) = \frac{\ell}{n},
\end{equation*}
which reflects the fact that $\ell$ of the $n$ strings are occupied at time $t$. 

We also note how the determinantal formulas give rise to translation invariance in $t$, which shows that the process is stationary.

We can similarly compute the density of up beads and down beads on each string. E.g.\ for up beads, the probability that $\mathrm{d}x$ contains an up bead is given by 
\begin{equation*}
   T H_{\ell,n}(\mathbf{e}^2,0,0) \; \mathrm{d}x = \frac{T}{n}\sum_{z\in \mathcal{R}_{\ell,n}}z \; \mathrm{d}x = -\frac{T}{n} \sum_{r = 1}^{\ell} z_r \; \mathrm{d}x = \frac{T}{n} \frac{\sin\left(\frac{\pi \ell}{n}\right)}{\sin\left(\frac{\pi}{n}\right)} \; \mathrm{d}x.
\end{equation*}
Similarly, we get $\frac{T'}{n} \frac{\sin\left(\frac{\pi \ell}{n}\right)}{\sin\left(\frac{\pi}{n}\right)} \; \mathrm{d}x$ for down beads. 

Now we show that the process $(X_t)_{t\in\mathbb{R}}$ is Markov - that is to say that for $t_1 < \cdots < t_{k} < t_{k+1}$ and $x_1,\dots,x_{k+1} \in [n]^{(\ell)}$,
\begin{equation}\label{eq:MarkovProperty}
    P^{\ell,n}\left(X_{t_{k+1}} = x_{{k+1}} \mid X_{t_{k}} = x_{{k}},\dots, X_{t_{1}} = x_{{1}} \right) = P^{\ell,n}\left(X_{t_{k+1}} = x_{{k+1}} \mid X_{t_{k}} = x_{{k}} \right).
\end{equation} This work can be replicated to prove Lemma~\ref{lem:markov}. First, let us make some definitions for notational ease.

It will prove only necessary to consider the formula in \eqref{eq:light2} for events $E(t^i,h^i,\mathbf{f}^i)$ where $\mathbf{f}^i = \mathbf{e}^1$. So the aforementioned formula simplifies to
\begin{align} \label{eq:light3}
 P^{\ell,n} \left( \bigcap_{i=1}^k E(t^i,h^i,\mathbf{e}^1) \right) = \det_{i,j=1}^r \left[ -H_{\ell,n}\left((t^i,h^i), (t^j,h^j)\right) \right],
\end{align}
where 
\begin{align} \label{eq:dform3bii3}
    H_{\ell,n}((t,h),(t',h')) =
    \begin{cases}
         \frac{1}{n} \sum_{w \in \mathcal{R}_{\ell,n}} w^{-(h' - h)} e^{ - T(w)(t'-t)} \qquad &\text{if $t' > t$},\\
-\frac{1}{n} \sum_{w \in \mathcal{L}_{\ell,n}} w^{-(h' - h)} e^{ - T(w)(t'-t)}  &\text{if $t' \leq t$},      
    \end{cases}
    \end{align}
    and $T(w) = Tw + T'w^{-1}$.

Define the events \begin{align*}
    C(x_{k+1}) &= \{X_{t_{k+1}} = x_{{k+1}}\},\\
    B(x_{k}) &= \{X_{t_{k}} = x_{{k}}\}, \\ A(x_1,\dots,x_{k-1}) &= \{X_{t_{k-1}} = x_{{k-1}},\dots, X_{t_{1}} = x_{{1}}\}.
\end{align*} Note that we have suppressed the time variable in the events, since they will remain fixed in our proof of the Markov property. Then we have the following lemma.
\begin{lemma}\label{lem: MarkovDependence}
    To show \eqref{eq:MarkovProperty}, it is enough to show that, for all $t_1 < \cdots < t_{k+1}$ and $x_1,\dots, x_{k+1} \in [n]^{(\ell)}$,
    \begin{equation}\label{eq:MarkovDependence}
    P^{\ell,n}\left(C(x_{k+1}),B(x_{k}),A(x_1,\dots,x_{k-1})\right) = P^{\ell,n}(C(x_{k+1}),B(x_{k}))f(x_1,\dots,x_k)
    \end{equation}
\end{lemma}
where $f$ is an arbitrary function, independent of $x_{k+1}$.

\begin{proof}
    From \eqref{eq:MarkovDependence}, we have that
    \begin{align}\label{eq: conditioning}
        P^{\ell,n}\left(C(x_{k+1})\mid B(x_{k}),A(x_1,\dots,x_{k-1})\right) &= \frac{P^{\ell,n}\left(C(x_{k+1}), B(x_{k})\right)}{P^{\ell,n}\left(B(x_{k}), A(x_1,\dots,x_{k-1})\right)}f(x_1,\dots,x_k)\nonumber\\
        &= \frac{P^{\ell,n}\left(C(x_{k+1})\mid B(x_{k})\right)}{P^{\ell,n}\left(A(x_1,\dots,x_{k-1})\mid B(x_{k})\right)}f(x_1,\dots,x_k)\nonumber\\
        &= P^{\ell,n}\left(C(x_{k+1})\mid B(x_{k})\right)g(x_1,\dots,x_k)
    \end{align}
    where $g(x_1,\dots,x_k) = \frac{f(x_1,\dots,x_k)}{P^{\ell,n}\left(A(x_1,\dots,x_{k-1})\mid B(x_{k})\right)}$. But then summing \eqref{eq: conditioning} over $x_{k+1} \in [n]^{(\ell)}$ gives that $g(x_1,\dots,x_k) \equiv 1$ and hence gives the required result.
\end{proof}

Lemma~\ref{lem: MarkovDependence}'s utility is apparent for our determinantal process.

\begin{thm}\label{thm:MarkovProperty}
    $(X_t)_{t \in \mathbb{R}}$ satisfies the Markov Property.
\end{thm}
\begin{proof}
    For $1 \leq r \leq k+1$, write $x_{r} = (x_{r,1},\dots,x_{r,\ell})$ where $x_{r,1} < \cdots < x_{r,\ell}$. Using \eqref{eq:light3} with \eqref{eq:MarkovDependence}, we see that 
    \begin{equation}\label{eq:bigmatrix}
        P^{\ell,n}\left(C(x_{k+1}),B(x_{k}),A(x_1,\dots,x_{k-1})\right) = \det\begin{bmatrix}
            \mathcal{H}_{k+1,k+1} & \mathcal{H}_{k+1,k} & \cdots &\mathcal{H}_{k+1,1}\\
            \mathcal{H}_{k,k+1} & \mathcal{H}_{k,k} & \cdots &\mathcal{H}_{k,1}\\
            \vdots&\vdots&\ddots&\vdots\\
            \mathcal{H}_{1,k+1} & \mathcal{H}_{1,k} & \cdots & \mathcal{H}_{1,1}
        \end{bmatrix} \eqqcolon \det(\mathbf{H}),
    \end{equation}
    where each $\mathcal{H}_{r,s}$ is an $\ell \times \ell$ block matrix, given by 
    \begin{equation}\label{eq:mathcalHprecise}
        (\mathcal{H}_{r,s})_{i,j} = -H_{\ell,n}\left((t_r,x_{r,i}),(t_s,x_{s,j})\right)
    \end{equation}
    Now observing the top-left of the matrix $\mathbf{H}$ in \eqref{eq:bigmatrix}, we note the appearance of the sub-matrix $\begin{bmatrix}
            \mathcal{H}_{k+1,k+1} & \mathcal{H}_{k+1,k}\\
            \mathcal{H}_{k,k+1} & \mathcal{H}_{k,k}
        \end{bmatrix}$, which has determinant $P^{\ell,n}\left(C(x_{k+1}),B(x_{k})\right)$. Thus, it would suffice to show that the sub-matrix
        \begin{equation}\label{eq:defHtilde}
            \tilde{\mathbf{H}} = \begin{bmatrix}
            
            \mathcal{H}_{k,k+1} & \mathcal{H}_{k,k}\\
            \vdots&\vdots\\
            \mathcal{H}_{1,k+1} & \mathcal{H}_{1,k}
        \end{bmatrix},
        \end{equation}
        is a matrix of rank $\ell$ and that the first $\ell$ rows are linearly independent. To see this, note that this would imply that the last $(k-1)\ell$ rows in $\tilde{\mathbf{H}}$ can be written as a linear combination of the first $\ell$ rows. So there exists row operations involving subtracting multiples of the first $\ell$ rows of $\tilde{\mathbf{H}}$ from the last $(k-1)\ell$ rows that sets all the entries of these last $(k-1)\ell$ rows to $0$. By performing the corresponding row operations in $\mathbf{H}$, the determinant in \eqref{eq:bigmatrix} can be equivalently written as
        \begin{align*}
            P^{\ell,n}\left(C(x_{k+1}),B(x_{k}),A(x_1,\dots,x_{k-1})\right) &= \det\left[
            \begin{array}{cc|c}
            \mathcal{H}_{k+1,k+1} & \mathcal{H}_{k+1,k} & \cdot \\
            \mathcal{H}_{k,k+1} & \mathcal{H}_{k,k} & \cdot \\
            \hline
            \multicolumn{2}{c|}{\mathbf{0}} & \mathcal{G} \\
            \end{array}
            \right]\\ &= \det\begin{bmatrix}
            \mathcal{H}_{k+1,k+1} & \mathcal{H}_{k+1,k}\\
            \mathcal{H}_{k,k+1} & \mathcal{H}_{k,k}
        \end{bmatrix} \det(\mathcal{G})
        \end{align*}
        where $\mathcal{G}$ is the resulting $(k-1)\ell \times (k-1)\ell$ block matrix formed by these row operations. By construction, $\mathcal{G}$ is independent of $x_{k+1}$ and so by Lemma~\ref{lem: MarkovDependence}, the proof would be complete.

        Now let us re-examine $\tilde{\mathbf{H}}$ in \eqref{eq:defHtilde}. Since $\det(\mathcal{H}_{k,k}) = P^{\ell,n}(B(x_k)) \neq 0$, the first $\ell$ rows are certainly linearly independent. Furthermore, from \eqref{eq:mathcalHprecise} and \eqref{eq:dform3bii3}, we can explicitly write down that, for $1 \leq r \leq k$, $s = k$ or $k+1$,
        \begin{equation}
            (\mathcal{H}_{r,s})_{i,j} = \frac{1}{n}\sum_{w \in \mathcal{L}_{\ell,n}}w^{x_{r,i} - x_{s,j}}e^{T(w)(t_r-t_s)}
        \end{equation}
        where we have used that $t_r \leq t_s$ for our selection of $r$ and $s$. But then, enumerating the elements of $\mathcal{L}_{\ell,n} = \{w_1,\dots,w_\ell\}$ it follows that $\tilde{\mathbf{H}}$ admits the factorisation
        $\tilde{\mathbf{H}} = \tilde{\mathbf{A}} \tilde{\mathbf{B}}^{\mathbf{T}}$, where $\tilde{\mathbf{A}}$ has dimensions $k\ell \times \ell$, $\tilde{\mathbf{B}}$ has dimensions $2\ell \times \ell$ and are given by
        \begin{equation}
            \tilde{\mathbf{A}} = \begin{bmatrix}
                \frac{1}{n}w_1^{x_{k,1}}e^{T(w_1)t_k}& \cdots & \frac{1}{n}w_\ell^{x_{k,1}}e^{T(w_\ell)t_k}\\
                \frac{1}{n}w_1^{x_{k,2}}e^{T(w_1)t_k}& \cdots & \frac{1}{n}w_\ell^{x_{k,2}}e^{T(w_\ell)t_k}\\
                \vdots & \vdots & \vdots\\
                \frac{1}{n}w_1^{x_{1,\ell}}e^{T(w_1)t_1}& \cdots & \frac{1}{n}w_\ell^{x_{1,\ell}}e^{T(w_\ell)t_1}\\
            \end{bmatrix}
        \end{equation}
        and
        \begin{equation}
            \tilde{\mathbf{B}} = \begin{bmatrix}
                w_1^{-x_{k+1,1}}e^{-T(w_1)t_{k+1}}& \cdots & w_\ell^{-x_{k+1,1}}e^{-T(w_\ell)t_{k+1}}\\
                w_1^{-x_{k+1,2}}e^{-T(w_1)t_{k+1}}& \cdots & w_\ell^{-x_{k+1,2}}e^{-T(w_\ell)t_{k+1}}\\
                \vdots & \vdots & \vdots\\
                w_1^{-x_{k,\ell}}e^{-T(w_1)t_{k}}& \cdots & w_\ell^{-x_{k,\ell}}e^{-T(w_\ell)t_{k}}\\
            \end{bmatrix}.
        \end{equation}
        It immediately follows that $\tilde{\mathbf{H}}$ has rank $\ell$.
\end{proof}

We mention that the connection between the Markov property (or rather, the Gibbs property) and determinantal correlations are studied in some detail for a certain model considered in Borodin and Shlosman \cite{BS}.

Now we investigate the transition rates of the Markov process $(X_t)$. 


\begin{lemma}\label{lem: transitionrates}
    Let $\mathbf{h} = (h_1,\dots,h_\ell)$ and $\mathbf{h}' = (h_1',\dots,h_\ell')$ be elements of $[n]^{(\ell)}$ with 
    \begin{equation*}
        h_j' \equiv h_j \bmod{n} \text{ for } j\leq \ell-1 \hspace{0.5cm} \text{and} \hspace{0.5cm} h_\ell' \equiv h_\ell + 1 \bmod{n}
    \end{equation*}
    Then,
    \begin{equation}\label{transitionrateprobup}
        \lim_{t \downarrow 0} \frac{1}{t} P^{\ell,n}(X_t = \{h_1',\dots,h_\ell'\}| X_0 = \{h_1,\dots,h_\ell\}) = T\frac{\Delta(\mathbf{h}')}{\Delta(\mathbf{h})}.
    \end{equation}
    If instead we suppose that $\mathbf{h}$ and $\mathbf{h}'$ are such that
    \begin{equation*}
        h_j' \equiv h_j \bmod{n} \text{ for } j\leq \ell-1 \hspace{0.5cm} \text{and} \hspace{0.5cm} h_\ell' \equiv h_\ell - 1 \bmod{n},
    \end{equation*}
    we find that 
    \begin{equation}\label{transitionrateprobdown}P^{\ell,n}(X_t = \{h_1',\dots,h_\ell'\}| X_0 = \{h_1,\dots,h_\ell\}) = T'\frac{\Delta(\mathbf{h}')}{\Delta(\mathbf{h})}.
    \end{equation}
\end{lemma}
\begin{proof}
    To prove Equation~\eqref{transitionrateprobup}, let us first write $E = \{h_1,\dots,h_\ell\}$ and $E' = \{h_1,\dots,h_{\ell-1},h_\ell'\}$. Consider the event $\{X_t = E', X_0 = E\}$. Any event involving more than 1 bead has probability $o(t)$ and the only possible event involving 1 bead is that $(0,h_1),\dots,(0,h_\ell)$ are occupied and the interval $[0,t)$ on row $h_\ell$ contains an up bead. So by Equation~\ref{eq:light2}, we have that
    \begin{align}\label{finalform}
        P^{\ell,n}(X_t = E', X_0 = E) &= P^{\ell,n}(y_1,\dots,y_\ell \text{ occupied}, \text{up bead in }[h_\ell \mathbf{e}^2, h_\ell \mathbf{e}^2 + t\mathbf{e}^1)) + o(t) \nonumber\\
        &= t\det_{i,j = 1}^{\ell + 1} H_{\ell,n}(\mathbf{f}^i,y_i,y_j) + o(t),
    \end{align}
    where we are using the notation $h_i\mathbf{e}^2 := (0,h_i)$, and where
    \begin{equation*}
        y_i = (0,h_i), \mathbf{f}^i = \mathbf{e}^1 \text{ for all }1 \leq j \leq \ell \text{ and } y_{\ell + 1} = (0,h_\ell), \mathbf{f}^i = \mathbf{e}^2.
    \end{equation*}
    Honing in on the determinant term, we have
    \begin{equation}\label{blockmatrix}
        \det_{i,j = 1}^{\ell + 1} H_{\ell,n}(\mathbf{f}^i,y_i, y_j) = \det \begin{pmatrix}
            K_{i,j} & K_{i,\ell} & K_{i,\ell + 1}\\
            K_{\ell,j} & K_{\ell,\ell} & K_{\ell,\ell + 1}\\
            J_{\ell + 1,j} & J_{\ell + 1,\ell} & J_{\ell + 1,\ell + 1}
        \end{pmatrix},
    \end{equation}
    where $K_{i,j} = H_{\ell,n}(\mathbf{e}^1,y_i, y_j)$, $J_{i,j} = H_{\ell,n}(\mathbf{e}^2,y_i, y_j)$ and Equation \eqref{blockmatrix} depicts an $(\ell - 1 + 1 + 1) \times (\ell - 1 + 1 + 1)$ block matrix. 

    We now leverage that the last two columns are similar. It is not hard to confirm that, as $t \downarrow 0$,
    \begin{align*}
        K_{i,\ell+1} &= K_{i,\ell},\\
        K_{\ell,\ell + 1} &= K_{\ell,\ell} - 1,\\
        J_{\ell + 1,\ell + 1}&= J_{\ell + 1,\ell} + o(1),
    \end{align*}
    and so equation \eqref{blockmatrix} simplifies to
    \begin{align*}
        \det_{i,j = 1}^{\ell + 1} H_{\ell,n}(\mathbf{f}^i,y_i, y_j) &= \det \begin{pmatrix}
            K_{i,j} & K_{i,\ell} & K_{i,\ell}\\
            K_{\ell,j} & K_{\ell,\ell} & K_{\ell,\ell} - 1\\
            J_{\ell + 1,j} & J_{\ell + 1,\ell} & J_{\ell + 1,\ell}
        \end{pmatrix} + o(1)\\
        &= \det \begin{pmatrix}
            K_{i,j} & K_{i,\ell} & 0\\
            K_{\ell,j} & K_{\ell,\ell} &- 1\\
            J_{\ell + 1,j} & J_{\ell + 1,\ell} & 0
        \end{pmatrix} + o(1)\\
        &= \det \begin{pmatrix}
            K_{i,j} & K_{i,\ell}\\
            J_{\ell + 1,j} & J_{\ell + 1,\ell}
        \end{pmatrix} + o(1),
    \end{align*}
    where the last line features a determinant of a $(\ell - 1 + 1) \times (\ell - 1 + 1)$ block matrix and we have used column operations and the fact that the -1 lies in position $(\ell + 1,\ell)$ to see that the determinant does not change sign upon expanding down the $(\ell + 1)$-th column.

    Now examining the matrix we have, we realise we can write it as $Q_{i,j}$ for $1 \leq i,j \leq \ell$, where
    \begin{align*}
        Q_{i,j} &= \begin{cases}
            K_{i,j} = \frac{1}{n} \sum_{z \in \mathcal{L}_{\ell,n}} z^{h_i - h_j} & \text{ for } 1 \leq i \leq \ell -1, 1 \leq j \leq \ell,\\
            J_{\ell + 1, j} = -\frac{T}{n} \sum_{z \in \mathcal{L}_{\ell,n}} z^{1 + h_\ell - h_j} + o(1) & \text{ for } i = \ell, 1 \leq j \leq \ell
        \end{cases}\\
        &= (-T)^{\indicpower{i = \ell}}\frac{1}{n} \sum_{z \in \mathcal{L}_{\ell,n}} z^{h'_i + n\indicpower{h_\ell = n-1, i = \ell} - h_j}  + o(1)\hspace{0.5cm}\text{ for } 1 \leq i,j \leq \ell\\
        &= \left((-1)^{1 + (n-\ell +1) \indicpower{h_\ell = n-1}}T\right)^{\indicpower{i = \ell}}\frac{1}{n} \sum_{z \in \mathcal{L}_{\ell,n}} z^{h'_i - h_j} + o(1) \hspace{0.5cm}\text{ for } 1 \leq i,j \leq \ell,
    \end{align*}
    and so we can further simplify Equation \eqref{blockmatrix} to
    \begin{align*}
        \det_{i,j = 1}^{\ell + 1} H_{\ell,n}(\mathbf{f}^i,y_i, y_j) &= \det_{i,j = 1}^{\ell} Q_{i,j} + o(1)\\
        &= (-1)^{1 + (n - \ell + 1) \indicpower{h_\ell = n-1}}T n^{-\ell} \left(\det_{i,r = 1}^{\ell}A_{i,r}^{\mathbf{h}'}\right)\overline{\left(\det_{r,j = 1}^{\ell}A_{r,j}^{\mathbf{h}}\right)} + o(1).
    \end{align*}
    To simplify this, we refer back to Lemma \ref{lem: vandermonde}, but with care, since there are two cases. 
    
    If $h_\ell \neq n-1$, we see that $\sgn(\sigma_\mathbf{h}) = \sgn(\sigma_\mathbf{h}')$, but that $(-1)^{\sum_{r=1}^{\ell} h'_r} = -(-1)^{\sum_{r=1}^{\ell} h_r}$. So ignoring the complex moduli, $\det_{i,r = 1}^{\ell}A_{i,r}^{\mathbf{h}'}$ is one sign flip different from $\det_{i,r = 1}^{\ell}A_{i,r}^{\mathbf{h}}$.
    
    If $h_\ell = n-1$, then $h_\ell' = 0$, so $\sgn(\sigma_{\mathbf{h}'}) = (-1)^{\ell-1}\sgn(\sigma_{\mathbf{h}})$ and $(-1)^{\sum_{r=1}^{\ell} h'_r}(-1)^{n-1} =(-1)^{\sum_{r=1}^{\ell} h_r}$. So ignoring the complex moduli again, $\det_{i,r = 1}^{\ell}A_{i,r}^{\mathbf{h}'}$ is $n-\ell$ sign flips different from $\det_{i,r = 1}^{\ell}A_{i,r}^{\mathbf{h}}$.

    This means that in either case, we have that 
    \begin{equation*}
        \det_{i,j = 1}^{\ell + 1} H_{\ell,n}(\mathbf{f}^i,y_i, y_j) = Tn^{-\ell} \Delta(\mathbf{h}')\Delta(\mathbf{h}) + o(1).
    \end{equation*}
    Substituting this back into Equation~\eqref{finalform} and using Corollary~\ref{cor: stateprobability} gives the result in Equation~\eqref{transitionrateprobup}.

    The proof for Equation~\eqref{transitionrateprobdown} is argued in the same fashion.
    \end{proof}

\section{\texorpdfstring{The measures $P^{\ell,n}$ as the law of non-colliding walkers in the ring}{The measures Pn,l as the law of non-colliding walkers in the ring} }
\label{sec:equiv}
In this section we work towards a proof of the following result.

\begin{thm} \label{thm:equiv}
The probability measure $P^{\ell,n}$ is precisely the stationary law of $\ell$ independent Poisson walkers on the ring $\mathbb{Z}_n$, that jump from $h$ to $h+1$ at rate $T$, from $h$ to $h-1$ at rate $T'$ and are conditioned to never collide. 
\end{thm}

\subsection{The cyclic Karlin-McGregor formula}

Before proving Theorem \ref{thm:equiv}, we make a brief digression on the Karlin-McGregor formula. Namely, we state and prove a version of the Karlin-McGregor formula \cite{KM} involving Kasteleyn signs that is suitable for processes on this torus. We emphasise that this weighted version of the formula is not new, with the idea appearing in work by Liechty and Wang \cite{LW} and by Fulmek \cite{fulmek}. 

We briefly state the Karlin-McGregor formula in its simplest incarnation. Suppose we have a right-continuous Markov process $X :=(X_t)_{t \geq 0}$ taking values in $\mathbb{Z}$ and with the property that $X$ experiences jumps of size at most one. For $x,y \in \mathbb{Z}$, let $p_t(x,y) := P_x(X_t = y)$ be the transition kernel of the process, where $P_x$ denotes the law of the process starting from $x$. The Karlin-McGregor formula states that if under $P_{x_1,\dots,x_\ell}$, $X_1(t),\dots,X_\ell(t)$ are independent copies of this Markov process starting at distinct spatial locations $x_1< \cdots < x_\ell$ in $\mathbb{Z}$, then for any $y_1 < \cdots < y_\ell$ in $\mathbb{Z}$ we have 
\begin{align}\label{eq:KMspecific}
P_{x_1,\dots,x_\ell} \left( X_1(t)=y_1,\dots,X_\ell(t) = y_\ell, \text{ no collisions} \right) = \det_{i,j =1}^\ell p_t(x_i,y_j), 
\end{align}
where the event $\{ \text{no collisions} \} := \{ \forall s \in [0,t] : X_i(s) \neq X_j(s) , i \neq j \}$ refers to the event that no two copies of the process occupy the same spatial location at the same time. Karlin and McGregor note in their original paper \cite{KM} that more broadly, letting $\mathcal{S}_\ell$ represent the group of permutations on $\ell$ elements, for independent copies $X_1(t),\dots,X_\ell(t)$ of $X(t)$, a Strong Markov process (continuous or otherwise) and $A_1,\dots,A_\ell$ Borel sets, we have 
\begin{equation} \label{eq:KMgen}
\sum_{ \sigma \in \mathcal{S}_\ell} \mathrm{sgn}(\sigma) P_{x_1,\dots,x_\ell}( X_1(t) \in A_{\sigma(1)},\dots,X_\ell(t) \in A_{\sigma(\ell)} , \text{ no collisions})  = \det_{i,j =1}^\ell p_t(x_i,y_j).
\end{equation}
It is simple to see that $P_{x_1,\dots,x_\ell}$ is a Feller process, so is therefore Strongly Markovian. So \eqref{eq:KMspecific} follows from \eqref{eq:KMgen} since no collisions implies that $X_1(t) < \cdots < X_\ell(t)$ for all $t$.
Consider now the setting where $E = [n]$ is a discrete torus and the individual copies of the Markov chain undergo jumps of size one. Then $P_{x_1,\dots,x_\ell}(A_\sigma)$ may only be non-zero when $\sigma$ is a power of the cyclic permutation $(12\cdots \ell)$. Then since $(12\cdots \ell)$ has sign $(-1)^{\ell+1}$, \eqref{eq:KMgen} reads 
\begin{equation} \label{eq:KMcyclic}
\sum_{j=0}^{\ell-1} (-1)^{(\ell+1)j} P_{x_1,\dots,x_\ell}( X_1(t) = y_{1+j},\dots,X_\ell(t) = y_{\ell+j} , \text{ no collisions})  = \det_{i,j =1}^\ell p_t(x_i,y_j),
\end{equation}
where $y_i := y_{i-\ell}$ for $i \geq \ell+1$. 

When $\ell$ is odd, the powers of $(-1)$ on the left-hand side of \eqref{eq:KMgen} are all $1$ and accordingly the left-hand side has the satisfying interpretation of measuring the probability that the $\ell$ independent particles occupy the set $\{y_1,\dots,y_\ell\}$ in any configuration, without having collided. 

When $\ell$ is even however, the left-hand side of \eqref{eq:KMcyclic} has no natural probabilistic interpretation.
It transpires, however, that one can introduce something analogous to a Kasteleyn weighting to the Karlin-McGregor formula to tackle the problem of even $k$. With this in mind, given a stochastic process $(X_t)_{t \geq 0}$ on the torus $[n]$, we define the winding number process to be
\begin{align*}
W(t) := \# \{ \text{Jumps in $[0,t]$ from $n-1$ to $0$} \} -  \# \{ \text{Jumps in $[0,t]$ from $0$ to $n-1$} \}.
\end{align*}
With this in hand, we now give a cyclic Karlin-McGregor formula suitable for even $\ell$.

\begin{thm}[Cyclic Karlin-McGregor formula \cite{fulmek}, \cite{LW}] \label{thm:cyclicKM}
Let $X_1(t),\dots,X_\ell(t)$ be independent copies of a Markov chain on $[n]$ undergoing jumps of size at most one and starting from distinct spatial locations $x_1,\dots,x_\ell$ under $P_{x_1,\dots,x_\ell}$. Define the kernel with Kasteleyn signs
\begin{align*}
p_t^\ell(x,y) := E_x [ \indic{X(t) = y }(-1)^{(\ell+1) W(t) } ],
\end{align*}
where $(W(t))_{t \geq 0}$ is the winding number process associated with $(X(t))_{t \geq 0}$. Then
\begin{equation} \label{eq:KMcyclic2}
P_{x_1,\dots,x_\ell}( \{ X_1(t),\dots,X_\ell(t)\} = \{ y_1,\dots,y_\ell \} , \text{ no collisions})  = \det_{i,j =1}^\ell p^\ell_t(x_i,y_j).
\end{equation}
\end{thm}
{For the proof of this result see \cite{fulmek} or \cite{LW}.}

We now consider applying the cyclic Karlin-McGregor formula to our setting. Let us begin by noting that the transition density $p_t(x,y)$ of a single particle on $[n]$ that jumps from $h$ to $h+1$ at rate $T$ and from $h$ to $h-1$ at rate $T'$ is given by 
\begin{align*}
p_t(x,y) = \sum_{j \geq 0}  \sum_{k \geq 0} \frac{(Tt)^je^{-Tt}}{j!}  \frac{ (T't)^k e^{-T't}}{k!} \indic{ y-x \equiv j-k~ \mathrm{mod} ~n }.
\end{align*}
More generally, the twisted transition kernel is given by 
\begin{align} \label{eq:ptth}
p^{\ell}_t(x,y) = \sum_{j \geq 0}  \sum_{k \geq 0} \frac{(Tt)^je^{-(Tt)}}{j!}  \frac{ (T't)^k e^{-(T't)}}{k!} \indic{ y - x + k - j \equiv 0~ \mathrm{mod} ~n } (-1)^{ (\ell+1)\frac{y-x +k-j}{n} }.
\end{align}
Consider now the identity for $h \in \mathbb{Z}$,
\begin{align} \label{eq:repa}
\frac{1}{n} \sum_{z^n = (-1)^{\theta}} z^{-h} = \indic{h \equiv 0 ~\mathrm{mod}~n } (-1)^{\theta h/n} .
\end{align} 
Using \eqref{eq:repa} in \eqref{eq:ptth} we have 
\begin{align} \label{eq:ptth2}
p^\ell_t(x,y) &= \sum_{j \geq 0}  \sum_{k \geq 0} \frac{(Tt)^je^{-(Tt)}}{j!}  \frac{ (T't)^k e^{-(T't)}}{k!} \frac{1}{n} \sum_{ z^n = (-1)^{\ell+1}} z^{ -y+x -k+j } \nonumber \\
 &= \frac{1}{n} \sum_{ z^n = (-1)^{\ell + 1} } z^{ - (y-x)} \sum_{j \geq 0}  \sum_{k \geq 0} \frac{(Ttz)^je^{-(Tt)}}{j!}  \frac{ (T'tz^{-1})^k e^{-(T't)}}{k!} \nonumber \\
 &= \frac{1}{n} \sum_{ z^n = (-1)^{\ell+1} } z^{ - (y-x)} {e^{ - (-Tz-T'z^{-1}+T+T')t}}.
\end{align}
Setting $z = -w$ in the end, we have
\begin{align} \label{eq:ptth3}
p^\ell_t(x,y) = \frac{(-1)^{y-x}}{n} \sum_{ w^n = (-1)^{n+\ell+1} } w^{ - (y-x)} {e^{ - (Tw+T'w^{-1}+T+T')t}};
\end{align}
c.f. the kernel in Theorem \ref{thm:noncoll}.

Wrapping together our work in this section, we have established the following:

\begin{lemma} \label{lem:nocol}
Under a probability measure $P_{x_1,\dots,x_\ell}$, suppose we have $\ell$ particles starting at distinct spacial locations $x_1,\dots,x_\ell$ in $[n]$ and who independently perform continuous-time random walks where they jump up at rate $T$ and down at rate $T'$. Let  $\tau$ be the first collision time of two of these particles. Then writing $X_i(t)$ for the position of particle $i$ at time $t$ we have 
\begin{align} \label{eq:safetransition}
&P_{x_1,\dots,x_\ell}( \{ X_1(t),\dots,X_\ell(t)\} = \{y_1,\dots,y_\ell\} , \tau > t) = \det_{i,j=1}^\ell \left( p^\ell_t(x_i,y_j) \right),
\end{align}
where $p_t^\ell(x,y)$ is as in \eqref{eq:ptth3}.
\end{lemma}
One can check that shifting $x_i$ and $y_i$ for $i = 1,\dots, \ell$ all up by $1$ does not change the value of $\det_{i,j = 1}^{\ell}\left(p_t^\ell(x,y)\right)$, even when any $x_i$ or $y_i$ loops.
\begin{proof}
Combine \eqref{thm:cyclicKM} with \eqref{eq:ptth2}.
\end{proof}

\subsection{Processes conditioned {never} to collide}

As in Section \ref{sec:noncoll} results, we now define the process associated with Poissonian walkers conditioned to never collide. 
We begin by studying the large-$t$ asymptotics of the non-collision probabilities in Lemma \ref{lem:nocol}. 

Recall that $\mathcal{L}_{\ell,n}$ and $\mathcal{R}_{\ell,n}$ denote the $\ell$ (resp.\ $n-\ell$) elements of the set $\{w \in \mathbb{C} : w^n = (-1)^{n+\ell+1} \}$ with least (resp.\ greatest) real part. It is a brief calculation to verify that 
\begin{align} \label{eq:munl}
\sum_{w \in \mathcal{L}_{\ell,n}} w = \sum_{w \in \mathcal{L}_{\ell,n}} w^{-1} = - \sin(\pi \ell/n)/\sin(\pi/n) \eqqcolon - \mu_{\ell,n}.
\end{align}
(The quickest way to do this is to note that the left-hand side is a negative real number and its modulus is the same as that of the geometric sum $\sum_{j=0}^{\ell-1} e^{2 \pi \iota j /n}$.)

We have the following result

\begin{lemma} \label{lem:nocol2}
Given $x_1 < \cdots < x_\ell$ and $y_1 <\cdots < y_\ell$ in $[n]$ we have 
\begin{align*}
P_{x_1,\dots,x_\ell}( \{ X_1(t),\dots,X_\ell(t)\} = \{y_1,\dots,y_\ell\} , \tau > t) = (1+o(1))n^{-\ell}\Delta(x)\Delta(y)e^{ - (T+T')(\ell - \mu_{\ell,n})t } 
\end{align*}
as $t \to \infty$.
\end{lemma}

\begin{proof}
Pulling the sum in \eqref{eq:ptth3} out of the determinant on the right-hand side of \eqref{eq:safetransition} we obtain
\begin{align} \label{eq:safetransition2}
p^\ell_t(x,y ) &:= P_{x_1,\dots,x_\ell}( X_i(t) = y_i \text{ for } i=1,\dots,\ell, \tau > t) \nonumber \\
 = & n^{-\ell} (-1)^{\sum_{i=1}^\ell (y_i-x_i) } e^{ - \ell(T+T')t }  \sum_{w_1,\dots,w_\ell} e^{ - ( T\sum_{i=1}^\ell w_i +  T' \sum_{i=1}^\ell w_i^{-1}) t  }  \prod_{i=1}^\ell w_i^{x_i} \det_{i,j=1}^\ell \left(  w_i^{ - y_j} \right),
\end{align}
where the outer sum is over all combinations $w_1,\dots,w_\ell$ of $n^{\text{th}}$ roots of $(-1)^{n+\ell+1}$. The determinant $\det_{i,j=1}^\ell(  w_i^{ - y_j})$ is zero if two such $w_i$ are the same. Thus, the sum is concentrated on distinct $w_1,\dots,w_\ell$ minimising the real part of $T\sum_{i=1}^\ell w_i +  T' \sum_{i=1}^\ell w_i^{-1}$. Thus the leading order comes from when $\{w_1,\dots,w_\ell \} = \mathcal{L}_{\ell,n}$, so that with $\mu_{\ell,n}$ as in \eqref{eq:munl} we have
\begin{align} \label{eq:safetransition3}
p^\ell_t(x,y)= (1+o(1)) n^{-\ell} (-1)^{\sum_{i=1}^\ell (y_i-x_i) } e^{ - (\ell-\mu_{\ell,n})(T+T')t }  \sum_{\mathcal{O}(\mathcal{L}_{\ell,n})} \prod_{i=1}^\ell w_i^{x_i} \det_{i,j=1}^\ell \left(  w_i^{ - y_j} \right),
\end{align}
where $\mathcal{O}(\mathcal{L}_{\ell,n})$ denotes the set of all orderings $(w_1,\dots,w_\ell)$ of the elements of $\mathcal{L}_{\ell,n}$. Write $z_j := e^{ \frac{ 2 \pi \iota }{ n} \left( \frac{n-\ell+1}{2} + k + 1 \right) }$, so that $\{z_1,\dots,z_\ell\} = \mathcal{L}_{\ell,n}$. Indexing these orderings using permutations $\sigma$ of $\mathcal{S}_\ell$ to obtain the first equality below and then subsequently using $\det_{i,j=1}^r (A_{\sigma(i),j}) = \mathrm{sgn}(\sigma)\det_{i,j=1}^r(A_{i,j})$  and then packaging the remaining terms into another determinant to obtain the second, we obtain  
\begin{align} \label{eq:safetransition4}
p^\ell_t(x,y) = (1+o(1)) n^{-\ell} (-1)^{\sum_{i=1}^\ell (y_i-x_i) } e^{ - (\ell-\mu_{\ell,n})(T+T')t }  \sum_{ \sigma \in \mathcal{S}_\ell} \prod_{i=1}^\ell z_{\sigma(i)}^{x_i} \det_{i,j=1}^\ell \left(  z_{\sigma(i)}^{ - y_j} \right) \nonumber \\
 = (1+o(1)) n^{-\ell} (-1)^{\sum_{i=1}^\ell (y_i-x_i) } e^{ - (\ell-\mu_{\ell,n})(T+T')t }  \det_{i,j=1}^\ell \left(  z_{i}^{ - y_j} \right) \det_{i,j=1}^\ell \left(  z_{i}^{ x_j} \right).
\end{align}
The result now follows by applying Lemma \ref{lem: vandermonde} to \eqref{eq:safetransition4}.
\end{proof}

Summing over all $0 \leq y_1 < \cdots < y_\ell \leq n-1$ we obtain the following corollary.
\begin{cor} \label{cor:safe}
Given $x_1 < \cdots < x_\ell$ in $[n]$ we have 
\begin{align*}
P_{x_1,\dots,x_\ell}( \tau > t) = (1+o(1))C_{\ell,n}\Delta(x)e^{ - (T+T')(\ell - \mu_{\ell,n})t } 
\end{align*}
as $t \to \infty$. Here
\begin{align*}
C_{\ell,n} := n^{-\ell} \sum_{0 \leq y_1 < \cdots < y_\ell \leq n-1} \Delta(y).
\end{align*}

\end{cor}

 For distinct $x = (x_1,\dots,x_\ell)$ where $0 \leq x_1 < \cdots < x_\ell \leq n-1$, let $(\mathcal{F}_t)_{t\geq 0}$ be the natural filtration to the process $(X_1(t),\dots,X_\ell(t))$ with law $P_x$ and $\mathcal{F}_{\infty} \coloneqq \sigma\left(\bigcup_{t \geq 0} \mathcal{F}_{t}\right)$. Then define a new probability measure $\mathbf{P}^{\ell,n}_x$ governing a Markov process $(X_t)_{t \geq 0}$ by setting
\begin{align*}
\mathbf{P}^{\ell,n}_x(A) := \lim_{s \uparrow \infty} P_x(A | \tau > s).
\end{align*}
 One can check that this limit is valid for events $A \in \bigcup_{t \geq 0} \mathcal{F}_t$ and so by Carath\'eodory's extension theorem, indeed uniquely defines a valid probability measure on $\mathcal{F}_\infty$. Furthermore, one can verify that it inherits the Markov property from $P_x$. We call $\mathbf{P}^{\ell,n}_x$ the law of $\ell$ walkers conditioned to never collide. Our next result provides an explicit expression for the transition probabilities of this Markov process.

\begin{lemma}
We have
\begin{align} \label{eq:Qtran}
\mathbf{P}^{\ell,n}_x(X_t = y) = \frac{\Delta(y)}{\Delta(x)} e^{ (\ell-\mu_{\ell,n})(T+T')t }\det_{i,j=1}^\ell p^\ell_t(x_i,y_j), 
\end{align}
where $p^\ell_t(x,y)$ is given in \eqref{eq:ptth3}.
\end{lemma}

\begin{proof}
A brief calculation using the definition of conditional expectation and the Markov property tells us that
\begin{align*}
\mathbf{P}^{\ell,n}_x(X_t = y) = P_x(X_t=y, \text{ no collision}) \lim_{s \uparrow \infty} \frac{P_y( \tau > s - t) }{ P_x( \tau > s) }.
\end{align*}
Now use Corollary \ref{cor:safe}.
\end{proof}

We are now equipped to prove Theorem \ref{thm:equiv}.

\begin{proof}[Proof of Theorem \ref{thm:equiv}]
Since $P^{\ell,n}$ and $\mathbf{P}^{\ell,n}_x$ govern Markov processes, it is sufficient to establish they have the same transition probabilities. In other words, in light of \eqref{transitionrateprobup}, \eqref{transitionrateprobdown} and \eqref{eq:Qtran} it is sufficient to establish that if $y= (y_1,\dots,y_\ell)$ is obtained from $x = (x_1,\dots,x_\ell)$ by setting $y_{i_0} = x_{i_0} + 1$ and $y_i = x_i$ for $i \neq i_0$, then
\begin{align} \label{eq:lower}
\lim_{t \downarrow 0}  \frac{1}{t} \frac{\Delta(y)}{\Delta(x)} e^{ (\ell-\mu_{\ell,n})(T+T')t }\det_{i,j=1}^\ell p^\ell_t(x_i,y_j) = T \frac{\Delta(y)}{\Delta(x)},
\end{align}
and that the analogous equation holds with $T'$ in place of $T$ on the RHS if we instead have $y_{i_0} = x_{i_0} - 1$. 

We now establish that \eqref{eq:lower} indeed holds. For the sake of simplicity we rule out the case where $i_0 = 1, x_1 = 0$ and $x_1$ jumps down, or when $i_0 = \ell, x_\ell = n-1$ and $x_\ell$ jumps up. Regardless, these cases follow by the toroidal invariance of $\det_{i,j=1}^\ell p^\ell_t(x_i,y_j)$. 

By the definition of $p_t^\ell(x,y)$ given in \eqref{eq:ptth}, the small-$t$ asymptotics of this transition kernel are given by 
\begin{align*}
p_t^\ell(x,y) = (1 - (T+T')t) \indic{ x= y} + Tt \indic{ y = x+1 } + T't \indic{ y = x-1} + o(t).
\end{align*}
It is now easily seen that with $x_i$ and $y_i$ as defined at beginning of the proof, we have
\begin{align} \label{eq:smalldet}
\lim_{t \downarrow 0}\frac{1}{t} \det_{i,j=1}^\ell p^\ell_t(x_i,y_j) = T.
\end{align}
Using \eqref{eq:smalldet}, we establish \eqref{eq:lower}.
\end{proof}

\begin{proof}[Proof of Theorem~\ref{thm:noncoll}]
    Proposition~\ref{prop:longer} (or more accurately, Equation~\eqref{eq:light2}) provides the formula and Theorem~\ref{thm:equiv} gives the equivalence between $P^{\ell,n}$ and $\mathbf{P}^{\ell,n}$ with the stationary initial distribution.
\end{proof}

\section{The traffic representation for ASEP}
\label{sec:traffic}

In this section we prove Theorem \ref{thm:traffic}.

    \subsection{Relationship to ASEP}
    We have seen that we may understand $P^{\ell, n} = \mathbf{P}^{\ell,n}$ as the law of a continuous Markov process $(X_t)_{t \in \mathbb{R}}$ in $\binom{[n]}{\ell}$ (set of subsets of $[n]$ of size $\ell$) with transition rates given by Lemma~\ref{lem: transitionrates}. To be more explicit, we can associate the set of functions $\left\{f\colon \binom{[n]}{\ell} \to \mathbb{R} \right\}$ with $\{f \colon [n]^{(\ell)} \to  \mathbb{R}, f \text{ symmetric}\}$ ($f$ is symmetric if for any permutation $\sigma$, $f(h_1,\dots,h_\ell) = f(h_{\sigma(1)},\dots, h_{\sigma(\ell)})$) and we can write the infinitesimal generator as the linear operator given by
    \begin{align*}
        \mathcal{G}^{\ell,n}f(\mathbf{h}) = &T\sum_{j: h_j + 1 \neq h_k \forall k} \frac{\Delta(\mathbf{h} + \mathbf{e}_j)}{\Delta(\mathbf{h})} (f(\mathbf{h} + \mathbf{e}_j) - f(\mathbf{h}))\\
        + &T'\sum_{j: h_j - 1 \neq h_k \forall k} \frac{\Delta(\mathbf{h} - \mathbf{e}_j)}{\Delta(\mathbf{h})} (f(\mathbf{h} - \mathbf{e}_j) - f(\mathbf{h})) 
    \end{align*}
    for any such symmetric function $f$. We will show its relation to ASEP with certain parameters, the continuous Markov process with infinitesimal generator given by the linear operator
    \begin{equation}\label{generatorASEP}
        \mathcal{G}^{\text{ASEP},\ell,n}f(\mathbf{h}) =  T\sum_{j: h_j + 1 \neq h_k \forall k} \left[f(\mathbf{h} + \mathbf{e}_j) - f(\mathbf{h})\right] \;+\; T'\sum_{j: h_j - 1 \neq h_k \forall k} \left[f(\mathbf{h} - \mathbf{e}_j) - f(\mathbf{h})\right].
    \end{equation}
    with corresponding laws $\mathbf{P}^{\text{ASEP},\ell,n}_x$ for initial states $x \in \binom{[n]}{\ell}$. Informally, this is the simple exclusion process with rate $T$ for up moves and $T'$ for down moves.

    For later, we will need some results and some definitions. We begin by defining the \textbf{traffic} of a configuration $\mathbf{h} \in [n]^{(\ell)}$, as
    \begin{equation*}
        \operatorname{Traffic}(\mathbf{h}) \coloneqq \#\{(j,k) : 1 \leq j, k \leq \ell, h_k = h_j + 1\},
    \end{equation*}
    which we note is equal to $\#\{(j,k) : 1 \leq j, k \leq \ell, h_k = h_j - 1\}$ and is a symmetric function, so can be analogously defined for $E \in \binom{[n]}{\ell}$.

    A sequence of results now. For a general continuous Markov process $(X_t)_{t \geq 0}$ on finite state space $V$ with infinitesimal generator $\mathcal{G}$ and distributions $\{\mathbf{P}_x: x \in V\}$, suppose we have a positive function $\Delta\colon V \to (0,\infty)$ and another function $p\colon V \to \mathbb{R}$ such that
    \begin{equation*}
        \mathcal{G}\Delta(x) = p(x) \Delta(x).
    \end{equation*}
    Then the stochastic process 
    \begin{equation*}
        M_t \coloneqq \frac{\Delta(X_t)}{\Delta(X_0)}\exp\left(-\int_{0}^{t} p(X_s) \; ds\right)
    \end{equation*}
    is a positive, unit-mean $\mathbf{P}_x$-martingale for each $x\in V$. Thus, we can define a new set of probability distributions $\{\mathbf{P}_x^{\Delta}: x \in V\}$ defined using the Radon-Nikodym martingale change of measure given by
    \begin{equation*}
        \left.\frac{d\mathbf{P}_x^{\Delta}}{d\mathbf{P}_x}\right|_{\mathcal{F}_t} = M_t,
    \end{equation*}
    where $(\mathcal{F}_t)_{t\geq 0}$ is the natural filtration. The measures $\mathbf{P}_x^{\Delta}$ and $\mathbf{P}_x$ are absolutely continuous with respect to each other. So it also follows that, for $x \neq y$
    \begin{equation}
        \lim_{t\downarrow 0} \frac{1}{t} \mathbf{P}_x^{\Delta}(X_t = y) = \lim_{t \downarrow 0} \frac{\Delta(y)}{\Delta(x)} \frac{1}{t} \mathbf{P}_x(X_t = y).
    \end{equation}
    This will be the tool that allows us to derive the change of measure to ASEP.

    Now, a computational lemma.
    \begin{lemma}\label{lem: computationallemma}
        For $\mathbf{h} \in [n]^{(\ell)}$, define $\phi(\mathbf{h}) = \det_{j,k = 1}^{\ell} A_{j,k}^{\mathbf{h}}$. Then 
        \begin{equation*}
            \sum_{j = 1}^\ell (-1)^{(n-\ell+1)\indicpower{h_j = n-1}}\phi(\mathbf{h} + \mathbf{e}_j) = -\mu_{\ell,n}\phi(\mathbf{h}),
        \end{equation*}
        and
        \begin{equation*}
            \sum_{j = 1}^\ell (-1)^{(n-\ell+1)\indicpower{h_j = 0}}\phi(\mathbf{h} - \mathbf{e}_j) = -\mu_{\ell,n}\phi(\mathbf{h}).
        \end{equation*}
        where $\mu_{\ell,n} = \frac{\sin(\pi \ell / n)}{\sin(\pi / n)}$.
    \end{lemma}
    \begin{proof}
        By direct computation,
        \begin{align*}
        \sum_{r = 1}^\ell (-1)^{(n-\ell+1)\indicpower{h_r = n-1}} \phi(\mathbf{h} + \mathbf{e}_r) &=  \sum_{r = 1}^\ell (-1)^{(n-\ell+1)\indicpower{h_r = n-1}} \det_{j,k = 1}^\ell \left(z_{k}^{h_j + \indicpower{j = r}(1 - n\indicpower{h_j = n-1})}\right) \\
        &= \sum_{r = 1}^\ell \det_{j,k = 1}^\ell \left(z_{k}^{h_j + \indicpower{j=r}}\right)\\
        &= \sum_{r = 1}^\ell \sum_{\sigma} \sgn(\sigma)\prod_{j = 1}^{\ell} z_{\sigma(j)}^{h_j + \indicpower{j = r}}\\
        &= \sum_{\sigma} \sgn(\sigma) \left(\sum_{r = 1}^\ell z_{\sigma(r)}\right)\prod_{j = 1}^{\ell} z_{\sigma(j)}^{h_j}\\
        &= \left(\sum_{z \in \mathcal{L}_{\ell,n}} z\right) \phi(\mathbf{h}),
        \end{align*}
        which gives the required result by noting again that the sum is geometric. The other result is true by the same method. 
    \end{proof}
    The main result can now be shown.
    
    \begin{proof}[Proof of Theorem~\ref{thm:traffic}]
        Set again $E = \{h_1,\dots,h_\ell\}$. Recalling Equation~\eqref{generatorASEP}, we have that
        \begin{align}\label{generatorequation}
        \mathcal{G}^{\text{ASEP},\ell,n}\Delta(\mathbf{h}) &= T\sum_{j: h_j + 1 \neq h_k \forall k} \left[\Delta(\mathbf{h} + \mathbf{e}_j) - \Delta(\mathbf{h})\right] \;+\; T'\sum_{j: h_j - 1 \neq h_k \forall k} \left[\Delta(\mathbf{h} - \mathbf{e}_j) - \Delta(\mathbf{h})\right] \nonumber\\
        &= \sum_{j = 1}^{\ell} \left[T \Delta(\mathbf{h} + \mathbf{e}_j) + T' \Delta(\mathbf{h} - \mathbf{e}_j)\right] \nonumber\\
        &- T(\ell - \operatorname{Traffic}(\mathbf{h}))\Delta(\mathbf{h}) - T'(\ell - \operatorname{Traffic}(\mathbf{h}))\Delta(\mathbf{h}).
        \end{align}
        Now examine the sum in the RHS of the last expression. Using Lemma~\ref{lem: vandermonde} and recalling the definition of $\phi$,
        \begin{equation*}
            \sum_{j=1}^\ell \Delta(\mathbf{h} + \mathbf{e}_j) = \sum_{j=1}^\ell \frac{\phi(\mathbf{h} + \mathbf{e}_j)}{\sgn(\sigma_{\mathbf{h} + \mathbf{e}_j})(-1)^{\sum_{k = 1}^{\ell}(h_k + \indicpower{j = k} - n\indicpower{j=k, h_j = n-1})}\iota^{\ell(\ell-1)/2}}.
        \end{equation*}
        Now when $h_j \neq n-1$, $\sgn(\sigma_{\mathbf{h} + \mathbf{e}_j}) = \sgn(\sigma_{\mathbf{h}})$ and $(-1)^{\sum_{k = 1}^{\ell}(h_k + \indicpower{j = k}- n\indicpower{j=k, h_j = n-1}))} = -(-1)^{\sum_{k = 1}^{\ell}h_k}$. Or instead, when $h_j = n-1$, $\sgn(\sigma_{\mathbf{h} + \mathbf{e}_j}) = (-1)^{\ell - 1}\sgn(\sigma_{\mathbf{h}})$ and \\$(-1)^{\sum_{k = 1}^{\ell}(h_k + \indicpower{j = k}- n\indicpower{j=k, h_j = n-1}))} = (-1)^{n-1}(-1)^{\sum_{k = 1}^{\ell}h_k}$. So we can use Lemma~\ref{lem: computationallemma} and write
        \begin{align*}
            \sum_{j=1}^\ell \Delta(\mathbf{h} + \mathbf{e}_j) &= -\sgn(\sigma_{\mathbf{h}})(-1)^{\sum_{k = 1}^{\ell}h_k}(-\iota)^{\ell(\ell-1)/2}\sum_{j=1}^\ell(-1)^{(n-\ell+1)\indicpower{h_j = n-1}}\phi(\mathbf{h} + \mathbf{e}_j)\\
            &= \sgn(\sigma_{\mathbf{h}})(-1)^{\sum_{k = 1}^{\ell}h_k}(-\iota)^{\ell(\ell-1)/2} \mu_{\ell,n} \phi(\mathbf{h})\\
            &= \mu_{\ell,n} \Delta(\mathbf{h}).
        \end{align*}
        Similarly, we can deduce that
        \begin{align*}
            \sum_{j=1}^\ell \Delta(\mathbf{h} - \mathbf{e}_j) 
            &= \sum_{j=1}^\ell \frac{\phi(\mathbf{h} - \mathbf{e}_j)}{\sgn(\sigma_{\mathbf{h} - \mathbf{e}_j})(-1)^{\sum_{k = 1}^{\ell}(h_k - \indicpower{j = k} + n\indicpower{j=k, h_j = 0}))}\iota^{\ell(\ell-1)/2}}\\
            &= -\sgn(\sigma_{\mathbf{h}})(-1)^{\sum_{k = 1}^{\ell}h_k}(-\iota)^{\ell(\ell-1)/2} \sum_{j=1}^\ell(-1)^{(n-\ell+1)\indicpower{h_j = 0}}\phi(\mathbf{h} - \mathbf{e}_j)\\
            &= \sgn(\sigma_{\mathbf{h}})(-1)^{\sum_{k = 1}^{\ell}h_k}(-\iota)^{\ell(\ell-1)/2} \mu_{\ell,n} \phi(\mathbf{h})\\
            &= \mu_{\ell,n} \Delta(\mathbf{h}),
        \end{align*}
        and so, returning to Equation~\eqref{generatorequation}, 
        \begin{align*}
            \mathcal{G}^{\text{ASEP},\ell,n}\Delta(\mathbf{h}) &= (T+T')[\mu_{\ell,n} - \ell + \operatorname{Traffic}(\mathbf{h})]\Delta(\mathbf{h}).
        \end{align*}
        So by our discussion preceding Lemma~\ref{lem: computationallemma}, \[M_t = \frac{\Delta(X_t)}{\Delta(X_0)} \exp\left(- (T+T')\int_{0}^t \mu_{\ell,n} - \ell + \operatorname{Traffic}(X_s) \;ds \right)\] is a unit-mean $\mathbf{P}^{\text{ASEP},\ell,n}$-martingale, which defines a martingale change of measure 
        \begin{equation*}
            \left.\frac{d\mathbf{P}^{\ell,n}}{d\mathbf{P}^{\text{ASEP},\ell,n}}\right|_{\mathcal{F}_t} = M_t
        \end{equation*}
        and tells us that
        \begin{equation*}
            \lim_{t \downarrow 0} \frac{1}{t} \mathbf{P}^{\ell,n}(X_t = E'|X_0 = E) = \frac{\Delta(E')}{\Delta(E)} \lim_{t \downarrow 0} \frac{1}{t} \mathbf{P}^{\text{ASEP},\ell,n}(X_t = E'|X_0 = E),
        \end{equation*}
        which coincides with our definition of our infinitesimal generator for the URD Continuous Snake Process. Reciprocating the change of measure completes the proof.
    \end{proof}




\end{document}